\documentclass[10pt]{amsart}
\usepackage{amssymb}

\newtheorem{thm}{Theorem}[section]
\newtheorem{lemma}[thm]{Lemma}
\newtheorem{proposition}[thm]{Proposition}

\newtheorem{definition}[thm]{Definition}
\newtheorem{corollary}[thm]{Corollary}

\newtheorem{notation}[thm]{Notation}

\newtheorem{inductivehyp}[thm]{Inductive Hypothesis}

\newcommand{\p}{\mathbb{P}}
\newcommand{\q}{\mathbb{Q}}

\newcommand{\ot}{\mathrm{ot}}
\newcommand{\cf}{\mathrm{cf}}
\newcommand{\cof}{\mathrm{cof}}
\newcommand{\dom}{\mathrm{dom}}
\newcommand{\ran}{\mathrm{ran}}

\newcommand{\h}{\mathrm{ht}}
\newcommand{\coll}{\mathrm{Col}}

\begin{document}

\title{Club isomorphisms on higher Aronszajn trees}

\author{John Krueger}

\address{John Krueger \\ Department of Mathematics \\ 
University of North Texas \\
1155 Union Circle \#311430 \\
Denton, TX 76203}
\email{jkrueger@unt.edu}

\date{July 2017; revised May 2018}



\thanks{This material is based upon work supported by the National Science Foundation under Grant
No. DMS-1464859.}

\begin{abstract}
We prove the consistency, assuming an ineffable cardinal, of the statement that 
\textsf{CH} holds and any two 
normal countably closed $\omega_2$-Aronszajn trees are club isomorphic. 
This work generalizes to higher cardinals 
the property of Abraham-Shelah \cite{AS} 
that any two normal $\omega_1$-Aronszajn trees are club isomorphic, which 
follows from \textsf{PFA}. 
The statement that any two normal countably closed $\omega_2$-Aronszajn trees 
are club isomorphic implies that there are no $\omega_2$-Suslin trees, so our proof 
also expands on the method of Laver-Shelah \cite{LS} for obtaining the 
$\omega_2$-Suslin hypothesis. 
\end{abstract}

\maketitle

One of the earliest applications of iterated forcing was the theorem of 
Solovay and Tennenbaum \cite{ST} on the consistency of Suslin's hypothesis, which states 
that there does not exist an $\omega_1$-Suslin tree. 
Later Abraham and Shelah \cite{AS} formulated a strengthening of Suslin's hypothesis, namely, the 
statement that any two normal $\omega_1$-Aronszajn trees are club isomorphic, and proved 
its consistency. 
The former result on Suslin's hypothesis was proved using the method of finite support 
iterations of $\omega_1$-c.c.\ forcings, and the latter was established using Shelah's technique 
of countable support iterations of proper forcing. 
Suslin's hypothesis also follows from $\textsf{MA}_{\omega_1}$, whereas the Abraham-Shelah property 
follows from \textsf{PFA} but not from $\textsf{MA}_{\omega_1}$. 
Using a much more sophisticated argument than appeared in the original consistency proof on 
Suslin's hypothesis, Laver and Shelah \cite{LS} proved the consistency, assuming a weakly compact 
cardinal, of \textsf{CH} together with the $\omega_2$-Suslin hypothesis, which asserts the 
nonexistence of an $\omega_2$-Suslin tree.  

In this paper, we generalize the Abraham-Shelah result on club isomorphisms of 
normal $\omega_1$-Aronszajn trees to normal countably closed $\omega_2$-Aronszajn trees, 
showing that it is consistent with \textsf{CH}, assuming an ineffable cardinal, 
that any two normal countably closed 
$\omega_2$-Aronszajn trees are club isomorphic. 
We note that it is not consistent with \textsf{CH} that any two normal $\omega_2$-Aronszajn 
are club isomorphic, since club isomorphisms preserve the closure of levels of the tree, 
and under \textsf{CH} there are normal $\omega_2$-Aronszajn trees which are countably closed 
and ones which are not closed at any level. 
Nevertheless, our generalization of Abraham-Shelah's property to $\omega_2$ 
does indeed imply the nonexistence of an $\omega_2$-Suslin tree, and thus provides 
a natural strengthening of the $\omega_2$-Suslin hypothesis which is analogous to the 
situation on $\omega_1$.

The Laver-Shelah consistency proof involved a countable support iteration of specializing 
forcings, and it employed a weakly compact cardinal $\kappa$ which is collapsed to 
$\omega_2$ in order to obtain the $\kappa$-c.c.\ of the iteration. 
Similarly, we will use a countable support forcing iteration at an ineffable cardinal of forcings 
which add club isomorphisms between pairs of normal countably closed $\omega_2$-Aronszajn trees. 
The main complication, besides having to iterate a more complex forcing poset, is 
replacing the $\kappa$-c.c.\ property of the iteration, which appeared 
in the Laver-Shelah proof, with $\kappa$-properness. 
This complication is handled by forcing over a stronger large cardinal, namely an ineffable 
cardinal, as well as employing in the iteration a rudimentary kind of side condition. 
We propose that the ideas and methods used 
in our proof could be useful for generalizing other consequences of 
\textsf{PFA} from $\omega_1$ to higher cardinals.

\bigskip

The first three sections are mostly preliminary, handling strong genericity, the ineffability ideal, 
and trees respectively. 
In Section 4 we present a generalization of the Abraham-Shelah forcing for adding 
club isomorphisms between normal $\omega_1$-Aronszajn trees to countably closed 
normal $\omega_2$-Aronszajn trees. 
In Section 5 we develop a kind of countable support forcing iteration of the forcing from Section 4. 
The main difficulty is proving that the forcing iteration preserves $\omega_2$, and this is handled in 
Sections 6 and 7. We note that Sections 6 and 7 contain as a proper subset adaptations of all of the arguments 
from the Laver-Shelah forcing construction \cite{LS}.

I would like to thank Assaf Rinot for helpful discussions on the topics in this paper 
and the referee for many useful comments.

\section{Strong genericity}

In this section, we will develop the material on strongly generic 
conditions which we will need in the results from Sections 6 and 7, which form the heart of the 
consistency proof. 
Much of this material was originally developed by Mitchell \cite{mitchell}.

Fix a forcing poset $\p$ for the remainder of the section.

\begin{definition}
For any set $N$, 
a condition $p \in \p$ is said to be 
\emph{strongly $(N,\p)$-generic} if for any $D$ which is a dense 
subset of the forcing poset $N \cap \p$, $D$ is predense below $p$ in $\p$.
\end{definition}

Note that if $\p \in N \prec H(\theta)$ for some infinite cardinal $\theta$, 
if $p$ is strongly $(N,\p)$-generic, then $p$ is 
$(N,\p)$-generic in the sense that for any set $D \in N$ which is dense in $\p$, 
$D \cap N$ is predense below $p$.

\begin{definition}
Let $N$ be a set, $q \in \p$, and $s \in N \cap \p$. 
We will write $*_N^\p(q,s)$ to mean that for all $t \le s$ in $N \cap \p$, 
$q$ and $t$ are compatible in $\p$.
\end{definition}

Note that if $*_N^\p(q,s)$ holds, then for all $s_1 \le s$ in $N \cap \p$, 
$*_N^\p(q,s_1)$ also holds. 
And if $*_N^\p(q,s)$ holds and $q \le q_0$, then $*_N^\p(q_0,s)$ holds.

\begin{lemma}
Let $N$ be a set and $p \in \p$. 
Then $p$ is strongly $(N,\p)$-generic iff for all $q \le p$ there is $s \in N \cap \p$ 
such that $*_N^\p(q,s)$ holds.
\end{lemma}

\begin{proof}
See \cite[Lemma 2.2]{jk26}.
\end{proof}

\begin{notation}
A set $N$ is said to be \emph{suitable for $\p$} if for some regular uncountable 
cardinal $\theta$ with $\p \in H(\theta)$, $\p \in N \prec H(\theta)$.
\end{notation}

\begin{notation}
Let $N$ be suitable for $\p$ and $p$ a strongly $(N,\p)$-generic condition. 
Let $H$ be a generic filter on $N \cap \p$. 
In $V[H]$, define a forcing poset $(\p / p) / H$ with underlying set 
the conditions $q \le p$ such that $q$ is compatible in $\p$ with every condition in $H$, and 
ordered the same as $\p$. 
\end{notation}

\begin{lemma}
Let $N$ be suitable for $\p$ and $p$ a strongly $(N,\p)$-generic condition. 
Then for all $q \le p$ and $s \in N \cap \p$, 
$*_N^\p(q,s)$ holds iff 
$s \Vdash_{N \cap \p} q \in (\p / p) / \dot G_{N \cap \p}$.
\end{lemma}

\begin{proof}
Assume that $*_N^\p(q,s)$ holds. 
Suppose for a contradiction that there is $t \le s$ in $N \cap \p$ which forces that 
$q$ is not in $(\p / p) / \dot G_{N \cap \p}$. 
Then we can find $u_1 \le t$ in $N \cap \p$ 
and $v \in N \cap \p$ which is incompatible with $q$ in $\p$ 
such that $u_1 \Vdash_{N \cap \p} v \in \dot G_{N \cap \p}$. 
As $u_1$ and $v$ are compatible in $N \cap \p$, fix $u_2 \le u_1, v$ 
in $N \cap \p$. 
Since $u_2 \le v$, $u_2$ and $q$ are incompatible in $\p$. 
But $u_2 \le u_1 \le t \le s$, which contradicts $*_N^\p(q,s)$.

Now assume that $*_N^\p(q,s)$ fails. 
Then there is $t \le s$ in $N \cap \p$ which is incompatible with $q$ in $\p$. 
As $t \Vdash_{N \cap \p} t \in \dot G_{N \cap \p}$, we have that 
$t \Vdash_{N \cap \p} q \notin (\p / p) / \dot G_{N \cap \p}$. 
Hence, $s$ does not force in $N \cap \p$ that 
$q \in (\p / p) / \dot G_{N \cap \p}$.
\end{proof}

\begin{lemma}
Let $N$ be suitable for $\p$ and $p$ a strongly $(N,\p)$-generic condition. 
\begin{enumerate}
\item If $q \le p$, $s \in N \cap \p$, and $q$ and $s$ are compatible, then 
there is $t \le s$ in $N \cap \p$ such that $*_N^\p(q,t)$;
\item if $D \subseteq \p$ is dense below $p$, then $N \cap \p$ forces that 
$D \cap (\p / p) / \dot G_{N \cap \p}$ is dense in $(\p / p) / \dot G_{N \cap \p}$;
\item if $D \subseteq \p$ is dense below $p$, 
then $N \cap \p$ forces that whenever $q \in (\p / p) / \dot G_{N \cap \p}$ 
and $s \in \dot G_{N \cap \p}$, then there is 
$r \in D \cap (\p / p) / \dot G_{N \cap \p}$ 
such that $r \le q, s$.
\end{enumerate}
\end{lemma}

\begin{proof}
(1) Fix $r \le q, s$ in $\p$. 
Then $r \le p$. 
Since $p$ is strongly $(N,\p)$-generic, 
there is $t \in N \cap \p$ such that $*_N^\p(r,t)$ holds. 
Moreover, by extending $t$ in $N \cap \p$ if necessary, we may assume 
that $t$ is either below $s$ or is incompatible with $s$. 
Since $*_N^\p(r,t)$ holds and $r \le q$, it follows that $*_N^\p(q,t)$ holds. 
We claim that $t \le s$, which completes the proof. 
If not, then $t$ is incompatible with $s$. 
But $t$ and $r$ are compatible and $r \le s$, so 
$t$ and $s$ are compatible.

(2) Suppose for a contradiction that $s \in N \cap \p$, $q \in \p$, 
and $s$ forces in $N \cap \p$ that $q	 \in (\p / p) / \dot G_{N \cap \p}$ 
but for all $r \le q$ in $D$, $r$ is not in $(\p / p) / \dot G_{N \cap \p}$. 
By Lemma 1.6, $*_N^\p(q,s)$ holds. 
In particular, $q$ and $s$ are compatible, so fix $r \le q, s$. 
Since $D$ is dense below $p$, by extending $r$ if necessary we may assume that $r \in D$. 
Since $r \le s$, in particular, $r$ and $s$ are compatible. 
So by (1), we can fix $t \le s$ in $N \cap \p$ such that $*_N^\p(r,t)$ holds. 
Then $t$ forces in $N \cap \p$ that $r \le q$ is in $D \cap (\p / p) / \dot G_{N \cap \p}$, 
which contradicts that $t \le s$.

(3) Suppose that $u \in N \cap \p$ forces that 
$q \in (\p / p) / \dot G_{N \cap \p}$ and $s \in \dot G_{N \cap \p}$. 
We will find $r \in D$ with $r \le q, s$ and $z \le u$ in $N \cap \p$ such that 
$z \Vdash_{N \cap \p} r \in (\p / p) / \dot G_{N \cap \p}$. 
Clearly $u$ and $s$ are compatible in $N \cap \p$, so fix $u_1 \le u, s$ in $N \cap \p$. 
Since $*_N^\p(q,u)$ holds by Lemma 1.6, also $*_N^\p(q,u_1)$ holds, so $u_1$ and $q$ are compatible. 
Fix $q_1 \le u_1, q$ in $\p$. 
Then $q_1$ and $u_1$ are obviously compatible, so by (1) fix 
$t \le u_1$ in $N \cap \p$ such that $*_N^\p(q_1,t)$ holds. 
By Lemma 1.6, $t$ forces in $N \cap \p$ that $q_1 \in (\p / p) / \dot G_{N \cap \p}$. 

Since $D$ is forced to be dense in $(\p / p) / \dot G_{N \cap \p}$ by (2), we can find 
$z \le t$ in $N \cap \p$ and $r \le q_1$ in $D$ such that 
$z$ forces in $N \cap \p$ that $r \in (\p / p) / \dot G_{N \cap \p}$. 
Then $r \le q_1 \le q$, $r \le q_1 \le u_1 \le s$, and $z \le t \le u_1 \le u$. 
So $r \in D$, $r \le q, s$, $z \le u$, and 
$z \Vdash_{N \cap \p} r \in (\p / p) / \dot G_{N \cap \p}$.
\end{proof}

\begin{proposition}
Let $N$ be suitable for $\p$ and $p$ a strongly $(N,\p)$-generic condition. 
Suppose that $G$ is a generic filter on $\p$ such that $p \in G$.
Let $H := N \cap G$. Then:
\begin{enumerate}
\item $H$ is a $V$-generic filter on $N \cap \p$;
\item $G' := G \cap (\p / p) / H$ is a $V[H]$-generic filter on $(\p / p) / H$;
\item $V[G] = V[H][G']$.
\end{enumerate}
\end{proposition}

\begin{proof}
(1) If $D$ is a dense subset of $N \cap \p$, then $D$ is predense below $p$. 
Since $p \in G$, we can fix $z \in D \cap G$. 
But then $z \in G \cap N = H$. 
So $D \cap H \ne \emptyset$.

To see that $H$ is a filter, note that $H$ is obviously closed upwards in $N \cap \p$. 
Suppose that $u$ and $v$ are in $H$. 
Then $u$ and $v$ are in $N$. 
Let $E$ be the set of $w \in N \cap \p$ which are either incompatible in $\p$ 
with at least one of $u$ or $v$, or below them both. 
By the elementarity of $N$, it is easy to check that $E$ is 
a dense subset of $N \cap \p$. 
Using the previous paragraph, fix $w \in E \cap H$. 
Since $G$ is a filter and $u$ and $v$ are in $G$, $w$ is compatible with both 
$u$ and $v$, and hence by the definition of $E$, $w \le u, v$.

(2) To see that $G'$ is a filter, note that 
it is clearly upwards closed in $(\p / p) / H$. 
Let $u$ and $v$ be in $G'$. 
Since $G$ is a filter, there is $w \le u, v$ in $G$. 
So $w \le p$. 
As $H \subseteq G$, $w$ is compatible in $\p$ with every member of $H$. 
So $w \in G \cap (\p / p) / H = G'$.

Let $D$ be a dense open subset of $(\p / p) / H$ in $V[H]$, and we will show 
that $D \cap G' \ne \emptyset$. 
Fix an $(N \cap \p)$-name $\dot D$ for a dense open subset of 
$(\p / p) / \dot G_{N \cap \p}$ such that $D = \dot D^H$.

Let $E$ be the set of $s \le p$ such that for some $u \in N \cap \p$ and $r \in \p$, 
$s \le u, r$, and $u$ forces in $N \cap \p$ that $r \in \dot D$. 
We claim that $E$ is dense below $p$. 
So let $q \le p$, and we will find $s \le q$ in $E$.

Since $p$ is strongly $(N,\p)$-generic, fix $t \in N \cap \p$ such that 
$*_N^\p(q,t)$ holds. 
Then $t \Vdash_{N \cap \p} q \in (\p / p) / \dot H$. 
Since $\dot D$ is forced to be dense, fix $u \le t$ in $N \cap \p$ and 
$r \le q$ such that $u \Vdash_{N \cap \p} r \in \dot D$. 
Then $*_N^\p(r,u)$ holds. 
In particular, $r$ and $u$ are compatible. 
Fix $s \le r, u$. 
Then $s \le q$ and $s \in E$, as witnessed by $u$ and $r$.

Since $E$ is dense below $p$, fix $s \in E \cap G$. 
Fix $u \in N \cap \p$ and $r$ such that $s \le u, r$, and 
$u \Vdash_{N \cap \p} r \in \dot D$. 
Then $s \le p$ and $s \in G$ easily implies that $s \in (\p / p) / H$. 
So $s \in G'$. 
Since $s \le u$ and $s \in G$, $u \in N \cap G = H$. 
Hence, $r \in \dot D^H = D$. 
As $s \le r$ and $s \in G'$, $r \in G'$. 
Thus, $G' \cap D \ne \emptyset$.

(3) follows from the fact that $G$ is the upwards closure in $\p$ of 
$G'$, which is easy to check (in fact, $G'$ is just the set of members of $G$ which are 
less than or equal to $p$).
\end{proof}

\begin{notation}
We will write $*_N^\p(p,q,s)$ to mean the conjunction of 
$*_N^\p(p,s)$ and $*_N^\p(q,s)$.
\end{notation}

\begin{lemma}
Let $N$ be suitable for $\p$, $p$ a strongly $(N,\p)$-generic condition, 
and $D$ and $E$ dense below $p$. 
Suppose that $*_N^\p(q,r,s)$ holds, where $q$ and $r$ are below $p$. 
Then there is $(p',q',s') \le (p,q,s)$ such that $p' \in D$, $q' \in E$, 
$p'$ and $q'$ are below $s$, and $*_N^\p(p',q',s')$ 
holds.\footnote{Throughout the paper, we will adopt the following notation. 
For a forcing poset $\p$ and tuples of conditions 
$(p_0,\ldots,p_n)$ and $(q_0,\ldots,p_n)$ in $\p$, we will write 
$(q_0,\ldots,q_n) \le (p_0,\ldots,p_n)$ to mean that 
$q_i \le p_i$ for all $i = 0, \ldots, n$.}
\end{lemma}

\begin{proof}
Since $*_N^\p(p,s)$ holds, $p$ and $s$ are compatible, so fix $p' \le p, s$ in $D$. 
Since $p' \le s$, $p'$ and $s$ are obviously compatible, 
so by Lemma 1.7(1), there is $s_0 \le s$ in 
$N \cap \p$ such that $*_N^\p(p',s_0)$. 
As $*_N^\p(q,s)$ holds and $s_0 \le s$ in $N \cap \p$, 
$s_0$ and $q$ are compatible, 
so fix $q' \le q, s_0$ which is in $E$. 
Then $q'$ and $s_0$ are obviously compatible, so by Lemma 1.7(1) fix $s' \le s_0$ in 
$N \cap \p$ such that $*_N^\p(q',s')$. 
As $*_N^\p(p',s_0)$ holds and $s' \le s_0$ is in 
$N \cap \p$, $*_N^\p(p',s')$ holds. 
To summarize, 
we have $(p',q',s') \le (p,q,s)$, $p' \in D$, $q' \in E$, 
$p'$ and $q'$ are below $s$, and $*_N^\p(p',q',s')$ holds.
\end{proof}

\section{Chains of models and ineffable cardinals}

Let $\kappa$ be a regular uncountable 
cardinal which is fixed for the remainder of this section. 
Recall that an \emph{ideal} $I$ on $\kappa$ is a collection of subsets of $\kappa$ 
which includes the emptyset, is closed under subsets, and closed under finite unions. 
Let $I^+$ denote the collection of subsets of $\kappa$ which are not in $I$. 
We say that $I$ is \emph{proper} if $I^+ \ne \emptyset$, or equivalently, 
$\kappa \notin I$. 
Let $I^*$ denote the dual filter $\{ \kappa \setminus A : A \in I \}$.

We will say that a proper ideal $I$ on $\kappa$ is \emph{normal} if it contains every 
bounded subset of $\kappa$, and whenever $S \in I^+$ and 
$f : S \to \kappa$ is a regressive function, then there is $T \subseteq S$ in $I^+$ 
on which $f$ is constant. 
By standard arguments, 
if $I$ is normal, then $I$ contains every nonstationary set, and is $\kappa$-complete, 
which means that any union of fewer than $\kappa$ many sets in $I$ is in $I$. 
In addition, if $I$ is normal then $I^*$ is closed under diagonal intersections, which 
means that whenever $\{ C_i : i < \kappa \} \subseteq I^*$, then 
$\triangle \{ C_i : i < \kappa \} := 
\{ \alpha < \kappa : \forall i < \alpha, \ \alpha \in C_i \}$ is in $I^*$.

We will use ideals in the context of increasing chains of models. 

\begin{definition}
Let us say that a sequence $\langle N_i : i \in S \rangle$ is \emph{suitable} if 
for some regular $\theta > \kappa$:
\begin{enumerate}
\item each $N_i$ is an elementary substructure of $H(\theta)$, $\kappa \in N_i$, 
$N_i \cap \kappa \in \kappa$, and $|N_i| = |N_i \cap \kappa|$;
\item $N_i \in N_j$ for all $i < j$ in $S$;
\item if $\delta \in S$ is a limit point of $S$, 
then $N_\delta = \bigcup \{ N_i : i \in S \cap \delta \}$.
\end{enumerate}
\end{definition}

Standard arguments show that if $\langle N_i : i \in S \rangle$ is suitable, 
then there is a club $C \subseteq \kappa$ such that for all $i \in C \cap S$, 
$N_i \cap \kappa = i$. 
In fact, this is true if (2) above is weakened to $N_i \cap \kappa \in N_j$.

\begin{lemma}
Let $I$ be a normal ideal on $\kappa$. 
Let $\langle N_i : i \in S \rangle$ be suitable, with union $N$, and $S \in I^+$. 
Suppose that $f$, $g$, and $h$ are functions with domain $S$ 
such that for all $i \in S$, $f(i) \in N_i$, 
$g(i) \in N \setminus N_i$, 
and $h(i)$ is a subset of $N \setminus N_i$ 
of size less than $\kappa$. 
Then there is $T \subseteq S$ in $I^+$ such that 
for all $i < j$ in $T$, $f(i) = f(j)$, $g(i) \in N_j$, and 
$h(i) \cap h(j) = \emptyset$.
\end{lemma}

\begin{proof}
Fix a bijection $F : \kappa \to N$. 
A standard argument shows that there is a club $C_0 \subseteq \kappa$ such that 
for all $i \in C_0 \cap S$, $F[i] = N_i$. 
Now the function which maps $i \in S \cap C_0$ to 
$F^{-1}(f(i))$ is regressive, and hence fixed 
on a set $S' \subseteq S \cap C_0$ in $I^+$. 
Therefore, if $i < j$ are in $S'$, then 
$F^{-1}(f(i)) = F^{-1}(f(j))$, and so applying $F$, $f(i) = f(j)$.

For each $i \in S'$, fix $\gamma_i < \kappa$ such that 
$h(i) \cup \{ g(i) \} \subseteq N_{\gamma_i}$. 
Let $C_1$ be the club set of $\xi < \kappa$ such that for all 
$i \in S' \cap \xi$, $\gamma_i < \xi$. 
Let $T := S' \cap C_1$, which is in $I^+$.

Consider $i < j$ in $T$. 
Then $i$ and $j$ are in $S'$, so $f(i) = f(j)$ as noted above. 
Also $\gamma_i < j$, so $h(i) \cup \{ g(i) \} \subseteq N_j$. 
Hence, $g(i) \in N_j$. 
By the choice of $h$, $h(j) \cap N_j = \emptyset$. 
As $h(i) \subseteq N_j$, we have that $h(i) \cap h(j) = \emptyset$.
\end{proof}

The cardinal $\kappa$ is said to be 
\emph{ineffable} if whenever 
$\langle A_\alpha : \alpha < \kappa \rangle$ is a sequence such that 
$A_\alpha \subseteq \alpha$ for all $\alpha < \kappa$, 
then there exists $A \subseteq \kappa$ such that the set 
$\{ \alpha < \kappa : A \cap \alpha = A_\alpha \}$ is stationary in $\kappa$. 
We recall that if $\kappa$ is ineffable, then $\kappa$ is weakly compact. 
See \cite{ineffable} for more information about ineffable cardinals.

Given a sequence $\vec A = \langle A_\alpha : \alpha < \kappa \rangle$ as above, 
let us say that a set $S \subseteq \kappa$ is \emph{coherent for $\vec A$} 
if for all $\alpha < \beta$ in $S$, $A_\alpha = A_\beta \cap \alpha$. 
Note that if $S$ is coherent for $\vec{A}$, then the set 
$A := \bigcup \{ A_\alpha : \alpha \in S \}$ satisfies that 
$A \cap \alpha = A_\alpha$ for all $\alpha \in S$. 
It easily follows that $S$ is coherent for $\vec A$ iff there is a set $A \subseteq \kappa$ 
such that for all $\alpha \in S$, $A \cap \alpha = A_\alpha$. 
And $\kappa$ is ineffable iff for any sequence $\vec A$ there is a stationary coherent set.

Define the \emph{ineffability ideal} on $\kappa$ to be 
the collection of all sets $S \subseteq \kappa$ such that for some sequence $\vec A$ 
as above, $S$ does not contain a stationary subset which is coherent for $\vec A$. 
It is easy to check that this collection is indeed an ideal which contains all 
nonstationary subsets of $\kappa$. 
Observe that $\kappa$ is ineffable iff the ineffability ideal on $\kappa$ is proper. 
It is also true that if $\kappa$ is ineffable, then the ineffability ideal on $\kappa$ 
is a normal ideal (see \cite[Theorem 2.4]{ineffable}).

Assume for the rest of this section 
that $\kappa$ is ineffable, and let $J$ denote the ineffability ideal on $\kappa$. 
Note that a set $S$ is in $J^+$ iff for any sequence $\vec A$, 
$S$ contains a stationary subset which is coherent for $\vec A$. 

We will use ineffability in the context of suitable sequences. 
This form of ineffability is described in the next lemma.

\begin{lemma}
Let $\langle N_\alpha : \alpha \in S \rangle$ be suitable, with union $N$, and $S \in J^+$. 
Let $\langle B_\alpha : \alpha \in S \rangle$ be a sequence such that 
for all $\alpha \in S$, $B_\alpha \subseteq N_\alpha$. 
Then there exists a set $B \subseteq N$ and a stationary set $T \subseteq S$ such that for all 
$\alpha \in T$, $B \cap N_\alpha = B_\alpha$.
\end{lemma}

\begin{proof}
Fix a bijection $F : \kappa \to N$. 
By standard arguments, there is a club 
$C \subseteq \kappa$ such that for all $\alpha \in S \cap C$, 
$F[\alpha] = N_\alpha$. 
Then $S \cap C \in J^+$.

Note that for all $\alpha \in S \cap C$, 
$A_\alpha := F^{-1}[B_\alpha] \subseteq F^{-1}[N_\alpha] = \alpha$. 
Applying the fact that $S \cap C \in J^+$ to the sequence 
$\langle A_\alpha : \alpha \in S \cap C \rangle$, there exists a set 
$A \subseteq \kappa$ and a stationary set $T \subseteq S \cap C$ 
such that for all $\alpha \in T$, $A \cap \alpha = A_\alpha$. 
Define $B := F[A]$. 
Then $B \subseteq N$, and for all $\alpha \in T$, 
$$
B \cap N_\alpha = F[A] \cap N_\alpha = F[A] \cap F[\alpha] = 
F[A \cap \alpha] = F[A_\alpha] = 
F[F^{-1}[B_\alpha]] = B_\alpha.
$$
\end{proof}

The next result describes a consequence of ineffability which will be used later in the paper. 
Since the proof involves some tedious arguments about names, in order to keep it as simple 
as possible we deal with names a bit loosely. 
The interested reader can fill in the complete details.

\begin{proposition}
Let $\langle N_\alpha : \alpha \in S \rangle$ be a suitable chain of elementary substructures 
of $H(\theta)$, with union $N$, and $S \in J^+$, 
such that for all $\alpha \in S$, $N_\alpha^\omega \subseteq N_\alpha$. 
Let $U$ be the set of $\alpha \in S$ satisfying the following property: 
for all countably closed forcing posets $\p \in N_\alpha$ and all nice 
$\p$-names $\dot T \in N_\alpha$, 
if $H(\theta)$ models that $\p$ forces that 
$\kappa = \omega_2$ and $\dot T$ is a tree with underlying set 
$\kappa$ with no chains of order type $\kappa$, 
then the forcing poset $N_\alpha \cap \p$ forces 
that $\alpha = \omega_2$ and $N_\alpha \cap \dot T$ is 
a tree with no chains of order type $\alpha$. 
Then $S \setminus U$ is in $J$.
\end{proposition}

\begin{proof}
Since the set of inaccessibles below $\kappa$ is in $J^*$, we may assume 
without loss of generality that every member of $S$ is inaccessible.

Suppose for a contradiction that $S \setminus U$ is in $J^+$. 
For each $\alpha \in S \setminus U$, fix a counterexample $\p_\alpha$ and 
$\dot T_\alpha$ in $N_\alpha$. 
By Lemma 2.2, we can find $W \subseteq S \setminus U$ in $J^+$, 
$\p$, and $\dot T$ such that for all $\alpha \in W$, 
$\p_\alpha = \p$ and $\dot T_\alpha = \dot T$. 
For each $\alpha \in W$, 
since $\p$ is $\omega_1$-closed and $N_\alpha$ is countably closed, 
$N_\alpha \cap \p$ is countably closed, and 
obviously $N_\alpha \cap \dot T$ is an $(N_\alpha \cap \p)$-name. 
An easy argument using the elementarity of $N_\alpha$ shows that 
$N_\alpha \cap \dot T$ is an $(N_\alpha \cap \p)$-name for a tree.

For each $\alpha \in W$, either $N_\alpha \cap \p$ does not force that 
$\alpha = \omega_2$, or some condition in $N_\alpha \cap \p$ forces 
in $N_\alpha \cap \p$ that $N_\alpha \cap \dot T$ 
has a chain of order type $\alpha$. 
So we can find $W_1 \subseteq W$ in $J^+$ such that all $\alpha \in W_1$ satisfy the first possibility, 
or all $\alpha \in W_1$ satisfy the second possibility.

Let us assume that for all $\alpha \in W_1$, there is $p_\alpha \in N_\alpha \cap \p$ which forces 
in $N_\alpha \cap \p$ that $N_\alpha \cap \dot T$ has a chain of order type $\alpha$. 
We will omit the other case since the arguments are very similar to those in this case. 
By Lemma 2.2, 
we may assume without loss of generality that for some $p \in \p$, $p_\alpha = p$ for 
all $\alpha \in W_1$. 
For each $\alpha \in W_1$, fix a nice $(N_\alpha \cap \p)$-name $\dot b_\alpha$ 
for a subset of 
$N_\alpha \cap T$ which $p$ forces is a chain of order type $\alpha$.
Note that $\dot b_\alpha \subseteq N_\alpha$.

Applying Lemma 2.3, fix $b \subseteq N$ and a stationary set 
$Z \subseteq W_1$ such that for all $\alpha \in Z$, $b \cap N_\alpha = \dot b_\alpha$. 
Then $b$ is a $\p$-name for a subset of $\dot T$. 
We claim that $p$ forces that $b$ is a chain of $\dot T$ of order type $\kappa$, 
which is a contradiction.

First let us see that $p$ forces that $b$ is a chain. 
Let $q \le p$ and suppose that $q$ forces that $x$ and $y$ are in $b$. 
Then for each of $z \in \{ x, y \}$, 
the set $D_z$ of $r \le q$ such that $(r^+,\check z)$ is in $b$ for some $r \le r^+$ 
is dense below $q$. 
Let $E$ be the set of conditions below $q$ which decide in $\p$ if and how 
$x$ and $y$ compare in $\dot T$, and note that $E$ is dense below $q$.

It easily follows that there is a club $C \subseteq \kappa$ such that for all 
$\alpha \in C \cap Z$, the sets $N_\alpha \cap D_x$, $N_\alpha \cap D_y$, 
and $N_\alpha \cap E$ are dense below $q$ in $N_\alpha \cap \p$, and moreover,  
for $z \in \{ x, y \}$, each $r \in D_z \cap N_\alpha$ has a witness $r^+$ as described 
above also in $N_\alpha$. 
Fix $\alpha \in C \cap Z$ such that $q$, $x$, and $y$ are in $N_\alpha$. 
Then we can find $v \le q$ in $N_\alpha \cap \p$ such that $v$ is below members of 
these three dense subsets of $N_\alpha \cap \p$. 
Then clearly $v$ forces in $N_\alpha \cap \p$ that $x$ and $y$ are in 
$N_\alpha \cap b = \dot b_\alpha$. 
Since $v \le p$ and $p$ forces in $N_\alpha \cap \p$ that $\dot b_\alpha$ is a chain, 
$v$ forces in $N_\alpha \cap \p$ that $x$ and $y$ are comparable in $N_\alpha \cap \dot T$. 
By extending $v$ further in $N_\alpha \cap \p$ if necessary, we may assume that $v$ decides 
how $x$ and $y$ compare. 
Without loss of generality, assume that 
$v \Vdash_{N_\alpha \cap \p} \ x <_{N_\alpha \cap \dot T} y$.

We claim that $v \Vdash_\p x <_{\dot T} y$. 
If not, then by elementarity there is $w \le v$ in $N_\alpha \cap \p$ 
such that $w \Vdash_\p x \not <_{\dot T} y$. 
But since $w \le v$, $w \Vdash_{N_\alpha \cap \p} \ x <_{N_\alpha \cap \dot T} y$. 
In particular, we can find $w_1 \le w$ in $N_\alpha \cap \p$ 
such that $(w_1^+,(\check x,\check y)) \in N_\alpha \cap \dot T$ for some $w_1 \le w_1^+$. 
But then $w_1$ forces in $\p$ that $x <_{\dot T} y$, which is a contradiction.

To see that $p$ forces that $b$ has order type $\kappa$, it suffices to show that for all 
$\beta < \kappa$, $p$ forces that $b$ contains an ordinal above $\beta$. 
Let $q \le p$ and $\beta$ be given. 
Fix $\alpha \in Z$ such that $N_\alpha \cap \kappa = \alpha$ and 
$q$ and $\beta$ are in $N_\alpha$. 
Since $p$ forces in $N_\alpha \cap \p$ 
that $\dot b_\alpha = N_\alpha \cap b$ has order type $\alpha$, and 
$N_\alpha \cap \kappa = \alpha$ is inaccessible, 
we can find $\gamma > \beta$ in $N_\alpha$ and $r \le q$ such that 
$(r^+,\check \gamma) \in \dot b_\alpha = N_\alpha \cap b$ 
for some $r \le r^+$ in $N_\alpha \cap \p$. 
Then $(r^+,\check \gamma) \in b$ implies that $r$ forces in $\p$ 
that $\gamma > \beta$ is in $b$.
\end{proof}

\section{Trees}

We assume that the reader is familiar with the basic definitions and facts about trees. 
We introduce some notation. 
Let $(T,<_T)$ be a tree. 
A \emph{chain} in $T$ is a linearly ordered subset of $T$, 
and a \emph{branch} of $T$ is a maximal chain in $T$. 
For each $x \in T$, 
we let $\h_T(x)$ denote the height of $x$ in $T$. 
For each ordinal $\delta$, let 
$T(\delta) := \{ x \in T : \h_T(x) = \delta \}$ denote level $\delta$ of $T$, and 
$T \restriction \delta := \{ x \in T : \h_T(x) < \delta \}$. 
The \emph{height of $T$} is the least ordinal $\theta$ such that $T(\theta) = \emptyset$. 
A branch $b$ of $T$ is \emph{cofinal} if $b \cap T(\delta) \ne \emptyset$ for all $\delta$ less 
than the height of $T$.

For $x \in T$ and $\gamma \le \h_T(x)$, let $p_T(x,\gamma)$ denote the unique node $y$ of 
$T$ such that $y \le_T x$ and $\h_T(y) = \gamma$. 
A useful observation is that if $x$ and $y$ are in $T$, 
$\gamma \le \min(\h_T(x), \h_T(y))$, 
and $p_T(x,\gamma) \ne p_T(y,\gamma)$, then $x$ and $y$ are not comparable in $T$, 
and for all $\gamma < \xi \le \min(\h_T(x), \h_T(y))$, 
$p_T(x,\xi) \ne p_T(y,\xi)$.

Let $\kappa$ be a regular uncountable cardinal which is fixed for the remainder of the section. 
A \emph{$\kappa$-tree} is a tree of height $\kappa$ such that for $\delta < \kappa$, the level 
$T(\delta)$ has size less than $\kappa$. 
A $\kappa$-tree is a \emph{$\kappa$-Aronszajn tree} if it has no cofinal branch (or equivalently, 
no chain of order type $\kappa$). 
For a set $A \subseteq \kappa$, let $T \restriction A$ denote the subtree of $T$ consisting 
of all nodes $x$ such that $\h_T(x) \in A$.

A $\kappa$-tree $T$ is \emph{normal} if:
\begin{enumerate}
\item for every $x \in T$ and 
$\h_T(x) < \gamma < \kappa$, there is $y \in T$ with $x <_T y$ and $\h_T(y) = \gamma$;
\item if $x$ and $y$ are distinct nodes of $T$ 
with the same 
height $\delta$, where $\delta$ is a limit ordinal, 
then there is $\xi < \delta$ such that 
$p_T(x,\xi) \ne p_T(y,\xi)$;
\item for every node $x$, there are $y$ and $z$ with 
$x <_T y$, $x <_T z$, and such that $y$ and $z$ are not comparable in $T$.
\end{enumerate}

\begin{definition}
Let $T$ and $U$ be $\kappa$-trees. 
A function $f : T \to U$ is an \emph{isomorphism} if 
$f$ is a bijection and for all $x$ and $y$ in $T$, 
$x <_T y$ iff $f(x) <_U f(y)$. 
We say that $T$ and $U$ are \emph{isomorphic} if there 
exists an isomorphism from $T$ to $U$.  
\end{definition}

It is easy to verify that if $f : T \to U$ is an isomorphism, then for all $x \in T$, 
$\h_T(x) = \h_U(f(x))$.

\begin{lemma}
Let $T$ and $U$ be $\kappa$-trees. 
Suppose that $f : T \to U$ is a bijection satisfying that 
$\h_T(x) = \h_U(f(x))$ for all $x \in T$, and such that 
$x <_T y$ implies $f(x) <_U f(y)$ for all $x, y \in T$. 
Then $f$ is an isomorphism.
\end{lemma}

\begin{proof}
Assume that $x$ and $y$ are in $T$ and $f(x) <_U f(y)$. 
We will show that $x <_T y$. 
Let $\gamma := \h_T(x)$. 
Then $\gamma = \h_T(x) = \h_U(f(x)) < \h_U(f(y)) = \h_T(y)$. 

Suppose for a contradiction that $x \not <_T y$. 
Let $x' := p_T(y,\gamma)$. 
Then $x' \ne x$. 
And $\h_U(f(x')) = \h_T(x') = \gamma$. 
But $x' <_T y$ implies that $f(x') <_U f(y)$. 
So $p_U(f(y),\gamma) = f(x')$. 
We are assuming that $f(x) <_U f(y)$, so 
$p_U(f(y),\gamma) = f(x)$. 
Hence, $f(x) = f(x')$, which contradicts that $f$ is injective.
\end{proof}

\begin{definition}
Let $T$ and $U$ be $\kappa$-trees. 
We say that $T$ and $U$ are \emph{club isomorphic} if 
there exists a club set $C \subseteq \kappa$ such that 
$T \restriction C$ and $U \restriction C$ are isomorphic.
\end{definition}

\begin{lemma}
Let $T$ and $U$ be normal $\kappa$-trees. 
Assume that there exists a cofinal set $A \subseteq \kappa$ 
such that $T \restriction A$ and $U \restriction A$ are isomorphic. 
Then $T$ and $U$ are club isomorphic.
\end{lemma}

\begin{proof}
Let $C := A \cup \lim(A)$. 
Then $C$ is a club subset of $\kappa$. 
Fix an isomorphism $f : T \restriction A \to U \restriction A$. 
We will show that there is an isomorphism $g : T \restriction C \to U \restriction C$ 
such that $f \subseteq g$.

Let $\beta \in C \setminus A$ and $x \in T_\beta$, and we define $g(x)$. 
Note that $\beta \in \lim(A) \setminus A$. 
Let $\gamma := \min(A \setminus \beta)$. 
So $\beta < \gamma$.  
Using the normality of $T$, 
fix $x' \in T(\gamma)$ such that $x <_T x'$. 
Then $f(x') \in U(\gamma)$. 
Define $g(x) := p_U(f(x'),\beta)$.

Let us prove that $g$ is well-defined. 
So consider another $x'' \in T(\gamma)$ such that $x <_T x''$. 
We claim that $p_U(f(x'),\beta) = p_U(f(x''),\beta)$. 
Suppose for a contradiction that $p_U(f(x'),\beta) \ne p_U(f(x''),\beta)$. 
Since $U$ is normal and $\beta$ is a limit ordinal, 
there is $\xi < \beta$ such that $p_U(f(x'),\xi) \ne p_U(f(x''),\xi)$. 
Moreover, as $\beta$ is a limit point of $A$, without loss of generality 
we can assume that $\xi \in A \cap \beta$.

Fix $y'$ and $y''$ in $T(\xi)$ such that 
$f(y') = p_U(f(x'),\xi)$ and $f(y'') = p_U(f(x''),\xi)$. 
Then $y' \ne y''$. 
As $f(y') <_U f(x')$ and $f(y'') <_U f(x'')$, it follows that 
$y' <_T x'$ and $y'' <_T x''$. 
But $x <_T x', x''$, and therefore 
$y' <_T x$ and $y'' <_T x$. 
This is impossible since $y'$ and $y''$ are both on the same level of $T$ 
and are different.

Now we prove that $g$ is injective. 
Since $f$ is injective and $g$ preserves the heights of nodes, 
it suffices to show that if $x_0 \ne x_1$ are in $T(\beta)$, where 
$\beta \in \lim(A) \setminus A$, then $g(x_0) \ne g(x_1)$. 
Let $\gamma = \min(A \setminus \beta)$. 
Fix $x_0'$ and $x_1'$ above $x_0$ and $x_1$ in $T$,  respectively, of height $\gamma$. 
Then $x_0' \ne x_1'$, and by the definition of $g$, 
$g(x_0) = p_U(f(x_0'),\gamma)$ and $g(x_1) = p_U(f(x_1'),\gamma)$. 
Since $f$ is injective, $f(x_0') \ne f(x_1')$. 
As $T$ is normal and $\beta$ is a limit ordinal, there is $\xi \in A \cap \beta$ 
such that $p_T(x_0,\xi) \ne p_T(x_1,\xi)$. 
Since $f$ is injective, $f(p_T(x_0,\xi)) \ne f(p_T(x_1,\xi))$. 
But $f(p_T(x_0,\xi)) = f(p_T(x_0',\xi)) = p_U(f(x_0'),\xi)$ and 
$f(p_T(x_1,\xi)) = f(p_T(x_1',\xi)) = p_U(f(x_1'),\xi)$. 
So $g(x_0) = p_U(f(x_0'),\beta)$ and 
$g(x_1) = p_U(f(x_1'),\beta)$ must be different.

To see that $g$ is surjective, let $\beta \in \lim(A) \setminus A$ 
and consider $y \in U(\beta)$. 
Let $\gamma := \min(A \setminus \beta)$. 
Using the normality of $U$, 
fix $y' \in U(\gamma)$ such that $y <_U y'$. 
Then $y' = f(x')$ for some $x' \in T(\gamma)$. 
Let $x := p_T(x',\beta)$. 
Then by the definition of $g$, $g(x) = y$.

So we have that $g : T \restriction C \to U \restriction C$ is a bijection which 
extends $f$, and it is easy to check by the definition of $g$ 
that $x <_T y$ in $T \restriction C$ implies that 
$g(x) <_U g(y)$ in $U \restriction C$. 
We also have by the definition of $g$ that $g$ preserves the heights of nodes. 
It follows by Lemma 3.2 that $g$ is an isomorphism.
\end{proof}

We now assume for the remainder of the section that $\kappa = \omega_2$. 
Let us say that an $\omega_2$-tree $T$ is \emph{standard} if:
\begin{enumerate}
\item the nodes of $T$ are ordinals in $\omega_2$;
\item  $x <_T y$ implies that $x < y$ as ordinals;
\item every level of $T$ has size $\omega_1$, and 
every node in $T$ has $\omega_1$ many successors on the next level.
\end{enumerate}
Standard arguments show that for every normal $\omega_2$-tree $T$, there is a club 
$C \subseteq \omega_2$ such that $T \restriction C$ is isomorphic to a standard 
normal $\omega_2$-tree.

An $\omega_2$-tree $T$ is said to be \emph{countably closed} if every countable 
chain in $T$ has an upper bound. 
For a limit ordinal $\delta < \omega_2$, we say that level $\delta$ of $T$ is \emph{closed} 
if every cofinal branch of $T \restriction \delta$ 
has an upper bound. 
Note that $T$ is countably closed iff for every limit ordinal $\delta < \omega_2$ with 
cofinality $\omega$, level $\delta$ of $T$ is closed. 

The goal of the paper is to prove the consistency of the statement that 
\textsf{CH} holds and any two normal countably closed 
$\omega_2$-Aronszajn trees are club isomorphic. 
Note that by the preceding comments it suffices to show the consistency that 
any two standard normal countably closed 
$\omega_2$-Aronszajn trees are club isomorphic.

An $\omega_2$-tree $T$ is said to be \emph{special} if there exists a 
function $f : T \to \omega_1$ such that $x <_T y$ implies that 
$f(x) \ne f(y)$, for all $x$ and $y$ in $T$. 
Observe that a special $\omega_2$-tree is an $\omega_2$-Aronszajn tree.

The next result is well-known. 
For example, assuming \textsf{CH}, 
the construction suggested by Exercise III.5.44 of \cite{kunen} can be used 
to build countably closed normal special $\omega_2$-trees, and also normal 
special $\omega_2$-trees which are not closed at any level.

\begin{proposition}
Assume \textsf{CH}. 
Then there exists a countably closed normal special $\omega_2$-tree. 
There also exists a normal special $\omega_2$-tree which is not closed at any level.
\end{proposition}

It is easy to see that if $T$ and $U$ are $\omega_2$-trees, 
$C \subseteq \omega_2$ is club, and $T \restriction C$ and $U \restriction C$ 
are isomorphic, then for all $\delta \in \lim(C)$, level $\delta$ of $T$ is closed 
iff level $\delta$ of $U$ is closed.
Therefore, the two trees described in Proposition 3.5 are not club isomorphic.
Hence:

\begin{corollary}
Assuming \textsf{CH}, there are normal $\omega_2$-Aronszajn trees which are 
not club isomorphic.
\end{corollary}

Recall that an $\omega_2$-Suslin tree is an $\omega_2$-tree which has 
no chains or antichains of size $\omega_2$. 
Standard arguments show:
\begin{enumerate}
\item if $S$ is a normal $\omega_2$-tree, then $S$ is 
Suslin iff $S$ has no antichains of size $\omega_2$;
\item if $S$ is an $\omega_2$-Suslin tree, then $S$ is not special;
\item if there exists an $\omega_2$-Suslin tree, then there exists a normal 
$\omega_2$-Suslin tree.
\end{enumerate}

\begin{proposition}
Assume \textsf{CH}. 
Let $S$ be a normal $\omega_2$-Aronszajn tree. 
Then there exists a countably closed normal $\omega_2$-Aronszajn tree $U$ 
such that $S$ is a subtree of $U$ and for all $x \in S$, 
$\h_S(x) = \h_U(x)$.
\end{proposition}

\begin{proof}
Using Proposition 3.5, fix a countably closed normal 
$\omega_2$-Aronszajn tree $T$. 
Without loss of generality, we assume that $T$ has a unique node of height $0$. 
For each branch $b$ of $S$ whose order type has countable cofinality, 
define $T_b := \{ (b,x) : x \in T \}$. 
(Note that since $b$ is a branch, it is maximal and hence 
it has no upper bound in $S$.) 
If there are no such branches, then $S$ itself is countably closed, and we are 
done letting $U := S$. 
So assume that there are.

Define a tree $U$ whose underlying set consists of the nodes of $S$ together with 
nodes of the form $(b,x)$, where $b$ is a branch of $S$ whose order type has 
countable cofinality and $x \in T$. 
Let $x <_U y$ if either:
\begin{enumerate}
\item $x$ and $y$ are both in $S$ and $x <_S y$, or 
\item $x = (b,x_1)$ and $y = (b,y_1)$ for some $b$ 
and $x_1 <_T y_1$, or 
\item $x \in b$ and $y = (b,y_1)$ for some $y_1 \in T$.
\end{enumerate}
In other words, we expand $S$ by placing a copy of $T$ above every 
branch $b$ as above.

It is straightforward to check that $U$ is a tree of height $\omega_2$. 
The nodes of $U$ of height $\alpha$ consist of the nodes of $S$ of height $\alpha$, 
together with nodes of the form $(b,y)$, where 
$\alpha = \ot(b) + \h_T(y)$. 
By \textsf{CH}, there are only $\omega_1$-many possibilities for $b$ such that 
$\ot(b) \le \alpha$. 
Thus, $U$ is an $\omega_2$-tree. 
It is easy to see that $U$ is normal. 
Clearly $S$ is a subtree of $U$, and for all $x \in S$, the predecessors of $x$ in $U$ are 
exactly the predecessors of $x$ in $S$. 
Thus, $\h_S(x) = \h_U(x)$ for all $x \in S$.

Any chain $c$ of $U$ either lies entirely in $S$, or has a tail consisting of nodes of the form 
$(b,y)$ for some fixed $b$. 
If the order type of $c$ has countable cofinality, then in the first case it has an upper bound 
by the definition of $U$, and in the second case it has an upper bound since $T$ 
is countably closed. 
Therefore, $U$ is countably closed. 
If $c$ is cofinal in $U$, then in the first case this contradicts that $S$ is an Aronszajn tree, and in 
the second case this contradicts that $T$ is an $\omega_2$-Aronszajn tree.
Hence, $U$ is an $\omega_2$-Aronszajn tree.
\end{proof}

Recall that the $\omega_2$-Suslin hypothesis is the statement that 
there does not exist an $\omega_2$-Suslin tree.

\begin{thm}
Assume that \textsf{CH} holds and 
any two countably closed normal $\omega_2$-Aronszajn trees 
are club isomorphic. 
Then the $\omega_2$-Suslin hypothesis holds.
In fact, any normal $\omega_2$-Aronszajn tree is special.
\end{thm}

\begin{proof}
Recall that if there exists an $\omega_2$-Suslin tree, then there exists a normal 
$\omega_2$-Suslin tree. 
So it suffices to prove that any normal $\omega_2$-Aronszajn tree $S$ is special. 
Applying Proposition 3.7, fix a countably closed normal $\omega_2$-Aronszajn tree $U$ 
such that $S$ is a subtree of $U$ and $\h_S(x) = \h_U(x)$ for all $x \in S$. 
Using Proposition 3.5, fix a countably closed normal special $\omega_2$-tree $T$. 
Let $h : T \to \omega_1$ be a function such that $x <_T y$ implies that $h(x) \ne h(y)$ for 
all $x$ and $y$ in $T$.

We have that $T$ and $U$ are both countably closed normal $\omega_2$-Aronszajn trees. 
So by assumption we can fix a club $C \subseteq \omega_2$ and an isomorphism 
$f : U \restriction C \to T \restriction C$. 
Note that $S \restriction C$ is a normal $\omega_2$-Aronszajn tree, 
and since the nodes of $S$ have the same height in $S$ and $U$, 
$S \restriction C \subseteq U \restriction C$. 
Define $h^* : S \restriction C \to \omega_1$ by 
$h^*(x) = h(f(x))$. 

Let $x$ and $y$ be in $S \restriction C$, and assume that $x <_S y$. 
Then $x <_U y$. 
As $f$ is an isomorphism, $f(x) <_T f(y)$. 
So $h^*(x) = h(f(x)) \ne h(f(y)) = h^*(y)$. 
Thus, $S \restriction C$ is special. 
Now in general, if an $\omega_2$-tree is special on club many levels, then it is special, 
by a straightforward argument. 
So $S$ is special.
\end{proof}

\section{The forcing poset}

Assume that \textsf{CH} holds and $T$ and $U$ are normal countably closed 
$\omega_2$-Aronszajn trees whose levels have size $\omega_1$. 
We will develop a forcing poset for adding a club isomorphism from $T$ to $U$. 
This forcing poset is a natural generalization of the one defined by 
Abraham and Shelah \cite[Section 5]{AS} for adding a 
club isomorphism between two normal $\omega_1$-Aronszajn trees.

Many of the lemmas and claims in this section have easy proofs, 
which we will sometimes omit.

\begin{definition}
The forcing poset $\p(T,U)$ consists of all pairs $(A,f)$ satisfying:
\begin{enumerate}
\item $A$ is a countable subset of $\omega_2 \cap \cof(\omega_1)$;
\item $f$ is an injective function;
\item $\dom(f)$ is a countable subset of $T \restriction A$ 
satisfying:
\begin{enumerate}
\item if $x \in \dom(f)$, then for all $\gamma < \h_T(x)$ in 
$A$, $p_T(x,\gamma) \in \dom(f)$, and 
\item if $x \in \dom(f)$ and 
$\h_T(x) < \beta \in A$, then there is $y \in \dom(f)$ with 
$\h_T(y) = \beta$ and $x <_T y$;
\end{enumerate}
\item for all $x \in \dom(f)$, $\h_T(x) = \h_U(f(x))$;
\item for all $x$ and $y$ in $\dom(f)$, 
if $x <_T y$ then $f(x) <_U f(y)$.
\end{enumerate}
Let $(B,g) \le (A,f)$ if $A \subseteq B$ and $f \subseteq g$.
\end{definition}

If $p \in \p(T,U)$, we will write $p = (A_p,f_p)$. 
If $B \subseteq A_p$, we will abbreviate $f_p \restriction (T \restriction B)$ as 
$f_p \restriction B$.

Observe that if $A \subseteq \omega_2 \cap \cof(\omega_1)$ is countable, then 
$(A,\emptyset) \in \p(T,U)$.

Note that if $p \in \p(T,U)$, then $\ran(f_p)$ is a countable subset 
of $U \restriction A_p$ satisfying (i) if $x \in \ran(f_p)$, then for all $\gamma < \h_U(x)$ in 
$A_p$, $p_U(x,\gamma) \in \ran(f_p)$, and (ii) if $x \in \ran(f_p)$ and 
$\h_U(x) < \beta \in A_p$, then there is $y \in \ran(f_p)$ with 
$\h_U(y) = \beta$ and $x <_U y$. 
As a consequence, one can easily check that 
the function which maps $p \in \p(T,U)$ 
to $(A_p,f_p^{-1})$ is an isomorphism from the forcing poset $\p(T,U)$ to $\p(U,T)$.

\begin{lemma}
The forcing poset $\p(T,U)$ is countably closed. 
In fact, if $\langle p_n : n < \omega \rangle$ is a descending sequence of conditions, 
then $(A,f)$ is the greatest lower bound, where 
$A := \bigcup \{ A_{p_n} : n < \omega \}$ and 
$f := \bigcup \{ f_{p_n} : n < \omega \}$.
\end{lemma}

\begin{notation}
If $p$ is a condition and $\gamma < \omega_2$, let 
$p \restriction \gamma := 
(A_p \cap \gamma,f \restriction (A_p \cap \gamma))$.
\end{notation}

\begin{lemma}
For any $p \in \p(T,U)$ and $\gamma < \omega_2$, 
$p \restriction \gamma$ is in $\p(T,U)$ and $p \le p \restriction \gamma$.
\end{lemma}

\begin{lemma}
Let $p \in \p(T,U)$. 
Suppose that $B$ is a countable subset of $\omega_2 \cap \cof(\omega_1)$ 
such that $\sup(A_p) < \min(B)$. 
Then there is $q \le p$ in $\p(T,U)$ satisfying:
\begin{enumerate}
\item $A_q = A_p \cup B$;
\item $f_q \restriction A_p = f_p$.
\end{enumerate}
\end{lemma}

\begin{proof}
We first prove the lemma in the case that $B = \{ \beta \}$ is a singleton. 
For each $x \in \dom(f_p)$, construct a cofinal branch $b_x$ in the tree 
$\dom(f)$ which contains $x$. 
Specifically, if $A_p$ has a maximum element, 
then just pick a node in $\dom(f)$ with height $\max(A_p)$ which is 
above $x$, and let $b_x$ be the set of nodes of $T \restriction A_p$ 
less than or equal to that node. 
If $A_p$ does not have a maximum element, then inductively build a chain of 
nodes above $x$ with order type $\omega$ and cofinal in $\dom(f)$, 
and let $b_x$ be the downward closure of that chain in $T \restriction A_p$.

Using the fact that $T$ is countably closed and each branch 
$b_x$ either has a top element or has an order type with countable cofinality, 
we can conclude that each $b_x$ has an upper bound in $T$. 
Since $T$ is normal, fix $m_x \in T$ which is an upper bound of $b_x$ 
with height $\beta$. 
Moreover, since $\sup(A_p) < \beta$, each chain $b_x$ has $\omega_1$ 
many upper bounds in $T_\beta$. 
Hence, we can arrange that if $x \ne y$, then $m_x \ne m_y$ (even if 
it happens that $b_x = b_y$, which is possible). 

For each $x \in \dom(f_p)$, let $c_x := f_p[b_x]$, which is a cofinal branch of $\ran(f_p)$ 
containing $f(x)$. 
Since $U$ is countably closed, each chain $c_x$ has an upper bound. 
As $\sup(A_p) < \beta$, each $c_x$ has $\omega_1$ many upper 
bounds of height $\beta$. 
So we can fix an upper bound $n_x$ of $c_x$ with height $\beta$ 
in such a way that $x \ne y$ implies $n_x \ne n_y$.

Now define $q$ by letting $A_q := A_p \cup \{ \beta \}$, 
$\dom(f_q) := \dom(f_p) \cup \{ m_x : x \in \dom(f_p) \}$, 
$f_q(x) := f_p(x)$ for all $x \in \dom(f_p)$, and 
$f_q(m_x) := n_x$ for all $x \in \dom(f_p)$.
It is easy to check that $q$ is as required.

To prove the lemma in general, enumerate $B$ in increasing order 
as $\langle \beta_i : i < \delta \rangle$, where $\delta < \omega_1$. 
Since $B$ is countable and 
consists of ordinals of cofinality $\omega_1$, for all $i < \delta$ 
we have that $\sup(A_p \cup \{ \beta_j : j < i \}) < \beta_i$. 

Define by induction a descending sequence of conditions 
$\langle q_i : i \le \delta \rangle$ as follows. 
Let $q_0 = p$. 
Suppose $q_i$ is given, where $i < \delta$ and 
$A_{q_i} = A_p \cup \{ \beta_j : j < i \}$. 
Then $\sup(A_{q_i}) < \beta_i$. 
By the case just proven, we can find $q_{i+1} \le q_i$ satisfying that 
$\dom(A_{q_{i+1}}) = \dom(A_{q_i}) \cup \{ \beta_i \}$ and 
$f_{q_{i+1}} \restriction A_{q_{i}} = f_{q_i}$.  

Suppose that $\delta_0 \le \delta$ is a limit ordinal and $q_i$ is defined 
for all $i < \delta_0$. 
Let $q_{\delta_0}$ be the greatest lower bound 
of $\langle q_i : i < \delta_0 \rangle$. 
This completes the construction. 
It is easy to check that $q := q_{\delta}$ is as required.
\end{proof}

\begin{lemma}
Let $p \in \p(T,U)$, $\beta \in A_p$, and $x \in T$ 
with $\h_T(x) = \beta$. 
Then there is $q \le p$ such that $x \in \dom(f_q)$.
\end{lemma}

\begin{proof}
If $x \in \dom(f_p)$, then we are done, so assume not. 
Then we also have that for all $y \in \dom(f_p)$, 
$x \not <_T y$. 
Extending $p$ if necessary using Lemma 4.5, we may assume 
without loss of generality that $\max(A_p)$ exists. 
Let $\beta := \max(A_p)$. 
We may also assume that $\h_T(x) = \beta$. 
For otherwise fix $x' \in T$ with $\h_T(x') = \beta$ and $x <_T x'$, and 
note that if $q \le p$ satisfies that $x' \in \dom(f_q)$, then 
also $x \in \dom(f_q)$.

Define $b := \{ y \in \dom(f_p) : y <_T x \}$ and $c := f_p[b]$ (it is possible that 
$b$ and $c$ are empty). 
Let $\gamma$ be the least member of $A_p$ such that $p_T(x,\gamma) \notin \dom(f_p)$, 
which exists since $x \notin \dom(f_p)$. 
Fix an upper bound $z_0$ of $c$ in $U$ with height $\gamma$ which is not in $\ran(f_p)$. 
This is possible by the countable closure and normality of $U$, together with the 
fact that $\cf(\gamma) = \omega_1$ whereas $c$ is countable. 
Fix $z \in U$ with height $\beta$ which is above $z_0$ in $U$. 
Note that since $z_0 \notin \ran(f_p)$, 
the set $\{ y : z_0 \le_U y \le_U z \}$ is disjoint from $\ran(f_p)$.

Define $A_q := A_p$ and let  
$\dom(f_q)$ be the union of the disjoint sets 
$\dom(f_p)$ and $\{ y : y \le_T x, \ \h_T(y) \in A \setminus \gamma \}$. 
Note that every member of $\dom(f_q) \setminus \dom(f_p)$ is 
an upper bound of $b$, and also $x \in \dom(f_q)$. 
Define $f_q$ so that $f_q \restriction \dom(f_p) = f_p$ and 
for all $y \in \dom(f_q) \setminus \dom(f_p)$, $f_q(y) := p_U(z,\h_T(y))$.
It is straightforward to check that $q = (A_q,f_q)$ is as required.
\end{proof}

\begin{lemma}
Let $p \in \p(T,U)$, $\beta \in A_p$, and $y \in U$ 
with $\h_U(y) = \beta$. 
Then there is $q \le p$ such that $y \in \ran(f_q)$.
\end{lemma}

\begin{proof}
Consider $(A_p,f_p^{-1})$, which is in $\p(U,T)$. 
By the previous lemma applied in $\p(U,T)$, 
there is $r \le (A_p,f_p^{-1})$ with $y \in \dom(f_r)$. 
Then $(A_r,f_r^{-1})$ is below $p$ in $\p(T,U)$ and satisfies that 
$y \in \ran(f_r^{-1})$.
\end{proof}

Since $\p(T,U)$ is countably closed, it preserves $\omega_1$. 
The issue of whether $\p(T,U)$ preserves $\omega_2$ will be dealt 
with in Section 7.

\begin{proposition}
Assume that $\p(T,U)$ preserves $\omega_2$. 
Then $\p(T,U)$ forces that $T$ and $U$ are club isomorphic.
\end{proposition}

\begin{proof}
Let $G$ be a generic filter on $\p(T,U)$. 
Define $A := \bigcup \{ A_p : p \in G \}$ and $f := \bigcup \{ f_p : p \in G \}$. 
An easy density argument using Lemma 4.5 shows that $A$ is unbounded in $\omega_2$, 
and similarly Lemmas 4.6 and 4.7 can be used to show that 
$\dom(f) = T \restriction A$ and $\ran(f) = U \restriction A$. 

The definition of the forcing poset $\p(T,U)$ implies that 
$f : T \restriction A \to U \restriction A$ is a bijection, 
$\h_T(x) = \h_U(f(x))$ for all $x \in \dom(f)$, 
and $x <_T y$ in $\dom(f)$ implies that $f(x) <_U f(y)$. 
By Lemma 3.2, $T \restriction A$ and $U \restriction A$ are isomorphic. 
By Lemma 3.4, $T$ and $U$ are club isomorphic.
\end{proof}

\begin{definition}
Let $p \in \p(T,U)$ and $\beta \in A_p$. 
We say that $p$ is \emph{injective on $\beta$} if for all distinct 
$x$ and $y$ in $\dom(f_p)$ with height $\beta$, 
there exists $\gamma \in A_p \cap \beta$ such that 
$p_T(x,\gamma) \ne p_T(y,\gamma)$.
\end{definition}

\begin{lemma}
Let $p \in \p(T,U)$ and $\beta \in A_p$ be a limit point of $\omega_2 \cap \cof(\omega_1)$. 
Then there is $q \le p$ such that $q$ is injective on $\beta$.
\end{lemma}

\begin{proof}
Since $T$ and $U$ are normal, $\cf(\beta) = \omega_1$, and 
$\dom(f_p)$ and $\ran(f_p)$ are countable, we can find $\gamma < \beta$ 
with cofinality $\omega_1$ such that:
\begin{enumerate}
\item $\sup(A_p \cap \beta) < \gamma$;
\item for all $x \ne y$ in $\dom(f_p)$ with height $\beta$, 
$p_T(x,\gamma) \ne p_T(y,\gamma)$;
\item for all $x \ne y$ in $\ran(f_p)$ with height $\beta$, 
$p_U(x,\gamma) \ne p_U(y,\gamma)$.
\end{enumerate}

Now define $q$ by letting $A_q := A_p \cup \{ \gamma \}$, 
$$
\dom(f_q) := \dom(f_p) \cup \{ p_T(x,\gamma) : x \in \dom(f_p), \ \h_T(x) = \beta \},
$$
$f_q \restriction \dom(f_p) := f_p$, and for all $x \in \dom(f_p)$ with height $\beta$, 
$$
f_q(p_T(x,\gamma)) := p_U(f_p(x),\gamma).
$$
It is easy to check that $q$ is a condition, 
$q \le p$, and $q$ is injective on $\beta$.
\end{proof}

\begin{lemma}
Let $\langle p_n : n < \omega \rangle$ be a descending sequence of 
conditions in $\p(T,U)$ and let $\beta \in A_{p_0}$. 
Suppose that for all $n < \omega$, $p_n$ is injective on $\beta$. 
Then the greatest lower bound of $\langle p_n : n < \omega \rangle$ 
is injective on $\beta$.
\end{lemma}

\begin{lemma}
Let $p \in \p(T,U)$ and $\beta \in A_p$, and assume that $p$ is injective on $\beta$. 
Suppose that $\xi$ is an ordinal with cofinality $\omega_1$ such that 
$\sup(A_p \cap \beta) < \xi < \beta$. 
Then:
\begin{enumerate}
\item for all distinct $x$ and $y$ in $\dom(f_p)$ with height $\beta$, 
$p_T(x,\xi) \ne p_T(y,\xi)$;
\item for all distinct $x$ and $y$ in $\ran(f_p)$ with height $\beta$, 
$p_U(x,\xi) \ne p_U(y,\xi)$.
\end{enumerate}
\end{lemma}

\begin{proof}
(1) is almost immediate. 
For (2), fix $x'$ and $y'$ with $f_p(x') = x$ and $f_p(y') = y$. 
Then $x'$ and $y'$ are distinct nodes in $\dom(f)$ with height $\beta$. 
Since $p$ is injective on $\beta$, fix $\gamma \in A_p \cap \beta$ 
such that $p_T(x',\gamma) \ne p_T(y',\gamma)$. 
As $f_p$ is injective, 
$x_0 := f_p(p_T(x',\gamma))$ and $y_0 := f_p(p_T(y',\gamma))$ are distinct 
nodes in $U$ of height $\gamma$. 
Also, $x_0 <_U x$ and $y_0 <_U y$. 
Since $x_0 \ne y_0$, 
$x_0 <_U p_U(x,\xi) <_U x$, and $y_0 <_U p_U(y,\xi)$, 
clearly $p_U(x,\xi) \ne p_U(y,\xi)$.
\end{proof}

\begin{lemma}
Let $p \in \p(T,U)$, $\beta \in A_p$, and assume that $p$ is injective on $\beta$. 
Suppose that $X \subseteq \beta \cap \cof(\omega_1)$ is countable 
and $\sup(A_p \cap \beta) < \min(X)$. 
Then there exists $q \le p$ satisfying:
\begin{enumerate}
\item $A_q = A_p \cup X$;
\item $f_q \restriction A_p = f_p$.
\end{enumerate}
\end{lemma}

\begin{proof}
We first handle the case in which $X = \{ \xi \}$ is a singleton. 
Define $A_q := A_p \cup \{ \xi \}$ and 
$$
\dom(f_q) := 
\dom(f_p) \cup \{ p_T(x,\xi) : x \in \dom(f_p), \ \h_T(x) = \beta \}.
$$
Define $f_q \restriction \dom(f_p) := f_p$, and for all 
$x \in \dom(f_p)$ with height $\beta$, 
$f_q(p_T(x,\xi)) := p_U(f_p(x),\xi)$.

We claim that $f_q$ is well-defined and injective. 
If $x' \in \dom(f_q)$ with height $\xi$, then by Lemma 4.12(1) there is a 
unique $x \in \dom(f_p)$ with height $\beta$ such that $x' = p_T(x,\xi)$. 
Thus, $f_q$ is well-defined. 
If $x'$ and $y'$ are distinct nodes in $\dom(f_q)$ with height $\xi$, 
then let $x$ and $y$ be the unique 
nodes above $x'$ and $y'$, respectively, in $\dom(f_p)$ with height $\beta$. 
Then obviously $x \ne y$, so $f_p(x) \ne f_p(y)$ since $f_p$ is injective. 
By Lemma 4.12(2), $f_q(x') = p_U(f_p(x),\xi)$ and $f_q(y') = p_U(f_p(y),\xi)$ are distinct. 
Thus, $f_q$ is injective.
It easily follows that $q$ is as required. 
Also note that $q$ is injective on $\beta$.

To prove the lemma in general, 
enumerate $X$ in increasing order as $\langle \xi_i : i < \delta \rangle$, 
where $\delta < \omega_1$. 
Define a descending sequence $\langle q_i : i \le \delta \rangle$ 
as follows. 
Let $q_0 := p$. 
Let $i < \delta$ and assume that $q_i$ is defined so that 
$A_{q_i} = A_p \cup \{ \xi_j : j < i \}$ and $q_i$ is injective on $\beta$. 
Since $\xi_i$ has cofinality $\omega_1$, 
$\sup(A_{q_i} \cap \beta) < \xi_i$. 
So by the previous paragraph, we can find $q_{i+1} \le q_i$ injective on $\beta$ 
such that 
$\dom(q_{i+1}) = \dom(q_i) \cup \{ \xi_i \}$ and 
$f_{q_{i+1}} \restriction \dom(f_{q_i}) = f_{q_i}$.

Let $\delta_0 \le \delta$ be a limit ordinal, and assume that for all $i < \delta_0$, 
$q_i$ is defined so that 
$A_{q_i} = A_p \cup \{ \xi_j : j < i \}$ and $q_i$ is injective on $\beta$. 
Let $q_{\delta_0}$ be the greatest lower bound of $\langle q_i : i < \delta_0 \rangle$. 
Then $A_{q_{\delta_0}} = A_p \cup \{ \xi_i : i < \delta \}$ and 
$q_{\delta_0}$ is injective on $\beta$ by Lemma 4.11. 
This completes the construction. 
It is easy to see that $q_\delta$ is as required.
\end{proof}

\begin{proposition}
Let $(A,f)$ and $(B,g)$ be in $\p(T,U)$. 
Assume that $\gamma < \xi$, where $\gamma \in A$ and $\xi \in B$, 
and the following properties hold:
\begin{enumerate}
\item $A \subseteq \xi$;
\item $(A,f) \restriction \gamma = (B,g) \restriction \xi$;
\item $(A,f)$ is injective on $\gamma$ and $(B,g)$ is injective on $\xi$.
\end{enumerate}
Assume, moreover, that every node of $\dom(f)$ with height $\gamma$ is 
incomparable in $T$ with every node of $\dom(g)$ with height $\xi$, and 
every node of $\ran(f)$ with height $\gamma$ is incomparable in $U$ with 
every node in $\ran(g)$ with height $\xi$. 
Then $(f,A)$ and $(g,B)$ are compatible.
\end{proposition}

\begin{proof}
Since $A$ is countable and $\cf(\xi) = \omega_1$, 
$\sup(A) < \xi = \min(B \setminus \xi)$. 
By Lemma 4.5, we can find $(A',f') \le (A,f)$ such that 
$A' = A \cup (B \setminus \xi)$ and $f' \restriction A = f$. 
By property (2), $A \cap \gamma = B \cap \xi$, so $B \cap \xi \subseteq A$. 
Hence, $A' = A \cup B$.

Since $B \cap \xi \subseteq \gamma$, 
$B$ is countable, and $\cf(\gamma) = \omega_1$, 
it follows that $\sup(B \cap \xi) < \gamma$. 
So $A \setminus \gamma$ is a countable subset of $\xi \cap \cof(\omega_1)$ 
and $\sup(B \cap \xi) < \gamma = \min(A \setminus \gamma)$. 
As $(B,g)$ is injective on $\xi$, by Lemma 4.13 we can fix 
$(g',B') \le (g,B)$ such that $B' = B \cup (A \setminus \gamma)$ 
and $g' \restriction B = g$. 
By property (2), $A \cap \gamma = B \cap \xi \subseteq B$. 
Therefore, $B' = A \cup B$.

Claim 1: If $x \in \dom(f')$ has height at least $\xi$ and 
$y \in \dom(g')$ has height at least $\xi$, then $x$ and $y$ are 
incomparable in $T$. 
Suppose for a contradiction that $x$ and $y$ are comparable in $T$. 
Then $p_T(x,\gamma) = p_T(y,\gamma)$, and so $p_T(x,\gamma) <_T p_T(y,\xi)$. 
As $f' \restriction A = f$ and $\gamma \in A$, $p_T(x,\gamma) \in \dom(f)$. 
And as $g' \restriction B = g$ and $\xi \in B$, 
$p_T(y,\xi) \in \dom(g)$. 
So $p_T(x,\gamma) \in \dom(f)$ has height $\gamma$ 
and $p_T(y,\xi) \in \dom(g)$ has height $\xi$, and therefore these nodes are 
incomparable in $T$ by assumption, giving a contradiction.

Claim 2: If $x \in \ran(f')$ has height at least $\xi$ 
and $y \in \ran(g')$ has height at least $\xi$, then 
$x$ and $y$ are incomparable in $U$.
The proof is similar to that of Claim 1.

Claim 3: Every node $x$ in $\dom(f')$ 
with $\h_T(x) \in A \setminus \gamma$ 
is incomparable in $T$ with every node $y$ 
in $\dom(g')$ with $\h_T(y) \in A \setminus \gamma$. 
Fix $z$ in $\dom(g')$ with height $\xi$ such that $y <_T z$. 
Since $g' \restriction B = g$, $z \in \dom(g)$. 
Suppose for a contradiction that $x$ and $y$ are comparable in $T$. 
Then $p_T(x,\gamma) = p_T(y,\gamma) <_T y <_T z$. 
Since $f' \restriction A  = f$, $p_T(x,\gamma) \in \dom(f)$. 
So $p_T(x,\gamma) \in \dom(f)$ has height $\gamma$, 
$z \in \dom(g)$ has height $\xi$, so by assumption, 
$p_T(x,\gamma)$ and $z$ are incomparable in $T$, 
which is a contradiction.

Claim 4: Every node $x$ in $\ran(f')$ with $\h_U(x) \in A \setminus \gamma$ 
is incomparable in $U$ with every node $y$ 
in $\ran(g')$ with $\h_U(y) \in A \setminus \gamma$. 
The proof is similar to that of Claim 3.

Define $(C,h)$ by letting $C := A \cup B$ and $h := f' \cup g'$. 
We claim that $(C,h)$ is a condition which extends $(A,f)$ and $(B,g)$. 
It is easy to see that if $(C,h)$ is a condition, then it extends 
$(A,f)$ and $(B,g)$. 
Clearly $C$ is a countable subset of $\omega_2 \cap \cof(\omega_1)$. 
By property (2), 
$f \restriction (A \cap \gamma) = g \restriction (B \cap \xi)$. 
As $A \subseteq \xi$ by property (1), it easily follows that $f \cup g$ is a function, 
since any node in $\dom(f) \cap \dom(g)$ has height 
in $B \cap \xi = A \cap \gamma$, and $f$ and $g$ agree on such nodes. 

Since $g' \restriction B = g$ and $f' \restriction A = f$, 
Claims 1 and 3 imply that any node in $\dom(f') \cap \dom(g')$ is 
in $\dom(f) \cap \dom(g)$. 
Namely, any node in $\dom(f') \cap \dom(g')$ has height below $\xi$ by Claim 1, and 
therefore by Claim 3 and the fact that $A' \cap \xi = A$ 
has height below $\gamma$. 
Hence, any such node has height in $A' \cap \gamma \subseteq A$ and in 
$B' \cap \gamma \subseteq B$.  
Therefore, any such node is in $\dom(f' \restriction A) = \dom(f)$ and in 
$\dom(g' \restriction B) = \dom(g)$. 
So $h = f' \cup g'$ is a function.

Note that $\h_T(x) = \h_U(h(x))$ for all $x \in \dom(h)$, since this is true of $f'$ and $g'$. 
Suppose that $x$ and $y$ are in $\dom(h)$ and $h(x) = h(y)$. 
Then $\h_T(x) = \h_U(h(x)) = \h_U(h(y)) = \h_T(y)$, so $\h_T(x) = \h_T(y)$. 
We will show that $x = y$, which proves that $h$ is injective. 
If $x$ and $y$ are either both in $\dom(f')$ or both in $\dom(g')$, then we are done since 
$f'$ and $g'$ are injective.

Assume that $x \in \dom(f') \setminus \dom(g')$ and $y \in \dom(g') \in \dom(f')$. 
The case with $f'$ and $g'$ switched follows by a symmetric argument. 
Then $f'(x) = g'(y)$. 
Claims 2 and 4 imply that the height of $f'(x) = g'(y)$ is less than $\gamma$. 
So $\h_T(x) = \h_T(y) < \gamma$. 
But then $y \in \dom(g' \restriction \gamma) = \dom(f \restriction \gamma)$, 
which contradicts our assumption that $y \notin \dom(f')$. 

So $(C,h)$ satisfies properties (1), (2), and (4) of Definition 4.1. 
Since $\dom(h) = \dom(f') \cup \dom(g')$, 
property (3) follows from the fact that $(A',f')$ and $(B',g')$ are conditions.

It remains to prove property (5). 
Let $x <_T y$ in $\dom(h)$, and we will prove that $h(x) <_U h(y)$. 
If $x$ and $y$ are either both in $\dom(f')$ or both in $\dom(g')$, then we are done 
since $(A',f')$ and $(B',g')$ are conditions. 
There remains two cases, namely (I) $x \in \dom(f') \setminus \dom(g')$ and 
$y \in \dom(g') \setminus \dom(f')$, or (II) 
$x \in \dom(g') \setminus \dom(f')$ and $y \in \dom(f') \setminus \dom(g')$. 
We claim that neither case is possible.

Since $f' \restriction A = f$ and $A' = A \cup (B \setminus \xi)$, 
any node of $\dom(f')$ with height below $\gamma$ has height in $A \cap \gamma$ 
and is in $\dom(f \restriction \gamma) = \dom(g \restriction \gamma)$. 
And any node of $\dom(g')$ with height below $\gamma$ has height in 
$B' \cap \gamma \subseteq B$, and since $g' \restriction B = g$, is in 
$\dom(g \restriction \gamma) = \dom(f \restriction \gamma)$. 
It follows that in either case (I) or (II), $\h_T(x) \ge \gamma$. 
But $x <_T y$ implies that $p_T(x,\gamma) = p_T(y,\gamma)$. 
So $p_T(x,\gamma)$ is in both $\dom(f')$ and $\dom(g')$, 
and has height $\gamma \in A \setminus \gamma$, contradicting Claim 3.
\end{proof}

\section{The forcing iteration}

In this section we will develop the forcing iteration which we will use to prove the 
consistency of the statement that any two countably closed normal $\omega_2$-Aronszajn 
trees are club isomorphic. 
This forcing will be a kind of countable support iteration 
of the forcing poset developed in Section 4, together with some 
rudimentary kind of side conditions.

For the remainder of the paper we fix an inaccessible cardinal $\kappa$ such that 
$2^\kappa = \kappa^+$. 
We define by induction a sequence of forcing posets $\langle \p_i : i \le \kappa^+ \rangle$. 
We maintain as inductive hypotheses the assumptions that each $\p_i$ is 
countably closed, preserves $\kappa$, is $\kappa^+$-c.c., and forces that 
$\kappa = \omega_2$ if $i > 0$.  
The proof of the preservation of $\kappa$ is complex and will be handled in Section 7. 
Our definition of the forcing posets will depend on these inductive hypotheses together with 
a fixed bookkeeping function which is in the background.

For each $\gamma < \kappa$, let $I_\gamma$ denote the interval of ordinals 
$[\omega_1 \cdot \gamma, \omega_1 \cdot (\gamma+1))$. 
We will assume that for any tree $T$ we are working with, the nodes on 
level $\gamma$ of $T$ belong to $I_\gamma$.

\bigskip

We now begin the definition of the forcing iteration. 
We will let $\p_0$ denote the trivial forcing and $\p_1$ the Levy collapse 
$\coll(\omega_1,<\! \kappa)$. 
However, we will write conditions in $\p_0$ and $\p_1$ in a specific form so they fit 
in with the later definitions.

So let $\p_0$ denote the trivial forcing whose single element is the pair  
$(\emptyset,\emptyset)$. 
Let $\p_1$ be the forcing whose conditions are either $(\emptyset,\emptyset)$, or 
$(p,\emptyset)$, where $p$ is a function with domain $\{ 0 \}$ such that 
$p(0) \in \coll(\omega_1,<\! \kappa)$. 
Let $(q,\emptyset) \le (p,\emptyset)$ in $\p_1$ if either $p = \emptyset$, 
or else $p$ and $q$ are nonempty 
and $q(0) \le p(0)$ in $\coll(\omega_1,<\! \kappa)$. 
Note that $\p_1$ is countably closed, and since $\kappa$ is inaccessible, 
it is also $\kappa$-c.c.\! by a standard $\Delta$-system lemma argument. 
We also define, for all $(p,\emptyset) \in \p_1$, 
$(p,\emptyset) \restriction 0 := (\emptyset,\emptyset)$, which is in $\p_0$.

Let $\p_\alpha := \p_1$ for all $\alpha \le \kappa$. 
In other words, we do not force with anything between $1$ and $\kappa$ 
in the iteration.

\bigskip

Now assume that $\kappa \le \beta < \kappa^+$ and $\p_\beta$ is defined, and 
we will define $\p_{\beta+1}$. 
We assume that our bookkeeping function provides us with a pair of 
nice $\p_\beta$-names $\dot T_\beta$ and $\dot U_\beta$ for standard 
countably closed normal $\kappa$-Aronszajn trees. 
Without loss of generality, we will assume that for all $\gamma < \kappa$, $\p_\beta$ 
forces that the nodes of $\dot T_\beta$ on level $\gamma$ 
consist of all even ordinals in $I_\gamma$, 
and the nodes of $\dot U_\beta$ on level $\gamma$ consist of all 
odd ordinals in $I_\gamma$.

Define $\p_{\beta+1}$ to be the set of all pairs $p = (a_p,X_p)$ satisfying:
\begin{enumerate}
\item $a_p$ is a function whose domain is a countable subset of $\beta+1$;
\item $X_p$ is a function whose domain is a countable subset of $\beta+2$;
\item $p \restriction \beta := (a_p \restriction \beta, X_p \restriction (\beta+1)) \in \p_\beta$;
\item if $\beta \in \dom(a_p)$, then $a_p(\beta)$ is a 
$\p_\beta$-name such that 
$p \restriction \beta$ forces that 
$$
a_p(\beta) = (A_{a_{p(\beta)}}, f_{a_{p(\beta)}}) \in 
\p(\dot T_\beta,\dot U_\beta);\footnote{Since we 
would like our forcings to be in $H(\kappa^{++})$, we 
will implicitly assume that 
$A_{a_{p(\beta)}}$ and $f_{a_{p(\beta)}}$ are nice names, and $a_p(\beta)$ is a 
canonical name for their pair.}
$$
\item if $\beta+1 \in \dom(X_p)$, then $X_p(\beta+1)$ is a countable subset 
of 
$$
\{ M \in P_{\kappa}(\beta+1) : 
M \cap \kappa \in \kappa \cap \cof(>\! \omega) \};
$$
\item if $M \in X_p(\beta+1)$ and $\gamma \in M \cap \dom(a_p)$, 
then 
$$
p \restriction \gamma \Vdash_{\p_\gamma} M \cap \kappa \in A_{a_p(\gamma)}.
$$
\end{enumerate}
Let $q \le p$ in $\p_{\beta+1}$ if 
\begin{enumerate}
\item $q \restriction \beta \le p \restriction \beta$ in $\p_\beta$;
\item if $\beta \in \dom(a_p)$, then $\beta \in \dom(a_q)$ and 
$$
q \restriction \beta \Vdash_{\p_\beta} 
a_q(\beta) \le a_p(\beta) \ \textrm{in} \ \p(\dot T_\beta,\dot U_\beta);
$$
\item if $\beta+1 \in \dom(X_p)$, then 
$\beta+1 \in \dom(X_q)$ and $X_p(\beta+1) \subseteq X_q(\beta+1)$.
\end{enumerate}

\bigskip

Now assume that $\kappa < \alpha < \kappa^+$ is a limit ordinal, and 
$\p_\beta$ is defined for all $\beta < \alpha$. 
Define $\p_\alpha$ to be the set of all pairs $p = (a_p,X_p)$ satisfying:
\begin{enumerate}
\item $a_p$ is a function whose domain is a countable subset of $\alpha$;
\item $X_p$ is a function whose domain is a countable subset of $\alpha+1$;
\item for all $\beta < \alpha$, 
$p \restriction \beta := (a_p \restriction \beta, X_p \restriction (\beta+1))$ 
is in $\p_\beta$;
\item if $\alpha \in \dom(X_p)$, then $X_p(\alpha)$ is a countable subset of 
$$
\{ M \in P_{\kappa}(\alpha) : M \cap \kappa \in \kappa \cap \cof(>\! \omega) \};
$$
\item if $M \in X_p(\alpha)$ and $\gamma \in M \cap \dom(a_p)$, then 
$$
p \restriction \gamma \Vdash_{\p_\gamma} M \cap \kappa \in A_{a_p(\gamma)}.
$$
\end{enumerate}
Let $q \le p$ if for all $\beta < \alpha$, 
$q \restriction \beta \le p \restriction \beta$ in $\p_\beta$, 
and moreover, if $\alpha \in \dom(X_p)$, then $\alpha \in \dom(X_q)$ and 
$X_p(\alpha) \subseteq X_q(\alpha)$.

\bigskip

Finally, define $\p_{\kappa^+} := \bigcup \{ \p_\alpha : \alpha < \kappa^+ \}$, and let 
$q \le p$ in $\p_{\kappa^+}$ if for all large enough $\alpha < \kappa^+$, 
$q \le p$ in $\p_\alpha$.\footnote{Let $p = (a_p,X_p)$ be a condition in $\p_\alpha$. 
We will often find it convenient to write $a_p(\beta)$ or $X_p(\gamma)$ for $\beta < \alpha$ 
and $\gamma \le \alpha$ without necessarily knowing that $\beta \in \dom(a_p)$ or 
$\gamma \in \dom(X_p)$. In the case that they are not, 
by default $a_p(\beta)$ and $X_p(\gamma)$ will denote the empty set.}

\bigskip

We will denote the order on $\p_\alpha$ as $\le_\alpha$, for all $\alpha \le \kappa^+$.

The definition of the forcing iteration was by necessity an induction. 
The next lemma provides 
a useful non-inductive description of conditions in 
$\p_\alpha$. 
The proof is a straightforward induction on $\alpha$.

\begin{lemma}
Let $\alpha \le \kappa^+$. 
A pair $(a,X)$ is in $\p_\alpha$ iff:
\begin{enumerate}
\item $a$ is a function whose domain is a countable subset of $\alpha$;
\item $X$ is a function whose domain is a countable subset of $(\alpha+1) \cap \kappa^+$;
\item $(a \restriction \kappa,X \restriction (\kappa+1)) = 
(a \restriction 1,\emptyset) \in \p_1$;
\item for all nonzero $\beta \in \dom(a)$, $a(\beta)$ is a $\p_\beta$-name such that, 
assuming that 
$(a \restriction \beta,X \restriction (\beta+1))$ is in 
$\p_{\beta}$, then this condition forces in $\p_\beta$ that 
$a(\beta) \in \p(\dot T_\beta,\dot U_\beta)$;
\item for all $\gamma \in \dom(X)$, $X(\gamma)$ is a countable subset of 
$$
\{ M \in P_{\kappa}(\gamma) : M \cap \kappa \in \kappa \cap \cof(>\! \omega) \};
$$
\item for all $\gamma \in \dom(X)$, if $M \in X(\gamma)$ and 
$\beta \in M \cap \dom(a)$, then, assuming that 
$(a \restriction \beta,X \restriction (\beta+1))$ is in $\p_\beta$, this condition forces in 
$\p_\beta$ that $M \cap \kappa \in A_{a(\beta)}$.
\end{enumerate}
Also $(b,Y) \le_\alpha (a,X)$ iff
\begin{enumerate}
\item $\dom(a) \subseteq \dom(b)$ and $\dom(X) \subseteq \dom(Y)$;
\item for all $\gamma \in \dom(X)$, $X(\gamma) \subseteq Y(\gamma)$;
\item assuming $0 \in \dom(a)$, $a(0) \subseteq b(0)$;
\item for all nonzero $\beta \in \dom(a)$, 
$(b \restriction \beta,Y \restriction (\beta+1)) \Vdash_{\p_\beta} b(\beta) \le a(\beta)$ in 
$\p(\dot T_\beta,\dot U_\beta)$.
\end{enumerate} 
\end{lemma}

Since $2^\kappa = \kappa^+$ and we are assuming 
as an inductive hypothesis that each $\p_\alpha$ is $\kappa^+$-c.c., 
it is straightforward to check that 
$\p_\alpha \in H(\kappa^{++})$ 
for all $\alpha \le \kappa^+$.

The proofs of the next two lemmas are straightforward.

\begin{lemma}
Let $\beta < \alpha \le \kappa^+$.
\begin{enumerate}
\item $\p_\beta$ is a suborder of $\p_\alpha$;
\item $p \le_\alpha p \restriction \beta$ for all $p \in \p_\alpha$;
\item if $p$ and $q$ are in $\p_\beta$ and $r \le_\alpha p, q$, 
then $r \restriction \beta \le_\beta p, q$.
\end{enumerate}
\end{lemma}

\begin{lemma}
For all $\alpha \le \kappa^+$, the 
forcing poset $\p_\alpha$ is countably closed. 
In fact, given a descending sequence of conditions 
$\langle q_n : n < \omega \rangle$, let $r = (a_r,X_r)$ satisfy:
\begin{enumerate}
\item $\dom(a_r) := \bigcup \{ \dom(a_{q_n}) : n < \omega \}$;
\item $\dom(X_r) := \bigcup \{ \dom(X_{q_n}) : n < \omega \}$;
\item for all $\beta \in \dom(a_r)$, $a_r(\beta)$ is a $\p_\beta$-name for 
the greatest lower bound in $\p(\dot T_\beta,\dot U_\beta)$ of 
$\langle a_{q_n}(\beta) : n < \omega \rangle$ as described in Lemma 4.2;
\item for all $\gamma \in \dom(X_r)$, 
$X_r(\gamma) := \bigcup \{ X_{q_n}(\gamma) : n < \omega \}$.
\end{enumerate}
Then $r$ is the greatest lower bound of $\langle q_n : n < \omega \rangle$.
\end{lemma}

\begin{definition}
A condition $p \in \p_\alpha$ is said to be \emph{determined} 
if for all nonzero $\gamma \in \dom(a_p)$, there is a pair $(A,f)$ in the ground model 
such that $a_p(\gamma)$ is the canonical $\p_\gamma$-name for $(A,f)$.
\end{definition}

Note that if $p$ is determined in $\p_\alpha$, then for all 
$\beta < \alpha$, $p \restriction \beta$ is determined in $\p_\beta$.

\begin{lemma}
The set of determined conditions is dense in $\p_\alpha$. 
Also, if $\langle r_n : n < \omega \rangle$ is a descending sequence of 
determined conditions, then the greatest lower bound as described in Lemma 5.3 
is also determined.
\end{lemma}

\begin{proof}
By Lemma 5.3, $\p_\alpha$ does not add any new countable subsets of 
the ground model. 
Given a condition $p$, using a standard bookkeeping argument we can define by induction 
a descending sequence of conditions 
$\langle q_n : n < \omega \rangle$ below $p$ 
so that for all $n < \omega$ 
and all nonzero $\gamma \in \dom(a_{q_n})$, 
there is $m \ge n$ such that $q_m \restriction \gamma$ 
decides $a_{q_n}(\gamma)$ as $(A_{\gamma,n},f_{\gamma,n})$. 
For each $\gamma \in \bigcup \{ \dom(a_{q_n}) : n < \omega \}$, 
let $A_\gamma := \bigcup \{ A_{\gamma,n} : n < \omega \}$ and 
$f_\gamma := \bigcup \{ f_{\gamma,n} : n < \omega \}$. 

Let $r$ be the greatest lower bound of $\langle q_n : n < \omega \rangle$. 
Define $r'$ by letting $X_{r'} := X_r$, $\dom(a_{r'}) := \dom(a_r)$, 
$a_{r'}(0) := a_r(0)$, 
and for all nonzero $\gamma \in \dom(a_r)$, 
letting $a_{r'}(\gamma)$ be a canonical $\p_\gamma$-name for the pair 
$(A_\gamma,f_\gamma)$. 
Then easily $r \le_\alpha r' \le_\alpha r$ and $r'$ is determined. 
The second statement follows by a similar but easier argument.
\end{proof}

We will informally identify a determined condition $p$ with the 
object $(a,X)$ such that $\dom(a) = \dom(a_p)$, 
$X = X_p$, $a(0) = a_p(0)$, and for all nonzero $\gamma \in \dom(a)$, 
$a(\gamma) = (A,f)$, where $a_p(\gamma)$ is the canonical 
$\p_\gamma$-name for $(A,f)$.

Note that for all nonzero $\alpha < \kappa^+$, there are 
$\kappa$ many determined conditions in $\p_\alpha$. 
Thus, $\p_\alpha$ has a dense subset of size $\kappa$ and hence is $\kappa^+$-c.c.

\begin{lemma}
For all $\alpha \le \kappa^+$, 
$\p_\alpha$ is $\kappa^+$-c.c.
\end{lemma}

\begin{proof}
By the preceding comments, it suffices to prove the statement for $\alpha = \kappa^+$. 
Suppose that $\langle p_i : i < \kappa^+ \rangle$ is 
a sequence of conditions in $\p_{\kappa^+}$. 
By extending these conditions if necessary, we may assume without 
loss of generality that each $p_i$ is determined.

A straightforward argument using the $\Delta$-system lemma shows that 
there is $Z \subseteq \kappa^+$ of size $\kappa^+$ and  
$\kappa < \beta < \kappa^+$ such that for all $i < j$ in $Z$:
\begin{enumerate}
\item $\dom(a_{p_i}) \cap \dom(a_{p_j}) \subseteq \beta$;
\item $\dom(X_{p_i}) \cap \dom(X_{p_j}) \subseteq \beta$;
\item $p_i \restriction \beta = p_j \restriction \beta$;
\item the intersection of the sets 
$$
Y_i := (\bigcup \bigcup \ran(X_{p_i})) \cup \dom(a_{p_i})
$$
and 
$$
Y_j := (\bigcup \bigcup \ran(X_{p_j})) \cup \dom(a_{p_j})
$$
is a subset of $\beta$.
\end{enumerate}

Let $i < j$ in $Z$, and we will prove that $p_i$ and $p_j$ are compatible. 
Define $q = (a,X)$ as follows:
\begin{enumerate}
\item[(a)] $\dom(a) := \dom(a_{p_i}) \cup \dom(a_{p_j})$;
\item[(b)] $\dom(X) := \dom(X_{p_i}) \cup \dom(X_{p_j})$;
\item[(c)] $a(\gamma) := a_{p_i}(\gamma)$ when $\gamma \in \dom(a_{p_i})$;
\item[(d)] $a(\gamma) := a_{p_j}(\gamma)$ when $\gamma \in \dom(a_{p_j})$;
\item[(e)] $X(\gamma) := X_{p_i}(\gamma)$ when $\gamma \in \dom(X_{p_i})$;
\item[(f)] $X(\gamma) := X_{p_j}(\gamma)$ when $\gamma \in \dom(X_{p_j})$.
\end{enumerate}
Note that (1), (2), and (3) above imply that $q$ is well-defined. 
Also, if $q$ is a condition, then easily $q \le_{\kappa^+} p_i, p_j$.

We prove that $q$ is a condition using Lemma 5.1. 
Note that for all $\xi < \kappa^+$, 
if $(a \restriction \xi,X \restriction (\xi+1))$ is in $\p_\xi$ then it 
is clearly below $p_i \restriction \xi$ and $p_j \restriction \xi$. 
Properties (1)--(5) of Lemma 5.1 are thereby trivial. 
So it remains to show that for all $\gamma \in \dom(X)$, $M \in X(\gamma)$, 
and $\xi \in M \cap \dom(a)$, assuming that 
$(a \restriction \xi,X \restriction (\xi+1))$ 
is a condition in $\p_\xi$, then this condition forces in $\p_\xi$ that 
$M \cap \kappa \in A_{a(\xi)}$. 

This statement is trivial if $M \in X_{p_i}(\gamma)$ 
and $\xi \in M \cap \dom(a_{p_i})$, or if 
$M \in X_{p_j}(\gamma)$ and $\xi \in M \cap \dom(a_{p_j})$, since 
$(a \restriction \xi, X \restriction (\xi+1))$ is below $p_i \restriction \xi$ 
and $p_j \restriction \xi$ assuming it is a condition. 
In particular, it is true if $\gamma < \beta$. 
Assume that $\gamma \ge \beta$. 
Then $\gamma$ is either in $\dom(X_{p_i})$ or $\dom(X_{p_j})$, but not both. 
Without loss of generality, assume that $\gamma \in \dom(X_{p_i})$. 
Then $M \in X_{p_i}(\gamma)$.

If $\xi \in \dom(a_{p_i})$ then we are done, so assume that 
$\xi \in \dom(a_{p_j}) \setminus \dom(a_{p_i})$. 
Since $\xi \in M$, we have that $\xi \in Y_i$, and since 
$\xi \in \dom(a_{p_j})$, we have that $\xi \in Y_j$. 
So $\xi \in Y_i \cap Y_j \subseteq \beta$. 
Hence, $\xi \in \dom(a_{p_j}) \cap \beta$. 
Since $p_i \restriction \beta = p_j \restriction \beta$, 
$a_{p_i} \restriction \beta = a_{p_j} \restriction \beta$. 
So $\dom(a_{p_i}) \cap \beta = \dom(a_{p_j}) \cap \beta$. 
It follows that $\beta \in \dom(a_{p_i})$, which is a contradiction.
\end{proof}

\begin{definition}
Let $\kappa < \beta < \alpha \le \kappa^+$. 
Define $D_{\beta,\alpha}$ as the set of $p \in \p_\alpha$ 
such that for all $\gamma \in \dom(X_p) \setminus \beta$, 
$$
\{ M \cap \beta : M \in X_p(\gamma) \} \subseteq X_p(\beta).
$$
\end{definition}

The next lemma is easy to prove.

\begin{lemma}
Let $\kappa < \beta < \alpha \le \kappa^+$. 
Then $D_{\beta,\alpha}$ is dense in $\p_\alpha$. 
In fact, let $p \in \p_\alpha$ and $x \subseteq \alpha \setminus (\kappa+1)$ be countable. 
Define $q$ by letting $a_q := a_p$, $\dom(X_q) := \dom(X_p) \cup x$, 
for all $\gamma \in \dom(X_p) \setminus x$, $X_q(\gamma) := X_p(\gamma)$, 
and for all $\gamma \in x$, 
$$
X_q(\gamma) := X_p(\gamma) \cup \{ M \cap \gamma : 
\exists \xi \in \dom(X_p) \setminus \gamma, \ M \in X_p(\xi) \}.
$$
Then $q \in \p_\alpha$, $q \le_\alpha p$, and 
for all $\beta \in x$, $q \in D_{\beta,\alpha}$.
\end{lemma}

\begin{notation}
Let $\beta < \alpha \le \kappa^+$ and $p \in \p_\alpha$. 
Let $r \le_\beta p \restriction \beta$. 
Define $r +_{\beta,\alpha} p$ to be the pair $(a,X)$ satisfying:
\begin{enumerate}
\item $a \restriction \beta = a_r$ 
and $a \restriction [\beta,\alpha) = a_p \restriction [\beta,\alpha)$;
\item $X \restriction (\beta+1) = X_r$ and 
$X \restriction (\beta+1,\alpha] = X_p \restriction (\beta+1,\alpha]$.
\end{enumerate}
\end{notation}

When $\beta$ and $\alpha$ are understood from context, we will abbreviate 
$r +_{\beta,\alpha} p$ as $r + p$.

\begin{lemma}
Let $\beta < \alpha \le \kappa^+$, $p \in \p_\alpha$, and 
$r \le_\beta p \restriction \beta$. 
Suppose that $p \in D_{\beta,\alpha}$ in the case that $\kappa < \beta$. 
Then $r + p$ is in $\p_\alpha$ and is below $p$ and $r$ in $\p_\alpha$.
\end{lemma}

\begin{proof}
We use Lemma 5.1 to show that $r + p$ is a condition. 
It is then easy to see that $r+p$ is below $p$ and $r$. 
In fact, for all $\xi \le \alpha$, assuming that 
$(r + p) \restriction \xi := 
(a_{r+p} \restriction \xi, X_{r+p} \restriction (\xi + 1))$ is in $\p_\xi$, it 
is easily seen to be below $p \restriction \xi$ and $r \restriction \xi$ in $\p_\xi$.

The only nontrivial property to check from Lemma 5.1 is (6). 
Suppose that $\gamma \in \dom(X_{r+p})$, 
$M \in X_{r+p}(\gamma)$, and $\xi \in M \cap \dom(a_{r+p})$ is nonzero.  
Assume that $(r + p) \restriction \xi := (a_{r+p} \restriction \xi,X_{r+p} \restriction (\xi+1))$ is a 
condition in $\p_\xi$ and is below $p \restriction \xi$ and $r \restriction \xi$. 
We claim that this condition forces in $\p_\xi$ that 
$M \cap \kappa \in A_{a_{r+p}(\xi)}$.

Assume first that $\gamma \le \beta$. 
Then $\xi < \beta$. 
And $X_{r+p}(\gamma) = X_r(\gamma)$ and $a_{r+p} \restriction \gamma = a_r \restriction \gamma$. 
So $M \in X_r(\gamma)$ and $\xi \in M \cap \dom(a_r)$. 
Since $r$ is a condition, it follows that 
$r \restriction \xi = (a_{r+p} \restriction \xi, X_{r+p} \restriction (\xi +1))$ forces that
$M \cap \kappa \in A_{a_{r}(\xi)} = A_{a_{r+p}(\xi)}$. 
In the second case that $\xi \ge \beta$, 
a similar argument works using the fact that $p$ is a condition. 

Thirdly, assume that $\xi < \beta < \gamma$. 
Note that since $\xi \ne 0$ and $\dom(a_p) \cap \kappa \subseteq \{ 0 \}$ by the definition 
of the forcing poset $\p_\alpha$, $\kappa \le \xi < \beta$. 
Since $p \in D_{\beta,\alpha}$, $r \le_\beta p \restriction \beta$, 
and $M \in X_{r+p}(\gamma) = X_p(\gamma)$, it follows that 
$M \cap \beta \in X_p(\beta) \subseteq X_r(\beta)$. 
Also, $\xi \in \dom(a_{r+p}) \cap \beta = \dom(a_r) \cap \beta$. 
As $r$ is a condition, $M \cap \beta \in X_r(\beta)$, 
and $\xi \in \dom(a_r) \cap (M \cap \beta)$, it follows that 
$r \restriction \xi = (r+p) \restriction \xi$ forces that 
$M \cap \kappa \in A_{a_r(\xi)} = A_{a_{r+p}(\xi)}$.
\end{proof}

\begin{proposition}
Let $\beta < \alpha \le \kappa^+$. 
Then $\p_\beta$ is a regular suborder of $\p_\alpha$.
\end{proposition}

\begin{proof}
By Lemma 5.2(1,3), $\p_\beta$ is a suborder of $\p_\alpha$, and 
if $p$ and $q$ are in $\p_\beta$ and are compatible in $\p_\alpha$, 
then they are compatible in $\p_\beta$.

It remains to show that if $A$ is a maximal antichain of $\p_\beta$, 
then $A$ is predense in $\p_\alpha$. 
So let $p \in \p_\alpha$, and we will find a member of $A$ which is 
compatible with $p$ in $\p_\alpha$. 
In the case that $\kappa < \beta$, extend $p$ to $q$ in $D_{\beta,\alpha}$ 
by Lemma 5.8, and otherwise let $q = p$.

Since $q \restriction \beta$ is in $\p_\beta$ and $A$ is a maximal 
antichain of $\p_\beta$, we can fix $r \in A$ and $s \in \p_\beta$ 
such that $s \le_\beta q \restriction \beta, r$. 
By Lemma 5.10, $s + q$ is in $\p_\alpha$ and is below 
$q$ and $s$ in $\p_\alpha$. 
So it is also below $q$ and $r$. 
Thus, $p$ is compatible with a member of $A$.
\end{proof}

In Section 7 we will prove that $\p_\alpha$ preserves $\kappa$, for all 
$\alpha < \kappa^+$, assuming that $\kappa$ is ineffable. 
The next result shows that this implies that $\p_{\kappa^+}$ 
preserves $\kappa^+$.

\begin{proposition}
Assume that for all $\alpha < \kappa^+$, $\p_\alpha$ preserves $\kappa$. 
Then $\p_{\kappa^+}$ preserves $\kappa$, and in particular, 
$\p_{\kappa^+}$ forces that $\kappa = \omega_2$.
\end{proposition}

\begin{proof}
As $\p_{\kappa^+}$ is countably closed by Lemma 5.3, 
it preserves $\omega_1$. 
Since $\p_1$ is a regular suborder of $\p_{\kappa^+}$ by Proposition 5.11 
and $\p_1$ is forcing 
equivalent to 
the forcing $\coll(\omega_1,<\! \kappa)$, 
the forcing poset $\p_{\kappa^+}$ collapses all 
uncountable cardinals less than $\kappa$ to have size $\omega_1$. 
Hence, to show that $\p_{\kappa^+}$ forces that $\kappa = \omega_2$ 
it suffices to show that it preserves $\kappa$.

Suppose that $\dot f$ is a nice $\p_{\kappa^+}$-name for a function from 
$\omega_1$ to $\kappa$. 
Then there is a sequence $\langle A_x : x \in \omega_1 \times \kappa \rangle$ 
of antichains in $\p_{\kappa^+}$ 
so that 
$$
\dot f = \{ (p,\check x) : x \in \omega_1 \times \kappa, \ 
p \in A_x \}.
$$ 
Since $\p_{\kappa^+}$ is $\kappa^+$-c.c.\! by Lemma 5.6, for each 
$x$ there is $\beta_x < \kappa^+$ such that $A_x \subseteq \p_{\beta_x}$.

Let $\beta < \kappa^+$ be larger than $\beta_x$ for all $x \in \omega_1 \times \kappa$. 
Then $\dot f$ is a $\p_\beta$-name, and for any generic filter $G$ on $\p_{\kappa^+}$, 
$\dot f^{G} = \dot f^{G \cap \p_\beta}$. 
Now $\p_\beta$ is a regular suborder of $\p_{\kappa^+}$ and $\p_\beta$ preserves $\kappa$. 
So if $G$ is a generic filter on $\p_{\kappa^+}$, then 
$\dot f^G = \dot f^{G \cap \p_\beta} \in V[G \cap \p_\beta]$ 
is not surjective onto $\kappa$, and hence it is not surjective onto $\kappa$ in $V[G]$.
\end{proof}

We now show that the forcing iteration does what it was intended to do, namely, to add 
club isomorphisms between Aronszajn trees. 
This follows from the next lemma and Proposition 4.8.

\begin{lemma}
Let $\kappa \le \beta < \kappa^+$. 
Then the forcing posets $\p_{\beta+1}$ and 
$\p_\beta * \p(\dot T_\beta,\dot U_\beta)$ are forcing equivalent.
\end{lemma}

\begin{proof}
Let $D$ be the set of $q \in \p_{\beta+1}$ such that $\beta \in \dom(a_q)$, and 
$q \in D_{\beta,\beta+1}$ if $\kappa < \beta$. 
We claim that $D$ is dense in $\p_{\beta+1}$. 
Given $p \in \p_{\beta+1}$, extend $p$ to $p_1$ such that $\beta \in \dom(a_{p_1})$. 
Namely, if $\beta \in \dom(a_p)$ already, let $p = p_1$; otherwise, 
add $\beta$ and define $a_{p_1}(\beta) = (A,\emptyset)$, where 
$A = \{ M \cap \kappa : M \in X_p(\beta+1) \}$. 
Now apply Lemma 5.8 and extend $p_1$ to $q$ in $D_{\beta,\beta+1}$, if $\kappa < \beta$.

Define $\pi : D \to \p_{\beta} * \p(\dot T_\beta,\dot U_\beta)$ 
by letting $\pi(p) = (p \restriction \beta) * a_p(\beta)$ for all $p \in D$. 
We claim that $\pi$ is a dense embedding.

It is immediate from the definitions 
that if $q \le_{\beta+1} p$ in $D$, then $\pi(q) \le \pi(p)$ 
in $\p_{\beta} * \p(\dot T_\beta,\dot U_\beta)$. 
We claim that $\pi$ is surjective, and in particular, the range of $\pi$ is dense. 
So let $s = (s \restriction \beta,s(\beta)) \in \p_\beta * \p(\dot T_\beta,\dot U_\beta)$. 
Define $p$ by letting $p \restriction \beta := s \restriction \beta$, 
$a_p(\beta) := s(\beta)$, and $X_p := X_{s \restriction \beta}$. 
It is easy to check that $p$ is a condition and $\pi(p) = s$.

Let $p$ and $q$ be conditions in $D$, and suppose that 
$\pi(p)$ and $\pi(q)$ are compatible in 
$\p_\beta * \p(\dot T_\beta,\dot U_\beta)$. 
We will show that $p$ and $q$ are compatible in $\p_{\beta+1}$. 
Fix 
$$
r = (r \restriction \beta,r(\beta)) \le \pi(p), \pi(q)
$$
in $\p_\beta * \p(\dot T_\beta,\dot U_\beta)$. 
Then in particular, 
$$
r \restriction \beta \le_\beta \pi(p) \restriction \beta = p \restriction \beta,
$$
$$ 
r \restriction \beta \le_\beta \pi(q) \restriction \beta = q \restriction \beta,
$$
and 
$r \restriction \beta$ forces in $\p_\beta$ that 
$r(\beta) \le \pi(p)(\beta) = a_p(\beta)$ and $r(\beta) \le \pi(q)(\beta) = a_q(\beta)$ 
in $\p(\dot T_\beta,\dot U_\beta)$.

Define $s \in \p_{\beta+1}$ as follows. 
Let $s \restriction \beta := r \restriction \beta$, 
$a_s(\beta) := r(\beta)$, $X_s \restriction \beta := X_{r \restriction \beta}$, 
and $X_s(\beta+1) := X_p(\beta+1) \cup X_q(\beta+1)$. 
It is straightforward to check that $s$ satisfy properties (1)--(5) of being a condition 
in $\p_{\beta+1}$, and that if $s$ is a condition, then 
$s \le_{\beta+1} p, q$. 
It remains to prove property (6). 

Assume that $M \in X_s(\beta+1)$ 
and $\gamma \in M \cap \dom(a_s)$. 
We will show that $s \restriction \gamma$ forces in $\p_\gamma$ 
that $M \cap \kappa \in A_{a_s(\gamma)}$. 
Since $X_s(\beta+1) = X_p(\beta+1) \cup X_q(\beta+1)$, 
either $M \in X_p(\beta+1)$ or $M \in X_q(\beta+1)$. 
By symmetry, it suffices to consider the case that $M \in X_p(\beta+1)$.

First, assume that $\gamma < \beta$. 
Then $a_s(\gamma) = a_r(\gamma)$ and 
$\gamma \in \dom(a_r) \cap (M \cap \beta)$. 
Since $p \in D_{\beta,\beta+1}$, $M \cap \beta \in X_p(\beta)$. 
As $r \restriction \beta \le_\beta p \restriction \beta$, $M \cap \beta \in X_r(\beta)$. 
Since $r$ is a condition, $r \restriction \gamma = s \restriction \gamma$ forces that 
$M \cap \kappa \in A_{a_r(\gamma)} = A_{a_s(\gamma)}$.

Secondly, assume that $\gamma = \beta$. 
Then $p \restriction \beta$ forces that $M \cap \kappa \in A_{a_p(\beta)}$. 
Since $r \restriction \beta = s \restriction \beta$ forces that 
$r(\beta) \le a_p(\beta)$ in $\p(\dot T_\beta,\dot U_\beta)$, it forces that 
$M \cap \kappa \in A_{r(\beta)} = A_{a_s(\beta)}$. 
\end{proof}

Let us show how the material in this section can be used to prove the main 
result of the paper. 
Start with a model in which $\kappa$ is ineffable and $2^\kappa = \kappa^+$. 
Assume, moreover, that the iterations defined in this section preserve $\kappa$, which 
will be verified in Section 7. 
Since $\p_{\kappa^+}$ is $\kappa^+$-c.c.\!, a standard nice name 
argument similar to the proof 
of Proposition 5.12 shows that any $\omega_2$-tree in a generic extension 
by $\p_{\kappa^+}$ appears in an intermediate generic extension by 
$\p_\beta$ for some $\beta < \kappa^+$. 
As $2^\kappa = \kappa^+$, standard arguments show that we can arrange 
our bookkeeping function to enumerate all nice names for standard countably closed 
normal $\omega_2$-Aronszajn trees in such a way that the iteration handles all possible 
pairs of such trees.

It follows that in a generic extension $V[G]$ by $\p_{\kappa^+}$, 
whenever $T$ and $U$ are standard countably closed normal 
$\omega_2$-Aronszajn trees such that for all $\gamma < \kappa$,  
the nodes of height $\gamma$ in $T$  
are exactly the 
even ordinals in $I_\gamma$, and the nodes of height $\gamma$ in $U$ are exactly 
the odd ordinals in $I_\gamma$, 
then for some $\beta < \kappa^+$, 
$\dot T_\beta^{G \cap \p_\beta} = T$ and $\dot U_\beta^{G \cap \p_\beta} = U$. 
But $\p_{\beta+1}$ is a regular suborder of $\p_{\kappa^+}$ by Proposition 5.11, 
and $\p_{\beta+1}$ is forcing equivalent to 
$\p_\beta * \p(\dot T_\beta,\dot U_\beta)$ by Lemma 5.13. 
It follows by Proposition 4.8 that in $V[G]$ there is a club isomorphism from $T$ to $U$.

By easy arguments, any countably closed normal $\omega_2$-Aronszajn tree 
in $V[G]$ is club isomorphic to trees of the type described in the previous paragraph. 
Hence, in $V[G]$ we have that any two countably closed normal 
$\omega_2$-Aronszajn trees are club isomorphic.

To complete the main result of the paper, it remains to show that the forcing posets defined 
in this section preserve $\kappa$, assuming that $\kappa$ is ineffable. 
The last two sections of the paper are devoted to proving this.

\section{Preparation for the preservation of $\kappa$}

For the remainder of the paper, we will assume that $\kappa$ is ineffable. 
Let $J$ denote the ineffability ideal on $\kappa$ as discussed in Section 2. 
Let us also identify a useful collection of models.

\begin{definition}
Let $\mathcal Y$ denote the collection of all sets $N$ in 
$P_{\kappa}(H(\kappa^{++}))$ satisfying:
\begin{enumerate}
\item $N \prec (H(\kappa^{++}),\in,\unlhd,J)$, where $\unlhd$ is a fixed 
well-ordering of $H(\kappa^{++})$;
\item $N \cap \kappa$ is inaccessible, $| N | = N \cap \kappa$, 
and $N^{<\! (N \cap \kappa)} \subseteq N$.
\end{enumerate}
\end{definition}

The next lemma uses the standard fact that the set of inaccessible cardinals 
in $\kappa$ is a member of $J^*$; see \cite[Proposition 2.5]{ineffable}.

\begin{lemma}
The set $\mathcal Y$ is stationary in $P_{\kappa}(H(\kappa^{++}))$.
\end{lemma}

\begin{proof}
Let $F : H(\kappa^{++})^{< \omega} \to H(\kappa^{++})$. 
Build a $\in$-increasing and 
continuous sequence $\langle N_i : i < \kappa \rangle$ of sets in 
$P_{\kappa}(H(\kappa^{++}))$ which are closed under $F$ and 
are elementary substructures of $(H(\kappa^{++}),\in,\unlhd,J)$. 
Since $\kappa$ is inaccessible, a standard argument shows that 
there is a club $C \subseteq \kappa$ such that for all 
$\alpha \in C$, $|N_\alpha| = \alpha = N_\alpha \cap \kappa$.  
As $\kappa$ is ineffable, it is Mahlo, so we can find an inaccessible $\alpha$ in $C$. 
Since $\alpha$ is inaccessible, it is easy to check that 
$N_\alpha^{< \alpha} \subseteq N_\alpha$. 
Then $N_\alpha$ is in $\mathcal Y$ and is closed under $F$.
\end{proof}

\begin{notation}
Let $\alpha < \kappa^+$. 
A set $N \in \mathcal Y$ is said to be \emph{$\alpha$-suitable} if $N$ is an elementary 
substructure of 
$$
( H(\kappa^{++}), \in, \unlhd, J, \mathcal Y, \langle \p_\beta : \beta \le \alpha \rangle ).
$$
\end{notation}

Note that if $N$ is $\alpha$-suitable, then for all $\beta \in N \cap \alpha$, 
$N$ is $\beta$-suitable.

\begin{lemma}
The collection of $N \in P_{\kappa}(H(\kappa^{++}))$ such that 
$N$ is $\alpha$-suitable is stationary.
\end{lemma}

\begin{proof}
Immediate from Lemma 6.2 and the definition of $\alpha$-suitable.
\end{proof}

\bigskip

Fix $\kappa < \alpha < \kappa^+$ for the remainder of the paper, and we will prove 
that $\p_\alpha$ preserves $\kappa$. 
It will then follow by Proposition 5.12 that $\p_{\kappa^+}$ preserves 
$\kappa$, which will complete the proof of the main result of the paper. 
We assume as an inductive hypothesis that for all $\beta < \alpha$, 
$\p_\beta$ preserves $\kappa$. 
We will identify two additional inductive hypotheses in Section 7. 

\begin{notation}
Let $N$ be $\alpha$-suitable. 
Define $p(N,\alpha) := (\emptyset,X)$, where 
$\dom(X) := \{ \alpha \}$ and $X(\alpha) := \{ N \cap \alpha \}$.
\end{notation}

Easily, $p(N,\alpha)$ is a condition in $\p_\alpha$. 
Note that if $p \in \p_\alpha$, then $p \le_\alpha p(N,\alpha)$ iff 
$N \cap \alpha \in X_p(\alpha)$.

\begin{lemma}
Let $N$ be $\alpha$-suitable, $p \le_\alpha p(N,\alpha)$, $\beta \in N \cap \alpha$, 
and $p \in D_{\beta,\alpha}$. 
Then $p \restriction \beta \le_\beta p(N,\beta)$. 
\end{lemma}

\begin{proof}
Since $p \le_\alpha p(N,\alpha)$, $N \cap \alpha \in X_p(\alpha)$. 
As $p \in D_{\beta,\alpha}$, 
$(N \cap \alpha) \cap \beta = N \cap \beta \in X_p(\beta) = X_{p \restriction \beta}(\beta)$.
\end{proof}

\begin{lemma}
Let $N$ be $\alpha$-suitable. 
Then for all $p \in N \cap \p_\alpha$, $p$ and $p(N,\alpha)$ are compatible 
in $\p_\alpha$.
\end{lemma}

\begin{proof}
By extending $p$ in $N$ if necessary, we may assume without loss of 
generality that $p$ is determined. 
Note that by elementarity, $\dom(a_p) \subseteq N$ and for all 
$\beta \in \dom(a_p)$, $A_{a_p(\beta)} \subseteq N \cap \kappa$. 
Also, $\dom(X_p) \subseteq N$ and for all $\gamma \in \dom(X_p)$, 
$X_p(\gamma) \subseteq N$.

Define $q = (b,Y)$ as follows. 
Let $\dom(b) := \dom(a_p)$. 
Define $b(0) := a_p(0)$. 
For all nonzero $\beta \in \dom(a_p)$, apply Lemma 4.5 to find a nice $\p_\beta$-name 
$b(\beta)$ for an extension of $a_p(\beta)$ in $\p(\dot T_\beta,\dot U_\beta)$ 
such that $N \cap \kappa \in A_{b(\beta)}$. 
Let $\dom(Y) := \dom(X_p) \cup \{ \alpha \}$, 
$Y(\beta) := X_p(\beta)$ for all $\beta \in \dom(X_p)$ different from $\alpha$, 
and $Y(\alpha) := X_p(\alpha) \cup \{ N \cap \alpha \}$. 
It is straightforward to check using Lemma 5.1 that $q$ is a condition, 
and clearly $q \le_\alpha p, p(N,\alpha)$.
\end{proof}

\begin{definition}
Let $N$ be $\alpha$-suitable. 
Define $D(N,\alpha)$ as the set of determined conditions $p \in \p_\alpha$ 
satisfying:
\begin{enumerate}
\item $p \le_\alpha p(N,\alpha)$;
\item for all $\gamma \in \dom(X_p)$, for all $M \in X_p(\gamma)$, 
if $M \cap \kappa < N \cap \kappa$ then $M \cap N \in X_p(\gamma)$;
\item for all nonzero $\beta \in N \cap \dom(a_p)$, $p \restriction \beta$ forces 
that $a_p(\beta)$ is injective on $N \cap \kappa$ (in the sense of Definition 4.9).
\end{enumerate}
\end{definition}

\begin{lemma}
Let $N$ be $\alpha$-suitable. 
Then the set $D(N,\alpha)$ is dense below $p(N,\alpha)$.
\end{lemma}

\begin{proof}
Let $q \le_\alpha p(N,\alpha)$, and we will find $r \le_\alpha q$ in $D(N,\alpha)$. 
Without loss of generality, assume that $q$ is determined. 

We define by induction a descending sequence of conditions 
$\langle q_n : n < \omega \rangle$ in $\p_\alpha$. 
Let $q_0 := q$. 
Fix $n < \omega$, and assume that $q_n$ is defined. 
If $n$ is odd, then let $q_{n+1}$ be an extension of $q_n$ which is determined.

Assume that $n$ is even and $q_n$ is determined. 
Since $N \cap \alpha \in X_{q_n}(\alpha)$, we have that for all 
$\beta \in N \cap \dom(a_{q_n})$, $N \cap \kappa \in A_{a_{q_n}(\beta)}$. 
Define $q_{n+1}$ as follows. 
Let $X_{q_{n+1}} := X_{q_n}$ and 
$\dom(a_{q_{n+1}}) := \dom(a_{q_n})$. 
For each $\beta \in \dom(a_{q_n}) \setminus N$, 
let $a_{q_{n+1}}(\beta) := a_{q_n}(\beta)$. 
For each $\beta \in N \cap \dom(a_{q_n})$, apply Lemma 4.10 to find a 
$\p_\beta$-name $a_{q_{n+1}}(\beta)$ 
for a condition in $\p(\dot T_\beta,\dot U_\beta)$ such that 
$q_n \restriction \beta$ forces that $a_{q_{n+1}}(\beta)$ is an 
extension of $a_{q_n}(\beta)$ which is injective on $N \cap \kappa$.

This completes the construction. 
Let $r'$ be the greatest lower bound of 
the sequence 
$\langle q_n : n < \omega \rangle$. 
By Lemma 5.5, $r$ is determined. 
By construction, for all $\beta \in N \cap \dom(a_{r'})$, for all large enough 
even $n < \omega$, $q_{n} \restriction \beta$, and hence 
$r' \restriction \beta$, forces that 
$a_{q_{n+1}}(\beta)$ is injective on $N \cap \kappa$. 
It easily follows by Lemma 4.11 
that for all for all $\beta \in N \cap \dom(a_{r'})$, 
$r' \restriction \beta$ forces 
that $a_{r'}(\beta)$ is injective on $N \cap \kappa$.

Now define $r$ as follows. 
Let $a_r := a_{r'}$. 
Define $\dom(X_r) := \dom(X_{r'})$, and for all 
$\gamma \in \dom(X_r)$, 
$$
X_r(\gamma) := X_{r'}(\gamma) \cup \{ M \cap N : M \in X_{r'}(\gamma), \ 
M \cap \kappa < N \cap \kappa \}.
$$
Using the fact that whenever $M \in X_{r'}(\gamma)$ and $M \cap \kappa < N \cap \kappa$, 
it follows that $M \cap N \cap \kappa = M \cap \kappa$, it is 
straightforward to check that $r \in \p_\alpha$. 
Also clearly $r \le_\alpha r'$ and $r \in D(N,\alpha)$.
\end{proof}

\begin{lemma}
Let $N$ be $\alpha$-suitable and $x$ a countable subset of 
$(N \cap \alpha) \setminus (\kappa+1)$. 
Then for any $p \le_\alpha p(N,\alpha)$, there is 
$q \le_\alpha p$ such that $q \in D(N,\alpha)$ and for all $\beta \in x$, 
$q \in D_{\beta,\alpha}$.
\end{lemma}

\begin{proof}
By extending $p$ if necessary using Lemma 6.9, we may assume without loss of 
generality that $p \in D(N,\alpha)$. 
Now define $q$ from $p$ and $x$ as described in the statement of Lemma 5.8. 
Then $q \le p$ and $q \in D_{\beta,\alpha}$ for all $\beta \in x$. 
Since $a_p = a_q$, it is routine to check that $q \in D(N,\alpha)$.
\end{proof}

\begin{lemma}
Let $M$ and $N$ be $\alpha$-suitable with $M \cap \kappa = N \cap \kappa$. 
Then $p(M,\alpha) = p(N,\alpha)$ and $D(M,\alpha) = D(N,\alpha)$.
\end{lemma}

\begin{proof}
Since $\alpha \in M \cap N$, $|\alpha| \le \kappa$, and 
$M \cap \kappa = N \cap \kappa$, it follows by standard arguments 
that $M \cap \alpha = N \cap \alpha$. 
Now note that the definitions of $p(K,\alpha)$ and $D(K,\alpha)$ given in 
Notation 6.5 and Definition 6.8 
depend only on $K \cap \alpha$, for any $\alpha$-suitable $K$.
\end{proof}

The proofs of the next two lemmas are straightforward.

\begin{lemma}
Let $N$ be $\alpha$-suitable, $\beta \in N \cap \alpha$, and 
$p \in D(N,\alpha) \cap D_{\beta,\alpha}$. Then:
\begin{enumerate}
\item $p \restriction \beta \in D(N,\beta)$.
\item If $u \le_\beta p \restriction \beta$ and 
$u \in D(N,\beta)$, then $u + p \in D(N,\alpha)$.
\end{enumerate}
\end{lemma}

\begin{lemma}
Let $N$ be $\alpha$-suitable. 
Suppose that $\langle p_n : n < \omega \rangle$ 
is a descending sequence of conditions 
in $D(N,\alpha)$ with greatest lower bound $q$. 
Then $q \in D(N,\alpha)$.
\end{lemma}

\begin{definition}
Let $N$ be $\alpha$-suitable and $p$ be determined. 
Define $p \restriction N$ to be the pair $(a,X)$ satisfying:
\begin{enumerate}
\item $a$ is a function with domain equal to $\dom(a_p) \cap N$;
\item $a(0) = a_p(0) \restriction (\omega_1 \times (N \cap \kappa))$; 
\item for all nonzero $\gamma \in \dom(a)$, 
$a(\gamma) = (A_{a_p(\gamma)} \cap (N \cap \kappa), 
f_{a_p(\gamma)} \restriction (N \cap \kappa))$;
\item $X$ is a function with domain equal to $\dom(X_p) \cap N$, 
and for all $\gamma \in \dom(X)$, 
$X(\gamma) = X_p(\gamma) \cap N$.
\end{enumerate}
\end{definition}

Observe that by Lemma 4.4, for all nonzero $\gamma \in \dom(a)$, 
$p \restriction \gamma$ forces that 
$a(\gamma) = a_p(\gamma) \restriction (N \cap \kappa)$ 
is in $\p(\dot T_\gamma,\dot U_\gamma)$. 
The object $p \restriction N$ is a member of $N$ by the closure of $N$, 
but it is not 
necessarily a condition. 
If it is a condition, then it is determined and $p \le_\alpha p \restriction N$. 
Usually we only consider $p \restriction N$ in the case that 
$p \le_\alpha p(N,\alpha)$.

The proofs of the next two lemmas are straightforward.

\begin{lemma}
Let $N$ be $\alpha$-suitable, $\beta \in N \cap \alpha$, 
and $p \in \p_\alpha$ determined. 
Then $(p \restriction \beta) \restriction N = 
(p \restriction N) \restriction \beta$. 
In particular, if $p \restriction N \in \p_\alpha$, then for all 
$\beta \in N \cap \alpha$, $(p \restriction \beta) \restriction N \in \p_\beta$.
\end{lemma}

By $(p \restriction N) \restriction \beta$, we mean $(a \restriction \beta,X \restriction (\beta+1))$, 
where $p \restriction N = (a,X)$. 
Of course if $p \restriction N$ is a condition, then this is the same as 
$(p \restriction N) \restriction \beta$ in the usual sense.

\begin{lemma}
Let $N$ be $\alpha$-suitable and $p \in \p_\alpha$ be determined. 
Suppose that $p \restriction N \in \p_\alpha$. 
If $p \le_\alpha s$, where $s \in N$ is determined, 
then $p \restriction N \le_\alpha s$.
\end{lemma}

\begin{lemma}
Let $M$ and $N$ be $\alpha$-suitable such that $M \cap \kappa = N \cap \kappa$. 
Let $p \in \p_\alpha$ be determined. 
Then $p \restriction M = p \restriction N$.
\end{lemma}

\begin{proof}
As usual, $M \cap \alpha = N \cap \alpha$, and it is easily checked that 
properties (1), (2), and (3) in the definition of $p \restriction K$ 
in Definition 6.14 depend only on $K \cap \alpha$. 
For (4), since $M$ and $N$ are closed under subsets of size less than 
$M \cap \kappa = N \cap \kappa$, it follows that 
for all $\gamma$ in $\dom(X_p) \cap N = \dom(X_p) \cap M$, 
the members of $X_p(\gamma) \cap N$ are exactly those sets in 
$X_p(\gamma)$ which are subsets of $M \cap \gamma = N \cap \gamma$ 
of size less than $M \cap \kappa = N \cap \kappa$, and similarly 
with $X_p(\gamma) \cap M$.  
Hence, $X_p(\gamma) \cap M = X_p(\gamma) \cap N$.
\end{proof}

\begin{definition}
Let $N$ be $\alpha$-suitable. 
Define $\#_N^\alpha(p,s)$ to mean:
\begin{enumerate}
\item $p \in D(N,\alpha)$;
\item $s \in N \cap \p_\alpha$;
\item $p \restriction N = s$.
\end{enumerate}
Also, we define $\#_N^\alpha(p,q,s)$ to mean the conjunction of 
$\#_N^\alpha(p,s)$ and $\#_N^\alpha(q,s)$.
\end{definition}

Note that if $p \in D(N,\alpha)$, then $\#_N^\alpha(p,p \restriction N)$ holds 
iff $p \restriction N \in \p_\alpha$.

Recall that we are assuming in this section that $\kappa < \alpha$. 
But let us extend the definitions of $p(N,\alpha)$, $D(N,\alpha)$, $p \restriction N$, 
and $\#_N^\alpha$ 
in the case $\alpha \le \kappa$ as follows. 
Assuming $\alpha \le \kappa$, 
let $p(N,\alpha) := (\emptyset,\emptyset)$, 
$D(N,\alpha) := \p_\alpha$, 
and 
$$
p \restriction N := (a_p(0) \restriction (\omega_1 \times (N \cap \kappa)), \emptyset).
$$ 
And let $\#_N^\alpha(p,s)$ mean that $p \in \p_\alpha$ and $p \restriction N = s$. 
Note that for $\alpha \le \kappa$, 
for all $p \in \p_\alpha$, $p \restriction N \in \p_\alpha$ and 
$\#_N^\alpha(p,p \restriction N)$ holds.

\begin{lemma}
Let $N$ be $\alpha$-suitable and assume that for all $\beta \in N \cap \alpha$, 
$p(N,\beta)$ is $(N,\p_\beta)$-generic. 
Then the set of $r \in \p_\alpha$ such that 
$\#_N^\alpha(r,r \restriction N)$ holds is dense below $p(N,\alpha)$.
\end{lemma}

\begin{proof}
We will prove the lemma by induction on $\alpha < \kappa^+$. 
Recall that we are assuming that $\kappa < \alpha$. 
However, the statement of the lemma is automatically true 
for $\alpha \le \kappa$ as well, by the comments preceding the lemma. 
This fact will serve as the base case of our induction. 

Now assume that $\kappa < \alpha$ and the lemma is true for all 
$\beta < \alpha$. 
Let $p \le_\alpha p(N,\alpha)$, and we will 
find $r \le_\alpha p$ satisfying $\#_N^\alpha(r,r \restriction N)$. 
Recall that if $r \in D(N,\alpha)$, then $\#_N^\alpha(r,r \restriction N)$ holds iff 
$r \restriction N \in \p_\alpha$.

\bigskip 

\noindent Case 1: $\alpha = \beta+1$ is a successor ordinal. 
By extending $p$ if necessary using Lemma 6.10, we may assume without 
loss of generality that $p \in D_{\beta,\alpha} \cap D(N,\alpha)$. 
Define $E$ as the dense open set of determined conditions in $\p_\beta$ which decide whether or not 
$a_p(\beta) \restriction (N \cap \kappa)$ is in $\p(\dot T_\beta,\dot U_\beta)$. 
By elementarity, $E \in N$. 
Since $p(N,\beta)$ is $(N,\p_\beta)$-generic, $N \cap E$ is predense below $p(N,\beta)$.

By Lemma 6.6, $p \restriction \beta \le_\beta p(N,\beta)$. 
Hence, we can find $u \le_\beta p \restriction \beta$ which is below some $s \in N \cap E$. 
By extending $u$ further if necessary using the inductive hypothesis, 
we may assume that $\#_N^\beta(u,u \restriction N)$ holds. 
Now define $r := u + p$.

We claim that $\#_N^\alpha(r,r \restriction N)$ holds. 
Since $\#_N^\beta(u,u \restriction N)$ holds, $u \in D(N,\beta)$. 
We also know that $p \in D(N,\alpha)$. 
By Lemma 6.12(2), it follows that $u + p = r$ is in $D(N,\alpha)$. 
So it suffices to show that $r \restriction N \in \p_\alpha$. 
Referring to the definition of $\p_{\beta+1}$ in Section 5, 
properties (1), (2), and (5) are immediate. 
For (3), by Lemma 6.15 we have that 
$(r \restriction N) \restriction \beta = (r \restriction \beta) \restriction N = 
u \restriction N \in \p_\beta$.

For (4), since $u \le_\beta s \in N \cap E$ and $E$ is open, by Lemma 6.16 we have that 
$(r \restriction N) \restriction \beta = u \restriction N \le_\beta s$ and so 
$(r \restriction N) \restriction \beta \in E$. 
So $u \restriction N = (r \restriction N) \restriction \beta$ decides whether or not 
$a_p(\beta) \restriction (N \cap \kappa) = a_{r \restriction N}(\beta)$ 
is in $\p(\dot T_\beta,\dot U_\beta)$. 
But $u$ is below $p \restriction \beta$ and $u \restriction N$ in $\p_\beta$, 
and $p \restriction \beta$ forces that 
$a_p(\beta) \restriction (N \cap \kappa)$ is in $\p(\dot T_\beta,\dot U_\beta)$. 
So $u \restriction N = (r \restriction \beta) \restriction N$ does as well. 
For (6), if $M \in X_{r \restriction N}(\beta+1) = X_r(\beta+1) \cap N$ and 
$\gamma \in M \cap \dom(a_{r \restriction N}) = M \cap \dom(a_r) \cap N = 
M \cap \dom(a_r)$, then since $r$ is a condition, 
we have that 
$M \cap \kappa \in A_{a_r(\gamma)} \cap N = A_{a_{r \restriction N}(\gamma)}$. 

\bigskip

\noindent Case 2: $\alpha$ is a limit ordinal with uncountable cofinality. 
Since $N^\omega \subseteq N$, 
easily $\sup(N \cap \alpha)$ has uncountable cofinality. 
So we can fix $\beta \in N \cap \alpha$ which is strictly 
greater than $\sup(\dom(a_p) \cap N)$ 
and $\sup(\dom(X_p) \cap N)$.

Using Lemma 6.10, fix $q \le_\alpha p$ in $D(N,\alpha) \cap D_{\beta,\alpha}$. 
By Lemma 6.6, $q \restriction \beta \le_\beta p(N,\beta)$. 
By the inductive hypothesis, fix $r_0 \le_\beta q \restriction \beta$ 
such that $\#_N^\beta(r_0,r_0 \restriction N)$ holds. 
Let $r := r_0 + q$. 
Then $r \le_\alpha q \le_\alpha p$.

We claim that $\#_N^\alpha(r,r \restriction N)$ holds.
By Lemma 6.12(2), $r \in D(N,\alpha)$. 
It remains to show that $r \restriction N \in \p_\alpha$. 
But since $\beta$ is strictly larger than any ordinal in 
$\dom(a_q) \cap N$ and $\dom(X_q) \cap N$, it easily follows that 
$r \restriction N = r_0 \restriction N$, which is in $\p_\beta$ and hence 
in $\p_\alpha$.

\bigskip

\noindent Case 3: $\alpha$ is a limit ordinal with cofinality $\omega$. 
Fix an increasing and cofinal sequence $\langle \alpha_n : n < \omega \rangle$ 
in $\alpha$ which is a member (and hence a subset) of $N$. 
By Lemma 6.10, we can fix $q \le_\alpha p$ such that 
$q \in D(N,\alpha)$ and $q \in D_{\alpha_n,\alpha}$ for all $n < \omega$. 
By Lemma 6.6, for all $n < \omega$, 
$q \restriction \alpha_n \le_{\alpha_n} p(N,\alpha_n)$.

We define by induction a descending sequence of conditions 
$\langle q_n : n < \omega \rangle$ in $D(N,\alpha)$ such that for all 
$n < m$, $(q_m \restriction \alpha_n) \restriction N \in \p_{\alpha_n}$. 
Let $q_0 := q$. 
Fix $n < \omega$, and assume that $q_n$ is defined so that 
$a_{q_n} \restriction [\alpha_n,\alpha) = a_q \restriction [\alpha_n,\alpha)$ and 
$X_{q_n} \restriction (\alpha_n,\alpha] = X_{q} \restriction (\alpha_n,\alpha]$. 
Since $q \in D_{\alpha_{n+1},\alpha}$, 
$q \restriction \alpha_{n+1} \le_{\alpha_{n+1}} p(N,\alpha_{n+1})$ 
by Lemma 6.6. 
As $q_n \restriction \alpha_{n+1} \le_{\alpha_{n+1}} q \restriction \alpha_{n+1}$, 
$q_n \restriction \alpha_{n+1} \le_{\alpha_{n+1}} p(N,\alpha_{n+1})$. 

By the inductive hypothesis, we can fix a condition 
$q_n' \le_{\alpha_{n+1}} q_n \restriction \alpha_{n+1}$ 
such that $\#_N^{\alpha_{n+1}}(q_n',q_n' \restriction N)$ holds. 
Define $q_{n+1} := q_n' + q$. 
Note that since $q_{n+1} \restriction \alpha_{n+1} = q_n' \in D(N,\alpha_{n+1})$, 
it follows by Lemma 6.12(2) that $q_{n+1} \in D(N,\alpha)$. 
Also, clearly 
$a_{q_{n+1}} \restriction [\alpha_{n+1},\alpha) = 
a_q \restriction [\alpha_{n+1},\alpha)$ and 
$X_{q_{n+1}} \restriction (\alpha_{n+1},\alpha] = 
X_{q} \restriction (\alpha_{n+1},\alpha]$. 
Now for all $k \le n$, 
$(q_{n+1} \restriction \alpha_k) \restriction N = (q_{n'} \restriction \alpha_k) \restriction N = 
(q_{n'} \restriction N) \restriction \alpha_k$, 
which is in $\p_{\alpha_k}$ since $q_{n'} \restriction N \in \p_{\alpha_{n+1}}$. 
So $q_{n+1}$ is as required.

This completes the construction of $\langle q_n : n < \omega \rangle$. 
Let $r$ be the greatest lower bound of this sequence. 
By Lemma 6.13, $r \in D(N,\alpha)$. 
It remains to show that $r \restriction N \in \p_\alpha$. 
Fix $n < \omega$. 
It is easy to check that 
$(r \restriction \alpha_n) \restriction N$ is the greatest lower bound 
of $\langle (q_m \restriction \alpha_n) \restriction N : n < m < \omega \rangle$, and hence 
is in $\p_{\alpha_n}$ since each $(q_m \restriction \alpha_n) \restriction N$ 
is in $\p_{\alpha_n}$. 

Properties (1)--(4) in the definition of $\p_\alpha$ from Section 5 
are now easy to verify for 
$r \restriction N$. 
For (5), suppose that $M \in X_{r \restriction N}(\alpha) = X_r(\alpha) \cap N = 
X_q(\alpha) \cap N$ and $\gamma \in \dom(a_{r \restriction N}) \cap M = 
\dom(a_r) \cap M$. 
Then for some $n < \omega$, $\gamma \in \dom(a_{q_n}) \cap M$, so 
$M \cap \kappa \in A_{a_{q_n}(\gamma)} \cap N \subseteq 
A_{a_{r \restriction N}(\gamma)}$. 
\end{proof}

\begin{lemma}
Let $N$ be $\alpha$-suitable. 
Suppose that $\langle p_n : n < \omega \rangle$ is a descending sequence 
of conditions satisfying $\#_N^\alpha(p_n,p_n \restriction N)$ for all $n < \omega$. 
Let $q$ be the greatest lower bound of this sequence and $s$ be the 
greatest lower bound of $\langle p_n \restriction N : n < \omega \rangle$. 
Then $s = q \restriction N$ and $\#_N^\alpha(q,s)$ holds.
\end{lemma}

\begin{proof}
By Lemma 6.13, $q \in D(N,\alpha)$. 
It is routine to check that $s = q \restriction N$. 
Since $s \in \p_\alpha$, $\#_N^\alpha(q,s)$ holds.
\end{proof}

\begin{lemma}
Let $N$ be $\alpha$-suitable, $p \in \p_\alpha$, 
and $\beta \in N \cap \alpha$.
Assume that $\#_N^\alpha(p,s)$ holds.
\begin{enumerate}

\item If $p \in D_{\beta,\alpha}$, then 
$\#_N^\beta(p \restriction \beta, s \restriction \beta)$ holds.

\item Suppose that $p$ and $s$ are in $D_{\beta,\alpha}$, 
$u \le_\beta p \restriction \beta$, 
$t \le_\beta s \restriction \beta$, and 
$\#_N^\beta(u,t)$ holds. 
Then $\#_N^\alpha(u + p,t + s)$ holds.
\end{enumerate}
\end{lemma}

\begin{proof}
(1) By Lemma 6.12(1), $p \restriction \beta \in D(N,\beta)$. 
By Lemma 6.15 and elementarity, 
$s \restriction \beta = 
(p \restriction N) \restriction \beta = (p \restriction \beta) \restriction N$ 
is in $N \cap \p_\beta$. 
So $\#_N^\beta(p \restriction \beta,s \restriction \beta)$ holds.

(2) By Lemma 6.12(2), $u + p$ is in $D(N,\alpha)$. 
And $t + s$ is in $N \cap \p_\alpha$ by elementarity. 
Now it is easy to check that 
$(u + p) \restriction N = (u \restriction N) + (p \restriction N) = t + s$. 
Thus, $\#_N^\alpha(u + p,t + s)$ holds.
\end{proof}

\begin{lemma}
Let $N$ be $\alpha$-suitable. 
Suppose that $\langle (p_n,q_n) : n < \omega \rangle$ is a descending 
sequence of conditions in $\p_\alpha$ and 
$\#_N^\alpha(p_n,q_n,s_n)$ holds for all $n  <  \omega$. 
Let $p$, $q$, and $s$ be the greatest lower bounds of the sequences 
$\langle p_n : n < \omega \rangle$, 
$\langle q_n : n < \omega \rangle$, 
and $\langle s_n : n < \omega \rangle$ respectively. 
Then $\#_N^\alpha(p,q,s)$ holds.
\end{lemma}

\begin{proof}
Immediate from Lemma 6.20.
\end{proof}

\begin{lemma}
Let $N$ be $\alpha$-suitable and assume that $\#_N^\alpha(p,q,s)$ holds. 
Let $x \subseteq N \cap \alpha$ be countable. 
Then there is $(p_0,q_0,s_0) \le_\alpha (p,q,s)$ satisfying 
$\#_N^\alpha(p_0,q_0,s_0)$ and for all $\beta \in x$, 
$p_0$, $q_0$, and $s_0$ are in $D_{\beta,\alpha}$.
\end{lemma}

\begin{proof}
Define $p_0$, $q_0$, and $s_0$ from $p$, $q$, and $s$ as described in Lemma 5.8. 
Using property (2) of Definition 6.8, 
it is routine to check that $\#_N^\alpha(p_0,q_0,s_0)$ holds.
\end{proof}

\begin{lemma}
Let $M$ and $N$ be $\alpha$-suitable, where $M \in N$. 
Assume that for all $\beta \in M \cap \alpha$, $p(M,\beta)$ is 
$(M,\p_\beta)$-generic. 
Suppose that $r \in \p_\alpha$ is such that 
$\#_N^\alpha(r,r \restriction N)$ holds 
and $r \restriction N \in M$. 
Then there is $t \le_\alpha r$ such that $\#_M^\alpha(t,t \restriction M)$ holds.
\end{lemma}

\begin{proof}
Define $s$ as follows. 
Let $\dom(a_s) := \dom(a_r)$ and $\dom(X_s) := \dom(X_r)$. 
Consider $\beta \in \dom(a_s)$. 
If $\beta \notin M$, then let $a_s(\beta) := a_r(\beta)$. 
Suppose that $\beta \in M$. 
Then $\beta \in N$. 
Since $r \restriction N \in M$, we have that 
$A_{a_r(\beta)} \cap N \cap \kappa \subseteq M \cap \kappa$. 
As $r \in D(N,\alpha)$, $r \restriction \beta$ forces that $a_r(\beta)$ is injective on 
$N \cap \kappa$. 
By Lemma 4.13, we can fix a $\p_\beta$-name $a_s(\beta)$ for an extension of $a_r(\beta)$ 
such that $M \cap \kappa \in A_{a_s(\beta)}$. 
Now consider $\gamma \in \dom(X_s)$. 
If $\gamma \notin M$, then let $X_s(\gamma) := X_r(\gamma)$. 
Suppose that $\gamma \in M$. 
Let $X_s(\gamma) := X_r(\gamma) \cup \{ M \cap \gamma \}$.

It is straightforward to check that $s$ is a condition, $s \le_\alpha r$, and 
$s \le_\alpha p(M,\alpha)$. 
By Lemma 6.19, we can fix $t \le_\alpha s$ such that $\#_M^\alpha(t,t \restriction M)$ holds.
\end{proof}

Recall from Definition 1.2 that for a forcing poset $\p$ and a set $N$, 
$*_N^\p(p,s)$ means that $p \in \p$, $s \in N \cap \p$, and every extension of $s$ in 
$N \cap \p$ is compatible with $p$.

\begin{definition}
Let $N$ be $\alpha$-suitable. 
Define $*_N^\alpha(p,s)$ to mean:
\begin{enumerate}
\item $p \in \p_\alpha$ is determined;
\item $p \le_\alpha p(N,\alpha)$;
\item $s \in N \cap \p_\alpha$ is determined;
\item for all $t \le_\alpha s$ in $N$, 
$t$ and $p$ are compatible in $\p_\alpha$.
\end{enumerate}
Also, we define $*_N^\alpha(p,q,s)$ to mean the conjunction of 
$*_N^\alpha(p,s)$ and $*_N^\alpha(q,s)$.
\end{definition}

Note that if $*_N^\alpha(p,s)$ holds, then for all determined 
$t \le_\alpha s$ in $N$, $*_N^\alpha(p,t)$ also holds.

\begin{lemma}
Let $N$ be $\alpha$-suitable. 
Then $*_N^\alpha(p,s)$ holds iff 
$*_N^{\p_\alpha}(p,s)$ holds, $p \le_\alpha p(N,\alpha)$, and $p$ and $s$ are determined.
\end{lemma}

\begin{lemma}
Let $N$ be $\alpha$-suitable and 
assume that $p(N,\alpha)$ is strongly $(N,\p_\alpha)$-generic. 
Suppose that $*_N^\alpha(p,q,s)$ holds, 
and $D$ and $E$ are sets of determined conditions which are dense below $p(N,\alpha)$. 
Then there is $(p',q',s') \le_\alpha (p,q,s)$ such that $p' \in D$, $q' \in E$, 
$p'$ and $q'$ are below $s$, and $*_N^\alpha(p',q',s')$ holds.
\end{lemma}

\begin{proof}
Note that $N$ is suitable for $\p_\alpha$ in the sense of Notation 1.4. 
And $*_N^\alpha(p,q,s)$ implies $*_N^{\p_\alpha}(p,q,s)$ as in Definition 1.2. 
By Lemma 1.10, there is $(p',q',s') \le_\alpha (p,q,s)$ 
such that $p' \in D$, $q' \in E$, $p'$ and $q'$ are below $s$, 
and $*_N^{\p_\alpha}(p',q',s')$ hold. 
Extending $s'$ further in $N$ if necessary, we may assume without loss of generality 
that $s'$ is determined. 
By Lemma 6.26, $*_N^\alpha(p',q',s')$ holds and we are done.
\end{proof}

\begin{lemma}
Let $N$ be $\alpha$-suitable 
and assume that $*_N^\alpha(r,u)$ holds. 
Then there is $v$ in $N \cap \p_\alpha$ satisfying:
\begin{enumerate}
\item $v \le_\alpha u$ is determined;

\item $\dom(a_r) \cap N \subseteq \dom(a_v)$;

\item $\dom(X_r) \cap N \subseteq \dom(X_v)$;

\item $a_r(0) \restriction 
(\omega_1 \times (N \cap \kappa)) \subseteq a_v(0)$;

\item for all $\beta \in \dom(a_r) \cap N$, 
$v \restriction \beta$ forces in $\p_\beta$ that 
$a_r(\beta) \restriction (N \cap \kappa)$ is in $\p(\dot T_\beta,\dot U_\beta)$ and 
$a_v(\beta) \le a_r(\beta) \restriction (N \cap \kappa)$ in $\p(\dot T_\beta,\dot U_\beta)$;

\item for all $\beta \in \dom(X_r) \cap N$, 
$X_r(\beta) \cap N \subseteq X_v(\beta)$.
\end{enumerate}
\end{lemma}

\begin{proof}
Since $*_N^\alpha(r,u)$ holds, we can fix a determined condition 
$v \le_\alpha r, u$. 
Note that statements (1)--(6) hold for $v$, and the parameters mentioned in 
these statements are in $N$. 
So by elementarity, there is $v \in N \cap \p_\alpha$ satisfying (1)--(6).
\end{proof}

\begin{lemma}
Let $N$ be $\alpha$-suitable 
and assume that $p(N,\alpha)$ is strongly $(N,\p_\alpha)$-generic. 
Suppose that $*_N^\alpha(p,q,s)$ holds. 
Then there is $(p',q',s') \le_\alpha (p,q,s)$ such that $\#_N^\alpha(p',q',s')$ holds.
\end{lemma}

\begin{proof}
We define by induction 
a descending sequence $\langle (p_n,q_n,s_n) : n < \omega \rangle$. 
Let $(p_0,q_0,s_0) := (p,q,s)$. 
Fix $n < \omega$ and assume that $(p_n,q_n,s_n)$ has been defined so that 
$p_n$ and $q_n$ are in $D(N,\alpha)$ and $*_N^\alpha(p_n,q_n,s_n)$ holds.

Since $*_N^\alpha(p_n,q_n,s_n)$ holds, 
in particular, we have that $*_N^\alpha(p_n,s_n)$ holds. 
So we can fix $t_n \in N \cap \p_\alpha$ satisfying 
properties (1)--(6) of Lemma 6.28, where 
$r = p_n$, $u = s_n$, and $v = t_n$. 

Since $t_n \le_\alpha s_n$ is determined in $N$ 
and $*_N^\alpha(q_n,s_n)$ holds, we also have that 
$*_N^\alpha(q_n,t_n)$ holds. 
So we can fix $w_n \in N \cap \p_\alpha$ satisfying properties 
(1)--(6) of Lemma 6.28, where 
$r = q_n$, $u = t_n$, and $v = w_n$.

Since $w_n \le_\alpha s_n$ is in $N$ and determined, we have that $*_N^\alpha(p_n,q_n,w_n)$ holds. 
By Lemma 6.27, we can fix 
$(p_{n+1},q_{n+1},s_{n+1}) \le_\alpha (p_n,q_n,w_n)$ such that 
$p_{n+1}$ and $q_{n+1}$ are in $D(N,\alpha)$, are below $w_n$, 
and $*_N^\alpha(p_{n+1},q_{n+1},s_{n+1})$ holds.

This complete the induction. 
Now let $p'$, $q'$, and $s'$ be the greatest lower bounds of the 
sequences $\langle p_n : n < \omega \rangle$, $\langle q_n : n < \omega \rangle$, 
and $\langle s_n : n < \omega \rangle$ respectively. 
We claim that $\#_N^\alpha(p',q',s')$ holds. 
By Lemma 6.13, $p'$ and $q'$ are in $D(N,\alpha)$, and clearly 
$s' \in N \cap \p_\alpha$. 
By construction, it is easy to check that $p' \restriction N  = q' \restriction N = s'$.
\end{proof}

\begin{definition}
Let $\delta < \kappa$, and 
suppose that $\theta$ and $\tau$ are ordinals in $I_\delta$ which are 
either both even or both odd. 
Let $\dot W_\alpha$ denote $\dot T_\alpha$ if they are even and 
$\dot U_\alpha$ if they are odd. 
Let $p$ and $q$ be in $\p_\alpha$. 
We say that $(p,q)$ 
\emph{$\delta$-separates $(\theta,\tau)$ in $\p_\alpha$} 
if there exists $\gamma < \delta$ 
and distinct ordinals $\theta'$ and $\tau'$ in $I_\gamma$ such that 
$$
p \Vdash_{\p_\alpha} \theta' <_{\dot W_\alpha} \theta, \ \ \textrm{and} \ \
q \Vdash_{\p_\alpha} \tau' <_{\dot W_\alpha} \tau.
$$
\end{definition}

Note that if $(p,q)$ $\delta$-separates $(\theta,\tau)$ in $\p_\alpha$, then for all 
$(p',q') \le_\alpha (p,q)$, 
$(p',q')$ also $\delta$-separates $(\theta,\tau)$ in $\p_\alpha$.

\begin{lemma}
Let $N$ be $\alpha$-suitable and let $\lambda := N \cap \kappa$. 
Suppose that $p(N,\alpha)$ is strongly $(N,\p_\alpha)$-generic, 
the forcing poset $N \cap \p_\alpha$ forces that $\lambda = \omega_2$, and 
$N \cap \dot T_\alpha$ and $N \cap \dot U_\alpha$ 
are $(N \cap \p_\alpha)$-names for trees with no chains of order type $\lambda$. 

Assume that $*_N^\alpha(p,s)$ holds. 
Let $\theta \in I_\lambda$. 
Then there is $(p_0,p_1,t) \le_\alpha (p,p,s)$ such that 
$*_N^\alpha(p_0,p_1,t)$ holds and $(p_0,p_1)$ $\lambda$-separates 
$(\theta,\theta)$ in $\p_\alpha$.
\end{lemma}

\begin{proof}
Let $\dot W$ denote $\dot T_\alpha$ if $\theta$ is even, 
and $\dot U_\alpha$ if $\theta$ is odd.

Since $*_N^\alpha(p,s)$ holds, $p$ and $s$ are compatible. 
Forcing below them, fix a generic filter $G$ on $\p_\alpha$ such that $p$ and $s$ are in $G$. 
Let $W := \dot W^G$ and $H := N \cap G$. 
Let $\q := (\p_\alpha / p(N,\alpha)) / H$, as described in Notation 1.5. 
Since $p \le_\alpha p(N,\alpha)$ and $p(N,\alpha)$ is strongly $(N,\p_\alpha)$-generic, 
we can apply Proposition 1.8 to conclude that 
$H$ is a $V$-generic filter on $N \cap \p_\alpha$, 
$G' := G \cap \q$ is $V[H]$-generic filter on $\q$, 
and $V[G] = V[H][G']$.

By the assumptions of the proposition, in $V[H]$ we have that 
$\lambda = \omega_2$ and 
$W' := (N \cap \dot W)^{H}$ is a tree with no chains of order type $\lambda$. 
Since $H \subseteq G$, $W'$ is obviously a subtree of $W$, and a 
straightforward argument using the 
elementarity of $N$ and the $(N \cap \p_{\alpha})$-genericity of $H$ 
shows in fact that $W' = W \restriction \lambda$.

Let $b_\theta = \{ x < \lambda : x <_{W} \theta \}$, and 
let $\dot b_\theta$ be a $\p_\alpha$-name which is  
forced to satisfy this definition. 
Since $b_\theta$ is a chain in $W'$ of order type $\lambda$, 
it follows that $b_\theta \notin V[H]$.

It follows by a standard argument that we can find 
$p_0', p_1' \le p$ in $\q$ and $\gamma < \lambda$ 
such that $p_0'$ and $p_1'$ decide in $\q$ the node of $\dot b_\theta$ on level 
$\gamma$ of $\dot W$ differently. 
Let $\theta_0$ and $\theta_1$ be the respective decisions of 
$p_0'$ and $p_1'$ for which ordinal of height $\gamma$ is in $\dot b_\theta$.
So $\theta_0 \ne \theta_1$ are in $I_\gamma$, 
$$
p_0' \Vdash_{\q}^{V[H]} \theta_0 <_{\dot W} \theta, \ \ \textrm{and} \ \  
p_1' \Vdash_{\q}^{V[H]} \theta_1 <_{\dot W} \theta.
$$

Fix $s' \le s$ in $H$ such that $s'$ forces that $p_0'$ and $p_1'$ are in $\q$ and 
that $p_0'$ and $p_1'$ forces the information above. 
Since $p_0' \in \q = (\p_\alpha / p(N,\alpha)) / H$ and $s' \in H$, 
by Lemma 1.7(3) we can fix a determined condition 
$p_0 \in \q$ such that $p_0 \le_\alpha p_0', s'$. 
Similarly, since $p_1' \in \q$ and $s \in H$, 
by Lemma 1.7(3) we can fix a determined 
condition $p_1 \in \q$ such that $p_1 \le_\alpha p_1', s'$. 
Fix $t \le s'$ in $H$ which is determined and 
forces in $N \cap \p_\alpha$ that $p_0$ and $p_1$ are in $\q$.

By Lemma 1.6, $*_N^{\p_\alpha}(p_0,p_1,t)$ holds. 
Since $p_0$ and $p_1$ are below $p(N,\alpha)$ and $p_0$, $p_1$, and $t$ are determined, 
it follows that $*_N^\alpha(p_0,p_1,t)$ holds by Lemma 6.26. 
Using the fact that $p_0 \le_\alpha s'$ and $p_1 \le_\alpha s'$, it is easy to check that 
$$
p_0 \Vdash_{\p_\alpha}^{V} \theta_0 <_{\dot W} \theta, \ \ \textrm{and} \ \  
p_1 \Vdash_{\p_\alpha}^{V} \theta_1 <_{\dot W} \theta.
$$
So $(p_0,p_1)$ $\lambda$-separates $(\theta,\theta)$ in $\p_\alpha$.
\end{proof}

\begin{proposition}
Let $N$ be $\alpha$-suitable and let $\lambda := N \cap \kappa$. 
Suppose that $p(N,\alpha)$ is strongly $(N,\p_\alpha)$-generic, 
the forcing poset $N \cap \p_\alpha$ forces that $\lambda = \omega_2$, and 
$N \cap \dot T_\alpha$ and $N \cap \dot U_\alpha$ 
are $(N \cap \p_\alpha)$-names for trees with no chains of order type $\lambda$. 

Assume that $*_N^\alpha(p,q,s)$ holds. 
Let $\theta$ and $\tau$ be ordinals in $I_{\lambda}$ 
which are either both even or both odd. 
Then there is $(p',q',s') \le_\alpha (p,q,s)$ 
such that $*_N^\alpha(p',q',s')$ holds and 
$(p',q')$ $\lambda$-separates $(\theta,\tau)$ in $\p_\alpha$.
\end{proposition}

\begin{proof}
Let $\dot W$ denote $\dot T_\alpha$ if $\theta$ and $\tau$ are both even, 
and $\dot U_\alpha$ if $\theta$ and $\tau$ are both odd. 
By Lemma 6.31, fix $(p_0,p_1,t) \le_\alpha (p,p,s)$ such that 
$*_N^\alpha(p_0,p_1,t)$ holds and $(p_0,p_1)$ $\lambda$-separates 
$(\theta,\theta)$ as witnessed by a pair of distinct ordinals 
$(\theta_0,\theta_1)$ in $I_\gamma$, where $\gamma < \lambda$. 

Since $*_N^\alpha(q,s)$ holds and $t \le_\alpha s$ is determined, 
also $*_N^\alpha(q,t)$ holds. 
Using Lemmas 1.6, 1.7(2), and 6.26, we can find $q' \le_\alpha q$ determined, 
$t' \le_\alpha t$ determined in $N \cap \p_\alpha$, 
and $\tau'$ such that $q'$ decides $p_{\dot W}(\tau,\gamma)$ as $\tau'$ 
and $*_{N}^\alpha(q',t')$ holds. 

We have that $\theta_0 \ne \theta_1$ and $\tau'$ are in $I_\gamma$. 
Let $\theta'$ be the ordinal in $\{ \theta_0, \theta_1 \}$ which is different from $\tau'$, 
and let $p' = p_i$, where $\theta' = \theta_i$. 
Then $(p',q',t') \le_\alpha (p,q,s)$, $*_N^\alpha(p',q',t')$ holds, and 
$(p',q')$ $\lambda$-separates $(\theta,\tau)$ as witnessed 
by $(\theta',\tau')$ in $I_\gamma$.
\end{proof}

\begin{definition}
Let $N$ be $\alpha$-suitable and $p$ and $q$ be determined conditions in $\p_\alpha$. 
Let $\lambda := N \cap \kappa$. 
We say that $p$ and $q$ are \emph{$N$-separated in $\p_\alpha$} 
if for all nonzero $\beta \in \dom(a_p) \cap \dom(a_q) \cap N$, for all 
$\theta \in (\dom(f_{a_p(\beta)}) \cup \ran(f_{a_p(\beta)})) \cap I_{\lambda}$ 
and $\tau \in (\dom(f_{a_q(\beta)}) \cup \ran(f_{a_q(\beta)})) \cap I_{\lambda}$, 
where $\theta$ and $\tau$ are either both even or both odd, 
the pair $(p \restriction \beta,q \restriction \beta)$ 
$\lambda$-separates $(\theta,\tau)$ in $\p_\beta$.
\end{definition}

Note that this definition also makes sense for $\alpha \le \kappa$, but in that case 
any two conditions in $\p_\alpha$ are vacuously $N$-separated.

The next two lemmas have easy proofs, which we omit.

\begin{lemma}
Let $N$ be $\alpha$-suitable, $\beta \in N \cap \alpha$, 
and $p$ and $q$ in $D_{\beta,\alpha}$. 
Assume that 
$$
\dom(a_p) \cap \dom(a_q) \cap N \subseteq \beta
$$
Let $p' \le_\beta p \restriction \beta$ and $q' \le_\beta q \restriction \beta$ in $\p_\beta$, and 
suppose that $p'$ and $q'$ are $N$-separated in $\p_\beta$. 
Then $p' + p$ and $q' + q$ are $N$-separated in $\p_\alpha$.
\end{lemma}

\begin{lemma}
Let $N$ be $\alpha$-suitable. 
Assume that $\cf(\alpha) = \omega$ and $\langle \alpha_n : n < \omega \rangle$ is a 
strictly increasing sequence of ordinals cofinal in $\alpha$ which belongs to $N$. 
Let $\langle p_n : n < \omega \rangle$ and 
$\langle q_n : n < \omega \rangle$ be descending sequences of conditions in 
$\p_\alpha$ such that for all positive $n < \omega$, 
$p_{n} \restriction \alpha_{n}$ and $q_{n} \restriction \alpha_{n}$ are 
$N$-separated in $\p_{\alpha_n}$. 
Let $p^*$ and $q^*$ be 
the greatest lower bounds of $\langle p_n : n < \omega \rangle$ and 
$\langle q_n : n < \omega \rangle$ respectively. 
Then $p^*$ and $q^*$ are $N$-separated in $\p_\alpha$.
\end{lemma}

\section{Preservation of $\kappa$}

We now complete the argument that the forcing iteration from Section 5 preserves $\kappa$, 
under the assumption that $\kappa$ is ineffable.

\begin{notation}
A sequence $\langle N_\gamma : \gamma \in S \rangle$, where $S \subseteq \kappa$, 
is said to be \emph{$\alpha$-suitable} if:
\begin{enumerate}
\item each $N_\gamma$ is $\alpha$-suitable;
\item $\gamma < \xi$ in $S$ implies that $N_\gamma \in N_\xi$;
\item if $\delta \in S$ is a limit point of $S$, then 
$N_\delta = \bigcup \{ N_\gamma : \gamma \in S \cap \delta \}$.
\end{enumerate}
\end{notation}

Note that an $\alpha$-suitable sequence is suitable in the sense of Definition 2.1.

Recall that we are assuming as an inductive hypothesis that for all 
$\beta < \alpha$, $\p_\beta$ preserves $\kappa$. 
We will need to isolate two additional inductive hypotheses. 
Later we will verify that these hypotheses hold at $\alpha$ as well.

\begin{inductivehyp}
Let $\beta < \alpha$. 
Suppose that 
$\langle N_\lambda : \lambda \in S \rangle$ is $\beta$-suitable, where 
$S \in J^+$. 
Let $T$ be the set of $\lambda \in S$ such that $p(N_\lambda,\beta)$ 
is strongly $(N_\lambda,\p_\beta)$-generic. 
Then $S \setminus T \in J$.
\end{inductivehyp}

\begin{inductivehyp}
Let $\beta < \alpha$. 
Suppose that 
$\langle N_\lambda : \lambda \in S \rangle$ is $\beta$-suitable, where 
$S \in J^+$. 
Let $T$ be the set of $\lambda \in S$ such that whenever 
$\#_{N_\lambda}^\beta(p,q,s)$ holds, then there is $s' \le_\alpha s$ such that 
$*_{N_\lambda}^\beta(p,q,s')$ holds.
Then $S \setminus T \in J$.
\end{inductivehyp}

\begin{lemma}
Let $\langle N_\lambda : \lambda \in S \rangle$ be $\alpha$-suitable, with union $N$, 
where $S \in J^+$. 
Let $T$ be the set of $\lambda \in S$ which satisfy that whenever 
$\#_{N_\lambda}^\alpha(p,q,s)$ holds, 
then there exists $(p',q',s') \le_\alpha (p,q,s)$ in $\p_\alpha$ satisfying:
\begin{enumerate}
\item $\#_{N_\lambda}^\alpha(p',q',s')$;
\item $p'$ and $q'$ are $N_\lambda$-separated in $\p_\alpha$.
\end{enumerate}
Then $S \setminus T \in J$, and in particular, $T \in J^+$.
\end{lemma}

\begin{proof}
The proof is by induction on $\alpha < \kappa^+$. 
Note that when $\alpha \le \kappa$, the statement is vacuously true. 
So assume that $\kappa < \alpha$, and the lemma holds for $\p_\beta$ 
for all $\beta < \alpha$. 
For each $\beta < \alpha$, let 
$S_\beta := \{ \lambda \in S : \beta \in N_\lambda \}$, which is a tail of $S$. 
Then the sequence $\langle N_\lambda : \lambda \in S_\beta \rangle$ 
is $\beta$-suitable, with union $N$, and $S_\beta \in J^+$. 
So by the inductive hypothesis, we can fix a set 
$C_\beta \in J^*$ such that for all $\lambda \in S_\beta \cap C_\beta$, 
whenever $\#_{N_\lambda}^\beta(p,q,s)$ holds, 
then there exists $(p',q',s') \le_\beta (p,q,s)$ satisfying:
\begin{enumerate}
\item $\#_{N_\lambda}^\beta(p',q',s')$;
\item $p'$ and $q'$ are $N_\lambda$-separated in $\p_\beta$.
\end{enumerate}

Define 
$$
C := \{ \lambda < \kappa : \textrm{if} \ \lambda \in S, \ \textrm{then} \ 
\forall \beta \in N_\lambda \cap \alpha, \ 
\lambda \in C_\beta \}.
$$
Then $C \in J^*$ since $J$ is normal. 
Specifically, if $C \notin J^*$, then $C' := \kappa \setminus C$ is a subset of $S$
in $J^+$. 
By Lemma 2.2, we can find $S' \subseteq C'$ in $J^+$ and $\beta < \alpha$
such that for all $\lambda \in S'$, $\lambda \notin C_\beta$. 
This is impossible, since $S' \in J^+$ and $C_\beta \in J^*$ imply that 
$S' \cap C_\beta \ne \emptyset$. 
So it suffices to show that $C \cap (S \setminus T) \in J$.

The proof splits into the three cases of whether 
$\alpha$ is a successor ordinal, $\alpha$ is a limit ordinal with 
countable cofinality, or $\alpha$ is a limit ordinal with uncountable cofinality.

\bigskip

\noindent Case 1: $\alpha$ is a limit ordinal with uncountable cofinality. 
We claim that $C \cap (S \setminus T) = \emptyset$. 
It suffices to show that if $\lambda \in C \cap S$, then $\lambda \in T$. 
So let $\lambda \in C \cap S$, and assume that 
$\#_{N_\lambda}^\alpha(p,q,s)$ holds. 
We will find $(p',q',s') \le_\alpha (p,q,s)$ in $\p_\alpha$ satisfying that 
$\#_{N_\lambda}^\alpha(p',q',s')$ and $p'$ and $q'$ are 
$N_\lambda$-separated in $\p_\alpha$.

Since $N_\lambda^\omega \subseteq N_\lambda$ and $\cf(\alpha) > \omega$, 
it easily follows that 
$\cf(\sup(N_\lambda \cap \alpha)) > \omega$.
So we can fix $\beta \in N_\lambda \cap \alpha$ such that 
$$
(\dom(a_p) \cup \dom(a_q)) \cap N_\lambda \subseteq \beta.
$$
By Lemma 6.23, fix $(p_0,q_0,s_0) \le_\alpha (p,q,s)$ in $D_{\beta,\alpha}$ such that 
$\#_{N_\lambda}^\alpha(p_0,q_0,s_0)$ holds. 
By Lemma 6.21(1),  
$\#_{N_\lambda}^\beta(p_0 \restriction \beta,q_0 \restriction \beta,s_0 \restriction \beta)$ 
holds.

Since $\beta \in N_\lambda$ and $\lambda \in C$, we have that 
$\lambda \in C_\beta \cap S_\beta$. 
By the definition of $C_\beta$, we can fix 
$(p_1,q_1,s_1) \le_\beta (p_0 \restriction \beta,q_0 \restriction \beta,s_0 \restriction \beta)$ 
in $\p_\beta$ satisfying that $\#_{N_\lambda}^\beta(p_1,q_1,s_1)$ holds and 
$p_1$ and $q_1$ are $N_\lambda$-separated in $\p_\beta$.

Define $p' := p_1 + p_0$, 
$q' := q_1 + q_0$, and 
$s' := s_1 + s_0$. 
Then $(p',q',s') \le_\alpha (p_0,q_0,s_0) \le_\alpha (p,q,s)$. 
By Lemma 6.21(2), $\#_N^\alpha(p',q',s')$ holds, and by Lemma 6.34, 
$p'$ and $q'$ are $N_\lambda$-separated in $\p_\alpha$.

\bigskip

\noindent Case 2: $\cf(\alpha) = \omega$. 
We claim that $C \cap (S \setminus T) = \emptyset$. 
It suffices to show that if $\lambda \in C \cap S$, then $\lambda \in T$. 
So let $\lambda \in C \cap S$, and assume that 
$\#_{N_\lambda}^\alpha(p,q,s)$ holds. 
We will find $(p',q',s') \le_\alpha (p,q,s)$ satisfying that 
$\#_{N_\lambda}^\alpha(p',q',s')$ holds and $p'$ and $q'$ are 
$N_\lambda$-separated in $\p_\alpha$.

Since $\alpha \in N_\lambda$, by elementarity we can fix in $N_\lambda$ 
a sequence $\langle \alpha_n : n < \omega \rangle$ which is 
increasing and cofinal in $\alpha$. 
Since $\lambda \in C$, we have that for all $n < \omega$, 
$\lambda \in C_{\alpha_n} \cap S_{\alpha_n}$.

We define by induction a descending sequence  
$\langle (p_n,q_n,s_n) : n < \omega \rangle$ in $\p_\alpha$ below 
$(p,q,s)$ 
satisfying that for all $n < \omega$:
\begin{enumerate}
\item $\#_{N_\lambda}^{\alpha}(p_n,q_n,s_n)$;
\item for all positive $n < \omega$, 
$p_{n} \restriction \alpha_n$ and $q_n \restriction \alpha_n$ 
are $N_\lambda$-separated in $\p_{\alpha_n}$.
\end{enumerate}
If such a sequence can be constructed, then let $p'$, $q'$, and $s'$ 
be the greatest lower bounds of 
$\langle p_n : n < \omega \rangle$, 
$\langle q_n : n < \omega \rangle$, 
and $\langle s_n : n < \omega \rangle$ respectively. 
By Lemma 6.22, $\#_{N_\lambda}^\alpha(p',q',s')$ holds, and by Lemma 6.35, 
$p'$ and $q'$ are $N_\lambda$-separated in $\p_\alpha$.

Let $p_0 := p$, $q_0 := q$, and $s_0 := s$. 
Now fix $n < \omega$ and assume that $(p_n,q_n,s_n)$ is defined as required. 
By Lemma 6.23, we can fix 
$(p_n',q_n',s_n') \le_\alpha (p_n,q_n,s_n)$ such that 
$\#_{N_\lambda}^\alpha(p_n',q_n',s_n')$ holds and 
for all $n < \omega$, 
$p_n'$, $q_n'$, and $s_n'$ are in $D_{\alpha_{n+1},\alpha}$. 
By Lemma 6.21(1), 
$\#_{N_\lambda}^{\alpha_{n+1}}(p_n' \restriction \alpha_{n+1},
q_n' \restriction \alpha_{n+1},s_n' \restriction \alpha_{n+1})$ holds.

Since $\lambda \in C_{\alpha_{n+1}} \cap S_{\alpha_{n+1}}$, 
there is 
$$
(u_n,v_n,t_n) \le_{\alpha_{n+1}} (p_n' \restriction \alpha_{n+1},
q_n' \restriction \alpha_{n+1},s_n' \restriction \alpha_{n+1})
$$ 
such that 
$\#_{N_\lambda}^{\alpha_{n+1}}(u_n,v_n,t_n)$ holds and 
$u_n$ and $v_n$ are $N_\lambda$-separated in $\p_{\alpha_{n+1}}$. 
Define $p_{n+1} := u_n + p_n'$, 
$q_{n+1} := v_n + q_n'$, and $s_{n+1} := t_n + s_n'$. 
By Lemma 6.21(2),  
$\#_{N_\lambda}^\alpha(p_{n+1},q_{n+1},s_{n+1})$ holds. 
And $p_{n+1} \restriction \alpha_{n+1} = u_n$ and 
$q_{n+1} \restriction \alpha_{n+1} = v_n$ 
are $N_\lambda$-separated in $\p_{\alpha_{n+1}}$. 
This completes the construction.

\bigskip

\noindent Case 3: $\alpha = \beta+1$ is a successor ordinal. 
Applying Inductive Hypotheses 7.2 and 7.3 to the $\beta$-suitable 
sequence $\langle N_\lambda : \lambda \in S_\beta \rangle$, 
we can fix a set $D \in J^*$ such that for all 
$\lambda \in S_\beta \cap D$:
\begin{enumerate}
\item $p(N_\lambda,\beta)$ is strongly $(N_\lambda,\p_\beta)$-generic;
\item whenever $\#_{N_\lambda}^\beta(p,q,s)$ holds, then there is $s' \le_\beta s$ such that 
$*_{N_\lambda}^\beta(p,q,s')$ holds.
\end{enumerate}

By the inductive hypotheses, $\p_\beta$ preserves $\kappa$, and hence 
forces that $\kappa = \omega_2$. 
By Proposition 2.4, we can fix a set $E \in J^*$ such that for 
all $\lambda \in S_\beta \cap E$, 
$N_\lambda \cap \p_\beta$ forces that $\lambda = \omega_2$, and 
$N_\lambda \cap \dot T_\beta$ and $N_\lambda \cap \dot U_\beta$ are 
$(N_\lambda \cap \p_\beta)$-names for trees which have no chains of 
order type $\lambda$.

We claim that $C \cap D \cap E \cap (S \setminus T) = \emptyset$. 
It suffices to show that if $\lambda \in C \cap D \cap E \cap S$, then $\lambda \in T$. 
So let $\lambda \in C \cap D \cap E \cap S$, and assume that 
$\#_{N_\lambda}^\alpha(p,q,s)$ holds. 
We will find $(p',q',s') \le_\alpha (p,q,s)$ satisfying that 
$\#_{N_\lambda}^\alpha(p',q',s')$ holds and $p'$ and $q'$ are 
$N_\lambda$-separated in $\p_\alpha$.

By Lemma 6.23, fix $(p_0,q_0,s_0) \le_\alpha (p,q,s)$ in $D_{\beta,\alpha}$ 
such that $\#_{N_\lambda}^\alpha(p_0,q_0,s_0)$ holds. 
Let $\langle \theta_n : n < r_0 \rangle$, where $r_0 \le \omega$, 
enumerate all ordinals in 
$\dom(f_{a_{p_0}(\beta)}) \cup \ran(f_{a_{p_0}(\beta)})$ in $I_\lambda$, 
and let $\langle \tau_n : n  <  r_1 \rangle$, where $r_1 \le \omega$, 
enumerate all ordinals in 
$\dom(f_{a_{q_0}(\beta)}) \cup \ran(f_{a_{q_0}(\beta)})$ in $I_\lambda$.

Let $n \mapsto (n_0,n_1)$ denote a bijection from $\omega$ onto 
$\omega \times \omega$. 
We will define by induction a sequence 
$\langle (u_n,v_n,t_n) : n < \omega \rangle$ of conditions in $\p_\beta$ 
below $(p_0 \restriction \beta,q_0 \restriction \beta,s_0 \restriction \beta)$ 
satisfying that for all $n < \omega$:
\begin{enumerate}
\item $\#_{N_\lambda}^\beta(u_{n},v_{n},t_{n})$ holds;
\item if $n_0 < r_0$ and $n_1 < r_1$ and 
$\theta_{n_0}$ and $\tau_{n_1}$ 
are either both even or both odd, then 
$(u_{n+1},v_{n+1})$ $\lambda$-separates $(\theta_{n_0},\tau_{n_1})$ in $\p_\beta$.
\end{enumerate}

Let $u_0 := p_0 \restriction \beta$, 
$v _0 := q_0 \restriction \beta$, and $t_0 := s_0 \restriction \beta$. 
By Lemma 6.21(1), 
$\#_{N_\lambda}^\beta(u_0,v_0,t_0)$ holds. 

Fix $n < \omega$ and assume that $(u_n,v_n,t_n)$ is defined as required. 
If it is not the case that $n_0 < r_0$, $n_1 < r_1$, and $\theta_{n_0}$ and 
$\tau_{n_1}$ are either both even or both odd, then let 
$(u_{n+1},v_{n+1},t_{n+1}) := (u_n,v_n,t_n)$. 

Otherwise, since $\lambda \in S_\beta \cap D$, we can fix $t_n^* \le_\beta t_n$ such that 
$*_{N_\lambda}^\beta(u_n,v_n,t_n^*)$ holds. 
Also, $p(N_\lambda,\beta)$ is strongly $(N_\lambda,\p_\beta)$-generic. 
By Proposition 6.32 (applied to $\beta$), there is 
$(u_n',v_n',t_n') \le_\beta (u_n,v_n,t_n^*)$ satisfying that 
$*_{N_\lambda}^\beta(u_n',v_n',t_n')$ holds 
and $(u_n',v_n')$ $\lambda$-separates $(\theta_{n_0},\tau_{n_1})$ in $\p_\beta$. 
By Lemma 6.29 (applied to $\beta$), we can fix 
$(u_{n+1},v_{n+1},t_{n+1}) \le_\beta (u_n',v_n',t_n')$ such that 
$\#_{N_\lambda}^\beta(u_{n+1},v_{n+1},t_{n+1})$ holds. 
Then also $(u_{n+1},v_{n+1})$ $\lambda$-separates $(\theta_{n_0},\tau_{n_1})$.

This completes the construction. 
Let $u$, $v$, and $t$ denote the greatest lower bounds in $\p_\beta$ 
of the sequences 
$\langle u_n : n < \omega \rangle$, 
$\langle v_n : n < \omega \rangle$, and 
$\langle t_n : n < \omega \rangle$ respectively. 
By Lemma 6.22, $\#_{N_\lambda}^\beta(u,v,t)$ holds.

Since $\lambda \in S_\beta \cap C_\beta$,  
there exists $(u',v',t') \le_\beta (u,v,t)$ satisfying:
\begin{enumerate}
\item $\#_{N_\lambda}^\beta(u',v',t')$;
\item $u'$ and $v'$ are $N_\lambda$-separated in $\p_\beta$.
\end{enumerate}

Let $p' := u' + p_0$, $q' := v' + q_0$, and $s' := t' + s_0$. 
Then $(p',q',s') \le_\alpha (p,q,s)$, and by Lemma 6.21(2), 
$\#_{N_\lambda}^\alpha(p',q',s')$ holds. 
As $u'$ and $v'$ are $N_\lambda$-separated in $\p_\beta$, by our construction which 
handled separation at $\beta$, 
it is easy to check that $p'$ and $q'$ are $N_\lambda$-separated in $\p_{\alpha}$.
\end{proof}

\begin{proposition}
Let $\langle N_\lambda : \lambda \in S \rangle$ be $\alpha$-suitable, 
with union $N$, where $S \in J^+$. 
Assume that for each $\lambda \in S$, we have fixed conditions 
$p_\lambda$, $q_\lambda$, and $s_\lambda$ such that 
$\#_{N_\lambda}^\alpha(p_\lambda,q_\lambda,s_\lambda)$ holds. 
Then there exists $U \subseteq S$ in $J^+$ such that 
for all $\lambda < \mu$ in $U$, 
$p_\lambda$ and $q_\mu$ are compatible in $\p_\alpha$.
\end{proposition}

\begin{proof}
By Lemma 7.4, we can fix $T \subseteq S$ in $J^+$ 
such that for all $\lambda \in T$,  there is 
$(p_\lambda^*,q_\lambda^*,s_\lambda^*) \le_\alpha (p_\lambda,q_\lambda,s_\lambda)$ 
satisfying that 
$\#_{N_\lambda}^\alpha(p_\lambda^*,q_\lambda^*,s_\lambda^*)$ holds, 
and $p_\lambda^*$ and $q_\lambda^*$ are $N_\lambda$-separated in $\p_\alpha$. 
By intersecting $T$ with a club if necessary, let us also assume that 
$N_\lambda \cap \kappa = \lambda$ for all $\lambda \in T$. 
It suffices to find $U \subseteq T$ in $J^+$ 
such that for all $\lambda < \mu$ in $U$, 
$p_\lambda^*$ and $q_\mu^*$ are compatible in $\p_\alpha$.

Consider $\lambda \in T$. 
Define 
$$
J_\lambda := (\dom(a_{p_\lambda^*}(0)) \cup \dom(a_{q_\lambda^*}(0))) 
\setminus (\omega_1 \times \lambda)
$$
Note that $J_\lambda$ is a subset of 
$N \setminus N_\lambda$ of size 
less than $\kappa$. 
Define 
$$
K_\lambda := \bigcup \{ M \cap N_\lambda : M \in 
\bigcup \ran(X_{q_\lambda^*}), \ M \cap \kappa < \lambda \}.
$$
Note that $K_\lambda$ is the union of 
countably many subsets of $N_\lambda$ 
each of size less than $\lambda$. 
Hence, $K_\lambda$ is a subset of $N_\lambda$ of size less than $\lambda$. 
As $N_\lambda^{< \lambda} \subseteq N_\lambda$, it follows that 
$K_\lambda \in N_\lambda$.

Since $\#_{N_\lambda}^\alpha(p_\lambda^*,q_\lambda^*,s_\lambda^*)$ holds, 
we have that $p_\lambda^* \restriction N_\lambda = s_\lambda^* = 
q_\lambda^* \restriction N_\lambda$. 
In particular, 
$$
\dom(a_{p_\lambda^*}) \cap N_\lambda = \dom(a_{s_\lambda^*}) = 
\dom(a_{q_\lambda^*}) \cap N_\lambda.
$$
Consider $\beta$ in $\dom(a_{s_\lambda^*})$. 
Enumerate all members of 
$\dom(f_{a_{p_{\lambda}^*}(\beta)}) \cup \ran(f_{a_{p_\lambda^*}(\beta)})$ 
in $I_\lambda$ as 
$$
\langle \theta(\lambda,\beta,m) : m < r_{\lambda,\beta,0} \rangle,
$$
where $r_{\lambda,\beta,0} \le \omega$, and enumerate all members of 
$\dom(f_{a_{q_{\lambda}^*}(\beta)}) \cup \ran(f_{a_{q_\lambda^*}(\beta)})$ 
in $I_\lambda$ as 
$$
\langle \tau(\lambda,\beta,n) : n < r_{\lambda,\beta,1} \rangle,
$$
where $r_{\lambda,\beta,1} \le \omega$.

Now define a partial function 
$$
f_\lambda : \dom(a_{s_\lambda^*}) \times r_{\lambda,\beta,0} \times 
r_{\lambda,\beta,1} \to \lambda \times \lambda
$$ 
so that $f_\lambda(\beta,m,n)$ is defined iff 
$\theta(\lambda,\beta,m)$ and $\tau(\lambda,\beta,n)$ 
are either both even or both odd, 
in which case $f_\lambda(\beta,m,n)$ is equal to some witness 
$(\theta',\tau')$ to the fact that the pair 
$(p_\lambda^* \restriction \beta, q_\lambda^* \restriction \beta)$ 
$\lambda$-separates 
$(\theta(\lambda,\beta,m), \tau(\lambda,\beta,n))$ 
in $\p_{\beta}$. 
Specifically, if $f_\lambda(\beta,m,n) = (\theta',\tau')$, then $\theta' \ne \tau'$, and 
for some $\gamma < \lambda$, $\theta'$ and $\tau'$ are in $I_\gamma$, 
$p_\lambda^* \restriction \beta \Vdash_{\p_\beta} \theta' <_{\dot W} \theta$ and 
$q_\lambda^* \restriction \beta \Vdash_{\p_\beta} \tau' <_{\dot W} \tau$, 
where $\dot W = \dot T_\beta$ in the case that 
$\theta(\lambda,\beta,m)$ and $\tau(\lambda,\beta,n)$ are 
both even, and $\dot W = \dot U_\beta$ in the case that 
$\theta(\lambda,\beta,m)$ and $\tau(\lambda,\beta,n)$ are both odd. 
In addition, define a function 
$g_\lambda : \dom(a_{s_\lambda^*}) \to 
[\omega]^{\le \omega}$ 
by letting $g_\lambda(\beta)$ be the set of $n < r_{\lambda,\beta,1}$ 
such that $\tau(\lambda,\beta,n)$ is even. 
Note that $f_\lambda$ and $g_\lambda$ 
are countable subsets of $N_\lambda$, and hence 
are members of $N_\lambda$. 

To summarize, we have that $s_\lambda^*$, $K_\lambda$, $f_\lambda$, 
and $g_\lambda$ are members of $N_\lambda$, 
and $J_\lambda$ is a subset of $N \setminus N_\lambda$ 
of size less than $\kappa$. 
By Lemma 2.2, we can fix $U \subseteq T$ in $J^+$ such that 
for all $\lambda < \mu$ in $U$, 
$s_\lambda^* = s_\mu^*$, 
$K_\lambda = K_\mu$ , $f_\lambda = f_\mu$, 
$g_\lambda = g_\mu$, 
$p_\lambda^* \in N_\mu$, 
and $J_\lambda \cap J_\mu = \emptyset$.

Fix $\lambda < \mu$ in $U$, and we will show that 
$p_\lambda^*$ and $q_\mu^*$ are compatible in $\p_{\alpha}$. 
For notational simplicity, let $p := p_\lambda^*$ and $q := q_\mu^*$. 
We will define a lower bound $r = (a_r,X_r)$ of $p$ and $q$.

Let the domain of $a_r$ be equal to 
$\dom(a_p) \cup \dom(a_q)$ and 
the domain of $X_r$ be equal to $\dom(X_p) \cup \dom(X_q)$. 
For each $\beta \in \dom(X_r)$, define 
$X_r(\beta) := X_p(\beta) \cup X_q(\beta)$.

We will define $a_r(\gamma)$ for all $\gamma \in \dom(a_r)$ 
by induction on $\gamma$. 
Assuming that $\beta \le \alpha$ and $a_r(\gamma)$ is defined 
for all $\gamma \in \dom(a_r) \cap \beta$, 
we will maintain two inductive hypotheses:
\begin{enumerate}

\item $r \restriction \beta := 
(a_r \restriction \beta,X_r \restriction (\beta+1))$ is in $\p_\beta$ and 
extends $p \restriction \beta$ and $q \restriction \beta$ in $\p_\beta$;

\item for all nonzero $\gamma \in \dom(a_r) \cap \beta$ and 
$\xi \in \dom(X_r)$ with $\gamma < \xi$, 
for all $M \in X_r(\xi)$ with $\gamma \in M$, 
$r \restriction \gamma \Vdash_{\p_\gamma} M \cap \kappa \in A_{a_r(\gamma)}$.
\end{enumerate}

\bigskip

To begin, let $a_r(0) := a_p(0) \cup a_q(0)$. 
Since $p \restriction N_\lambda = s_\lambda^* = s_\mu^* = 
q \restriction N_\mu$, 
we have that 
$$
a_{p}(0) \restriction (\omega_1 \times \lambda) = 
a_{s_\lambda^*}(0) = a_{s_\mu^*}(0) = 
a_{q}(0) \restriction (\omega_1 \times \mu).
$$
Also, $\dom(a_{p}(0)) \setminus (\omega_1 \times \lambda) \subseteq 
J_\lambda$, 
$\dom(a_{q}(0)) \setminus (\omega_1 \times \mu) \subseteq 
J_\mu$, 
and $J_\lambda \cap J_\mu = \emptyset$. 
It easily follows that $a_r(0)$ is a function, and hence 
is in $\coll(\omega_1, <\! \kappa)$, and 
$a_r(0)$ extends $a_p(0)$ and $a_q(0)$ in $\coll(\omega_1, <\! \kappa)$. 
Thus, $r \restriction 1 = (a_r \restriction 1,\emptyset)$ is as required. 
We also have that $r \restriction \beta = r \restriction 1$ for all 
$1 \le \beta \le \kappa$, and these objects obviously satisfy the 
inductive hypotheses.

\bigskip

Assume that $\kappa < \beta \le \alpha$ is a limit ordinal and 
$a_r(\gamma)$ is defined for all $\gamma \in \dom(a_r) \cap \beta$. 
Assume that for all 
$\beta' < \beta$, $r \restriction \beta'$ satisfies the inductive hypotheses. 
Let us check that 
$r \restriction \beta = (a_r \restriction \beta,X_r \restriction (\beta+1))$ 
is as required.

By inductive hypothesis (1), we know that for all $\beta' < \beta$, 
$$
r \restriction \beta' = (a_r \restriction \beta',X_r \restriction (\beta'+1))
$$
is in 
$\p_{\beta'}$ and is below $p \restriction \beta'$ and $q \restriction \beta'$.

Let us show that $r \restriction \beta$ is in $\p_\beta$. 
Referring to the definition of $\p_\beta$ from Section 5 
in the case that $\beta$ is a limit ordinal, 
requirements (1)--(4) are immediate. 
For requirement (5), we need to show that 
if $M \in X_r(\beta)$ and 
$\gamma \in M \cap \dom(a_r) \cap \beta$ 
is nonzero, then 
$r \restriction \gamma \Vdash_{\p_\gamma} M \cap \kappa \in A_{a_r(\gamma)}$. 
But this follows immediately from inductive hypothesis (2) holding 
for $r \restriction (\gamma+1)$.

Using inductive hypothesis (1) and the definition of $X_r$, 
it is simple to check that $r \restriction \beta$ extends $p \restriction \beta$ and 
$q \restriction \beta$ in $\p_\beta$. 
Also, inductive hypothesis (2) for $r \restriction \beta$ follows immediately
from the fact that it holds for $r \restriction \beta'$ for all $\beta' < \beta$.

\bigskip

Now assume that $\kappa \le \beta < \alpha$ 
and $a_r \restriction \beta$ is defined as required. 
We will define $a_r(\beta)$, and then show that 
$r \restriction (\beta+1) = (a_r \restriction (\beta+1), X_r \restriction (\beta+2))$ 
satisfies the inductive hypotheses. 
We will consider four separate cases.

\bigskip

\noindent Case 1: $\beta \notin \dom(a_p) \cup \dom(a_q)$. 
In this case, $\beta \notin \dom(a_r)$ and 
$$
r \restriction (\beta+1) = (a_r \restriction \beta,X_r \restriction (\beta+2)).
$$
It is easy to check that $r \restriction (\beta+1)$ is as required.

\bigskip

\noindent Case 2: $\beta \in \dom(a_q) \setminus \dom(a_p)$. 
Since $\dom(a_q) \cap N_\mu = \dom(a_{s_\mu^*}) = \dom(a_{s_\lambda^*}) = 
\dom(a_p) \cap N_\lambda$ and $\beta \notin \dom(a_p)$, it follows that 
$\beta \notin N_\mu$. 
Define $a_r(\beta) := a_q(\beta)$.

Let us show that if $\xi > \beta$ is in $\dom(X_r)$, $M \in X_r(\xi)$, and 
$\beta \in M$, then 
$r \restriction \beta \Vdash_{\p_\beta} M \cap \kappa \in A_{a_q(\beta)}$. 
This statement together with the inductive hypotheses for $r \restriction \beta$ 
easily imply that $r \restriction (\beta+1)$ is as required.

If $M \in X_q(\xi)$, then we are done since $q$ is a condition and 
$r \restriction \beta \le_\beta q \restriction \beta$. 
Assume that $M \in X_p(\xi)$. 
Since $p \in N_\mu$, $M \in N_\mu$. 
So $M \subseteq N_\mu$. 
But then $\beta \in M \subseteq N_\mu$, so $\beta \in N_\mu$. 
This contradicts the observation above that $\beta \notin N_\mu$.

\bigskip

\noindent Case 3: $\beta \in \dom(a_{p}) \setminus \dom(a_q)$. 
Then $\beta \notin N_\lambda$, for otherwise 
$\beta \in \dom(a_p) \cap N_\lambda = \dom(a_{s_\lambda^*}) = 
\dom(a_{s_\mu^*}) = \dom(a_q) \cap N_\mu$, which 
contradicts that $\beta \notin \dom(a_q)$.

Since $p \in N_\mu$, we have that $\sup(A_{a_p(\beta)}) < N_\mu \cap \kappa = \mu$. 
Let $x$ be the set of ordinals of the form $M \cap \kappa$, where 
$M \in X_q(\xi)$ for some $\xi > \beta$ and $M \cap \kappa \ge \mu$. 
Then $x$ is countable and consists of ordinals of uncountable cofinality which 
are greater than or equal to $\mu$. 
By Lemma 4.5, we can fix a nice $\p_\beta$-name $a_r(\beta)$ 
which $r \restriction \beta$ forces is an extension of $a_p(\beta)$ 
such that $x \subseteq A_{a_r(\beta)}$.

We will show that whenever $\xi > \beta$ is in $\dom(X_r)$, $M \in X_r(\xi)$, and 
$\beta \in M$, then 
$r \restriction \beta \Vdash_{\p_\beta} M \cap \kappa \in a_r(\beta)$. 
This claim together with the inductive hypotheses easily imply that 
$r \restriction (\beta+1)$ is as required.

Since $r \restriction \beta$ forces that $A_{a_p(\beta)} \subseteq A_{a_r(\beta)}$, 
if $M \in X_p(\xi)$ then we are done since $p$ is a condition. 
Suppose that $M \in X_q(\xi)$. 
If $M \cap \kappa \ge \mu$, then $M \cap \kappa \in x$, so we are done 
by the choice of $a_r(\beta)$.

We claim that the remaining case $M \cap \kappa < \mu$ does not occur. 
Assume for a contradiction that $M \cap \kappa < \mu$.  
Then $M \cap N_\mu \subseteq K_\mu$ by the definition of $K_\mu$. 
But $K_\mu = K_\lambda$. 
So $M \cap N_\mu \subseteq K_\lambda \subseteq N_\lambda$. 
Note that $\beta \in M \cap N_\mu$; namely, $\beta \in M$ by assumption, and 
$\beta \in N_\mu$, since $p \in N_\mu$. 
So $\beta \in N_\lambda$. 
But we observed above that $\beta \notin N_\lambda$, and we have a contradiction.

\bigskip

\noindent Case 4: $\beta \in \dom(a_p) \cap \dom(a_q)$. 
We claim that it suffices to show that $r \restriction \beta$ forces that 
$a_p(\beta)$ and $a_q(\beta)$ are compatible in 
$\p(\dot T_\beta,\dot U_\beta)$. 
For then we can choose a $\p_\beta$-name $a_r(\beta)$ which 
$r \restriction \beta$ forces is below $a_p(\beta)$ and $a_q(\beta)$. 
As in the previous cases, it then suffices to verify that 
whenever $\xi > \beta$ is in $\dom(X_r)$, $M \in X_r(\xi)$, and 
$\beta \in M$, then 
$r \restriction \beta \Vdash_{\p_\beta} M \cap \kappa \in a_r(\beta)$. 
Since $\beta \in \dom(a_p) \cap \dom(a_q)$, if $M \in X_p(\xi)$ then 
$p \restriction \beta$, and hence $r \restriction \beta$, forces that 
$M \cap \kappa \in A_{a_p(\beta)}$. 
And if $M \in X_q(\xi)$, then $q \restriction \beta$, 
and hence $r \restriction \beta$, forces 
that $M \cap \kappa \in A_{a_q(\xi)}$. 
But $r \restriction \beta$ forces that 
$A_{a_p(\beta)} \cup A_{a_q(\beta)} \subseteq A_{a_r(\beta)}$. 
So in either case, $r \restriction \beta$ forces that 
$M \cap \kappa \in A_{a_r(\beta)}$.

We now complete the proof by showing that 
$r \restriction \beta$ forces that 
$a_p(\beta)$ and $a_q(\beta)$ are compatible. 
For notational simplicity, let $a_p(\beta) = (f,A)$, 
$a_q(\beta) = (g,B)$, and $s = s_\lambda^* = s_\mu^*$. 
Then $p \restriction N_\lambda = s = q \restriction N_\mu$. 
So $\beta \in \dom(a_s)$ and 
$(A,f) \restriction \lambda = a_s(\beta) = (B,g) \restriction \mu$.

Since $\#_{N_\lambda}^\alpha(p,s)$ and $\#_{N_\mu}^\alpha(q,s)$ hold, in particular, 
$p \in D(N_\lambda,\alpha)$ and $q \in D(N_\mu,\alpha)$. 
As $p \in N_\mu$, by elementarity $\beta \in N_\mu$. 
Hence, $\beta \in \dom(a_q) \cap N_\mu = 
\dom(a_{s}) = \dom(a_p) \cap N_\lambda$. 
So $\beta \in N_\lambda \cap N_\mu$. 
Also, since $p \in N_\mu$, by elementarity $A \subseteq \mu$.

The fact that $p \in D(N_\lambda,\alpha)$ and 
$\beta \in \dom(a_p) \cap N_\lambda$ implies by definition that 
$p \restriction \beta \Vdash_{\p_\beta} a_p(\beta)$ is injective on 
$N_\lambda \cap \kappa = \lambda$. 
Similarly, $q \restriction \beta \Vdash_{\p_\beta} a_q(\beta)$ 
is injective on $\mu$.

Since $p \in D(N_\lambda,\alpha)$ and $q \in D(N_\mu,\alpha)$, 
we have that $N_\lambda \cap \alpha \in X_p(\alpha)$ and 
$N_\mu \cap \alpha \in X_q(\alpha)$. 
As $\beta \in \dom(a_p) \cap (N_\lambda \cap \alpha)$, 
$p$ being a condition implies that $p \restriction \beta$ forces that 
$(N_\lambda \cap \alpha) \cap \kappa = N_\lambda \cap \kappa = \lambda \in 
A_{a_p(\beta)}$. 
Hence, $\lambda \in A_{a_p(\beta)} = A$. 
Similarly, $\mu \in B$.

To summarize, the following statements are forced by $r \restriction \beta$ 
to be true:
\begin{enumerate}
\item $\lambda < \mu$, $\lambda \in A$, and $\mu \in B$;
\item $A \subseteq \mu$;
\item $(A,f) \restriction \lambda = (B,g) \restriction \mu$;
\item $(A,f)$ is injective on $\lambda$ and 
$(B,g)$ is injective on $\mu$.
\end{enumerate}

By Proposition 4.14, to show that $r \restriction \beta$ forces that 
$a_p(\beta)$ and $a_q(\beta)$ are compatible, it suffices to show that 
$r \restriction \beta$ forces that every node of $\dom(f)$ in $I_\lambda$ 
is incomparable in $\dot T_\beta$ 
with every node of $\dom(g)$ in $I_\mu$, 
and every node of $\ran(f)$ in $I_\lambda$ is 
incomparable in $\dot U_\beta$ with every node of $\ran(g)$ in $I_\mu$.

Consider $\theta \in \dom(f)$ in $I_\lambda$ 
and $\tau \in \dom(g)$ in $I_\mu$. 
Then $\theta$ and $\tau$ are both even. 
Fix $m < r_{\lambda,\beta,0}$ and $n < r_{\mu,\beta,1}$ such that 
$\theta = \theta(\lambda,\beta,m)$ and $\tau = \tau(\mu,\beta,n)$.

Since $\tau$ is even, by the definition of $g_\mu$ we have that 
$n \in g_\mu(\beta)$. 
As $g_\lambda = g_\mu$, $n \in g_\lambda(\beta)$. 
Therefore, $\tau(\lambda,\beta,n)$ is even. 
We also have that $\theta = \theta(\lambda,\beta,m)$ is even. 
By the definition of $f_\lambda$, 
$(\theta',\tau') := f_\lambda(\beta,m,n)$ is defined, 
$\theta'$ and $\tau'$ are distinct ordinals 
in $I_\gamma$ for some $\gamma < \lambda$, and 
$p \restriction \beta \Vdash_{\p_\beta} \theta' <_{\dot T_\beta} \theta$. 
But $f_\lambda = f_\mu$. 
So we also have that $f_\mu(\beta,m,n) = (\theta',\tau')$, and therefore 
$q \restriction \beta \Vdash_{\p_\beta} \tau' <_{\dot T_\beta} \tau$. 

Since $r \restriction \beta$ is below $p \restriction \beta$ and 
$q \restriction \beta$, it forces that 
$\theta' <_{\dot T_\beta} \theta$ and 
$\tau' <_{\dot T_\beta} \tau$. 
Since $\theta' \ne \tau'$ are both in $I_\gamma$, 
$r \restriction \beta$ forces that 
$p_{\dot T_\beta}(\theta,\gamma) = \theta' \ne 
p_{\dot T_\beta}(\tau,\gamma) = \tau'$. 
This implies that $r \restriction \beta$ forces that 
$\theta$ and $\tau$ are incomparable in $\dot T_\beta$, since otherwise they would 
have the same nodes below them on levels less than $\lambda$. 
By a symmetric argument, $r \restriction \beta$ also forces that 
if $\theta \in \ran(f)$ is in $I_\lambda$ 
and $\tau \in \ran(g)$ is in $I_\mu$, then $\theta$ and $\tau$ are incomparable 
in $\dot U_\beta$.
\end{proof}

\begin{proposition}
There exists a function $F : H(\kappa^{++})^{<\omega} \to H(\kappa^{++})$ 
such that for any $\alpha$-suitable $N$ which is closed under $F$, 
$p(N,\alpha)$ is $(N,\p_\alpha)$-generic.
\end{proposition}

\begin{proof}
The proof is by induction on $\alpha$. 
Recall that for all $0 < \alpha \le \kappa$, $\p_\alpha = \p_1$ is forcing equivalent to 
the Levy collapse $\coll(\omega_1, <\! \kappa)$, which is $\kappa$-c.c. 
And $p(N,\alpha)$ is the maximum condition in $\p_\alpha$  
and is $(N,\p_\alpha)$-generic for 
any $\alpha$-suitable $N$.

Now assume that $\kappa < \alpha < \kappa^+$ and the proposition holds for all 
ordinals below $\alpha$. 
For each $\beta < \alpha$, fix a function 
$F_\beta : H(\kappa^{++})^{<\omega} \to H(\kappa^{++})$ 
such that for any $\beta$-suitable $N$ which is closed under $F_\beta$, 
$p(N,\beta)$ is $(N,\p_\beta)$-generic. 
Define $F^* : \alpha \times H(\kappa^{++})^{<\omega} \to H(\kappa^{++})$ 
by $F^*(\beta,x) = F_\beta(x)$.

Let $\mathcal X$ denote the collection of all $\alpha$-suitable sequences. 
Fix a Skolem function $F$ for the structure 
$(H(\kappa^{++}), \in, \unlhd, \mathcal X, F^*)$. 
Let $N$ be $\alpha$-suitable and closed under $F$, and we will show 
that $p(N,\alpha)$ is $(N,\p_\alpha)$-generic. 
Let $\lambda := N \cap \kappa$. 
Note that by elementarity, 
for all $\beta \in N \cap \alpha$, $N$ is closed under $F_\beta$, 
and therefore $p(N,\beta)$ is $(N,\p_\beta)$-generic.

To prove that $p(N,\alpha)$ is $(N,\p_\alpha)$-generic, 
fix a dense open subset $D$ of $\p_\alpha$ which is a member of $N$. 
We will show that $D \cap N$ is predense below $p(N,\alpha)$. 
Let $r_0 \le_\alpha p(N,\alpha)$, and we will prove that there exists a member of 
$D \cap N$ which is compatible in $\p_\alpha$ with $r_0$. 
Since $D$ is dense open, we can fix $r \le_\alpha r_0$ in $D$. 
Moreover, by extending $r$ further if necessary using Lemma 6.19, 
we may assume without loss of generality that $\#_N^\alpha(r,r \restriction N)$ holds.

Suppose for a contradiction that there is no member of $D \cap N$ which 
is compatible with $r$. 
We claim that there exists a sequence 
$\langle (r_i,N_i) : i \in S \rangle$ in $N$ satisfying:
\begin{enumerate}
\item the sequence $\langle N_i : i \in S \rangle$ is $\alpha$-suitable;
\item $S \in J^+$;
\item for all $i \in S$, $N_i \cap \kappa = i$ if $i > 0$, $N_i$ is closed under $F^*$, 
$r_i \in D$, and $\#_{N_i}^\alpha(r_i,r_i \restriction N_i)$ holds;
\item for all $i < j$ in $S$, $r_i \in N_j$ and $r_i$ and $r_j$ are 
incompatible in $\p_\alpha$.
\end{enumerate}

The definition of such a sequence is by induction. 
For the base case, we let $0 \in S$, 
and 
pick an $\alpha$-suitable $N_0$ which is closed under $F^*$ 
such that $r \restriction N \in N_0$, 
and some $r_0$ such that $r_0 \in D$ and 
$\#_{N_0}^\alpha(r_0,r_0 \restriction N_0)$ holds. 
Such objects clearly exist by the elementarity of $N$.

Let $\beta < \kappa$ and assume that $S \cap \beta$ and  
$\langle (r_i,N_i) : i \in S \cap \beta \rangle$ are defined and satisfy 
properties (1), (3), and (4) above for $S \cap \beta$ in place of $S$. 
Let $\gamma \ge \beta$ be the least ordinal below $\kappa$, if it exists, 
for which there is a pair $(N,r)$ such that 
$\langle (r_i,N_i) : i \in S \cap \beta \rangle \cup 
\{ (\gamma, (N,r) \}$ still satisfies properties (1), (3), and (4) for 
$(S \cap \beta) \cup \{ \gamma \}$ in place of $S$. 
If such an ordinal $\gamma$ does not exist, then let $S = S \cap \beta$ 
and we are done. 
Otherwise, define $S \cap (\gamma+1) := (S \cap \beta) \cup \{ \gamma \}$, 
and let $(N_\gamma,r_\gamma)$ be the $\unlhd$-least pair such that 
$\langle (N_i,r_i) : i \in S \cap (\gamma+1) \rangle$ satisfies properties 
(1), (3), and (4).

This completes the definition of 
$\langle (r_i,N_i) : i \in S \rangle$. 
Note that by elementarity, this sequence is in $N$. 
It remains to show that $S \in J^+$. 

First let us note that $S$ is unbounded in $\kappa$. 
If not, then by elementarity $\sup(S) < \lambda$. 
But recall that $N$ is $\alpha$-suitable and $r$ is incompatible with 
every member of $D \cap N$, and in particular, with $r_i$ for all 
$i \in S \cap \lambda$. 
And for all $i \in S \cap \lambda$, $r_i$ and $N_i$ are in $N$. 
Also, $\#_N^\alpha(r,r \restriction N)$ holds. 
It easily follows that the sequence 
$\langle (r_i,N_i) : i \in S \rangle \cup \{ (\lambda, (r,N) ) \}$ 
satisfies properties (1), (3), and (4) above, which contradicts that 
$\sup(S) < \lambda$.

So indeed $S$ is cofinal in $\kappa$. 
It follows that the set $\lim(S)$ is club in $\kappa$, and hence is in $J^*$. 
Also, as we know, the set of inaccessibles below $\kappa$ is in $J^*$. 
We will show that every inaccessible limit point of $S$ is in $S$, 
which implies that in fact $S \in J^*$.

Suppose for a contradiction there is an inaccessible limit point of $S$ 
which is not in $S$. 
By elementarity, we can fix such an inaccessible $\mu$ in $N \cap \kappa$. 
Let $M := \bigcup \{ N_\gamma : \gamma \in S \cap \mu \}$. 
Then $M \cap \kappa = \mu$. 
Note that $M$ is closed under $F^*$. 
Therefore, if $M$ is $\alpha$-suitable, which we will check in a moment, it follows that 
for all $\beta \in M \cap \alpha$, 
$p(M,\beta)$ is $(M,\p_\beta)$-generic.

We claim that there exists a condition $r_\mu \in \p_\alpha$ such that the following 
statements are satisfied:
\begin{enumerate}
\item[(a)] $\langle N_i : i \in S \cap \mu \rangle \cup \{ (\mu,M) \}$ is $\alpha$-suitable;
\item[(b)] $r_\mu \in D$ and $\#_{M}^\alpha(r_\mu,r_\mu \restriction M)$;
\item[(c)] for all $i \in S \cap \mu$, $r_i \in M$ and $r_i$ and $r_\mu$ 
are incompatible in $\p_\alpha$.
\end{enumerate}
It follows from this claim and the definition of $S$ that in fact $\mu \in S$, 
which contradicts the choice of $\mu$.

By Notation 7.1 and the definition of $M$, 
statement (a) holds provided that $M$ is $\alpha$-suitable. 
By Definition 6.1 and Notation 6.3, 
that easily follows from the definition of $M$ together with the fact 
that $M \cap \kappa = \mu$ is inaccessible. 
For part of statement (c), we know that for all $i \in S \cap \mu$, 
$r_i \in N_{\min(S \setminus (i+1))} \subseteq M$, so $r_i \in M$.

We have that $M \in N$ are both $\alpha$-suitable and for all 
$\beta \in M \cap \alpha$, $p(M,\beta)$ is $(M,\p_\beta)$-generic. 
And also $\#_N^\alpha(r,r \restriction N)$ holds and $r \restriction N \in M$. 
By Lemma 6.24, there is $r_\mu \le_\alpha r$ such that 
$\#_M^\alpha(r_\mu,r_\mu \restriction M)$ holds. 
As $D$ is dense open and $r \in D$, $r_\mu \in D$. 
And for all $i \in S \cap \mu$, $r_i$ and $r$ are incompatible, which implies 
that $r_i$ and $r_\mu$ are incompatible. 
This completes the proof of the claim, and we have a contradiction.

To summarize, 
we have an $\alpha$-suitable 
sequence $\langle N_i : i \in S \rangle$, where $S \in J^+$. 
And for each $i \in S$, we have a condition $r_i$ such that 
$\#_{N_i}^\alpha(r_i,r_i \restriction N_i)$ holds, 
and for all $i < j$ in $S$, $r_i$ and $r_j$ are incompatible in $\p_\alpha$. 
Thus, the assumptions of Proposition 7.5 hold for the $\alpha$-suitable 
sequence $\langle N_i : i \in S \rangle$ and the associated conditions 
$r_i$, $r_i$, and $r_i \restriction N_i$. 
It easily follows from the conclusion of Proposition 7.5 that 
there are $i < j$ in $S$ such that $r_i$ and $r_j$ are compatible, 
which is a contradiction.
\end{proof}

\begin{corollary}
The forcing poset $\p_\alpha$ is $\kappa$-proper on a stationary set. 
In particular, $\p_\alpha$ preserves $\kappa$.
\end{corollary}

\begin{proof}
Let $\theta > \kappa^+$ be a regular cardinal. 
Let $F$ be the function described in Proposition 7.6. 
Since $\mathcal Y$ is stationary in $P_{\kappa}(H(\kappa^{++}))$ 
by Lemma 6.2, there are stationarily many $M \in P_{\kappa}(H(\theta))$ 
such that $M \cap H(\kappa^{++})$ is $\alpha$-suitable 
and closed under $F$.

Consider such a set $M$. 
We will prove that every member of $M \cap \p_\alpha$ has an extension 
which is $(M,\p_\alpha)$-generic. 
Let $N := M \cap H(\kappa^{++})$. 
By the choice of $F$, $p(N,\alpha)$ is $(N,\p_\alpha)$-generic. 
As $\p_\alpha$ is $\kappa^+$-c.c.\! and is a member of $H(\kappa^{++})$, 
every maximal antichain of $\p_\alpha$ is a member of $H(\kappa^{++})$. 
Hence, every maximal antichain of $\p_\alpha$ which is a member of 
$M$ is also a member of $N$. 
It easily follows that $p(N,\alpha)$ is $(M,\p_\alpha)$-generic. 
Now if $p \in M \cap \p_\alpha$, then $p \in N \cap \p_\alpha$. 
Hence by Lemma 6.7, we can fix $r \le_\alpha p, p(N,\alpha)$. 
Since $r \le_\alpha p(N,\alpha)$ and $p(N,\alpha)$ is $(M,\p_\alpha)$-generic, 
$r$ is also $(M,\p_\alpha)$-generic, and we are done.
\end{proof}

This completes the proof that $\p_\alpha$ preserves $\kappa$. 
It remains to show that Inductive Hypotheses 7.2 and 7.3 hold for $\alpha$.

\begin{proposition}
Suppose that 
$\langle N_\lambda : \lambda \in S \rangle$ is $\alpha$-suitable, 
with union $N$, where $S \in J^+$. 
Let $T$ be the set of $\lambda \in S$ such that $p(N_\lambda,\alpha)$ 
is strongly $(N_\lambda,\p_\alpha)$-generic. 
Then $S \setminus T \in J$.
\end{proposition}

\begin{proof}
Suppose for a contradiction that $S \setminus T \in J^+$. 
For each $\lambda \in S \setminus T$, 
the fact that $p(N_\lambda,\alpha)$ is not strongly 
$(N_\lambda,\p_\alpha)$-generic means that there exists a set $D_\lambda$ 
which is a dense subset of $N_\lambda \cap \p_\alpha$ 
such that $D_\lambda$ is not predense below $p(N_\lambda,\alpha)$. 
By extending members of $D_\lambda$ to determined conditions, 
which is possible since the determined conditions are dense in 
$N_\lambda \cap \p_\alpha$, we may assume without loss of 
generality that $D_\lambda$ consists of determined conditions.

As $S \setminus T \in J^+$, by Lemma 2.3 there 
exists a set $D \subseteq N$ and a stationary set $U \subseteq S \setminus T$ 
such that for all $\lambda \in U$, $D \cap N_\lambda = D_\lambda$. 
Note that $D$ is a dense subset of $\p_\alpha$. 
Namely, if $p \in \p_\alpha$, then for any large enough 
$\lambda \in U$, $p \in N_\lambda \cap \p_\alpha$. 
But then $D \cap N_\lambda = D_\lambda$, which is a dense subset of $N_\lambda \cap \p_\alpha$. 
Thus, we can find $q \le_\alpha p$ in $D_\lambda = D \cap N_\lambda$. 
Also note that $D$ consists of determined conditions.

Let $F$ be a function as described in Proposition 7.6. 
For each $\lambda \in U$, let 
$M_\lambda$ be the unique smallest elementary substructure of  
$$
(H(\kappa^{++}),\in,\unlhd,J,\mathcal Y,\langle \p_\beta : \beta \le \alpha \rangle,F)
$$
such that $N_\lambda \cup \{ D \} \subseteq M_\lambda$ and 
$M_\lambda^{<\lambda} \subseteq M_\lambda$. 
Standard arguments show that $M_\lambda$ exists. 
The sequence $\langle M_\lambda : \lambda \in U \rangle$ is $\subseteq$-increasing 
and continuous at limit points of $U$ in $U$. 
So we can fix a club $C \subseteq \kappa$ such that for all $\lambda \in C \cap U$, 
$M_\lambda \cap \kappa = \lambda$.

We claim that the sequence 
$\langle M_\lambda : \lambda \in C \cap U \rangle$
is $\alpha$-suitable. 
Observe that for all $\lambda \in C \cap U$, since $\lambda$ is inaccessible we have that 
$|M_\lambda| = |N_\lambda| = \lambda = M_\lambda \cap \kappa$. 
So $M_\lambda^{< (M_\lambda \cap \kappa)} \subseteq M_\lambda$. 
Hence, $M_\lambda$ is $\alpha$-suitable. 
By the closure of $M_\lambda$, if $\mu < \lambda$ is in $C \cap U$ then 
$M_\mu \in M_\lambda$. 
So the sequence $\langle M_\lambda : \lambda \in C \cap U \rangle$ is $\in$-increasing 
and continuous at limit points of $U$ in $U$, and thus is $\alpha$-suitable.

By the choice of $F$, for all $\lambda \in C \cap U$, 
$p(M_\lambda,\alpha)$ is $(M_\lambda,\p_\alpha)$-generic. 
Since $M_\lambda \cap \kappa = N_\lambda \cap \kappa$, Lemma 6.11 implies that 
$p(M_\lambda,\alpha) = p(N_\lambda,\alpha)$.

Since $p(M_\lambda,\alpha)$ is $(M_\lambda,\p_\alpha)$-generic 
and $D \in M_\lambda$ is a dense subset of $\p_\alpha$, 
$D \cap M_\lambda$ is predense below $p(M_\lambda,\alpha)$. 
So $D \cap M_\lambda$ is predense below $p(N_\lambda,\alpha)$. 
Since there are only $\kappa$ many determined conditions in 
$\p_\alpha$ and $M_\lambda \cap \kappa = N_\lambda \cap \kappa$, 
it follows that $M_\lambda$ and $N_\lambda$ contain exactly the same 
determined conditions of $\p_\alpha$. 
Since $D$ consists of determined conditions, we get that 
$D \cap M_\lambda = D \cap N_\lambda = D_\lambda$. 
So $D_\lambda$ is predense below $p(N_\lambda,\alpha)$, which contradicts the 
choice of $D_\lambda$.
\end{proof}

\begin{proposition}
Let $\langle N_\lambda : \lambda \in S \rangle$ be $\alpha$-suitable, 
with union $N$, where $S \in J^+$. 
Let $T$ be the set of $\lambda \in S$ which satisfy that 
whenever $\#_{N_\lambda}^\alpha(p,q,s)$ holds, then there is $s' \le_\alpha s$ 
such that $*_{N_\lambda}^\alpha(p,q,s')$ holds. 
Then $S \setminus T \in J$.
\end{proposition}

\begin{proof}
Suppose for a contradiction that $S \setminus T \in J^+$. 
For each $\lambda \in S \setminus T$, we can fix 
$p_\lambda$, $q_\lambda$, and $s_\lambda$ such that 
$\#_{N_\lambda}^\alpha(p_\lambda,q_\lambda,s_\lambda)$ holds, 
but for all $s' \le_\alpha s_\lambda$ in $N_\lambda \cap \p_\alpha$, 
$*_{N_\lambda}^\alpha(p_\lambda,q_\lambda,s')$ does not hold.

By Proposition 7.5, fix $U \subseteq S \setminus T$ in $J^+$ 
such that 
for all $\lambda < \mu$ in $U$, 
$p_\lambda$ and $q_\mu$ are compatible in $\p_\alpha$. 
By Proposition 7.8, fix $U_1 \subseteq U$ in $J^+$ such that 
for all $\lambda \in U_1$, $p(N_\lambda,\alpha)$ is strongly 
$(N_\lambda,\p_\alpha)$-generic.

Consider $\lambda \in U_1$. 
Define $H_\lambda$ as the set of $u \in N_\lambda \cap \p_\alpha$ such that 
either $u$ and $s_\lambda$ are incompatible in $\p_\alpha$, or 
else $u \le_\alpha s_\lambda$ and $u$ is incompatible in $\p_\alpha$ with either 
$p_\lambda$ or $q_\lambda$. 

We claim that $H_\lambda$ is dense in $N_\lambda \cap \p_\alpha$. 
So let $p \in N_\lambda \cap \p_\alpha$. 
If $p$ and $s_\lambda$ are incompatible, then $p \in H_\lambda$ and we are done. 
Otherwise, by the elementarity of $N_\lambda$ we can 
fix a determined condition $t \le_\alpha p, s_\lambda$ in $N_\lambda \cap \p_\alpha$. 
Now by assumption, for all $s' \le_\alpha s_\lambda$ in $N_\lambda \cap \p_\alpha$, 
$*_{N_\lambda}^\alpha(p_\lambda,q_\lambda,s')$ is false. 
Since $t \le_\alpha s_\lambda$, in particular, 
$*_{N_\lambda}^\alpha(p_\lambda,q_\lambda,t)$ is false. 
By the definition of $*_{N_\lambda}^\alpha$, there is $u \le_\alpha t$ in 
$N_\lambda \cap \p_\alpha$ such that $u$ is incompatible in $\p_\alpha$ with either 
$p_\lambda$ or $q_\lambda$. 
So $u \le_\alpha p$ and $u \in H_\lambda$.

Since $H_\lambda$ is dense in $N_\lambda \cap \p_\alpha$, we can  
fix a maximal antichain $A_\lambda$ 
of the forcing poset $N_\lambda \cap \p_\alpha$ which is a 
subset of $H_\lambda$. 
Define a function $F_\lambda : A_\lambda \to 3$ as follows. 
Given $u \in A_\lambda$, if $u$ is incompatible with 
$s_\lambda$, then let $F_\lambda(u) := 0$. 
Otherwise, $u \le_\alpha s_\lambda$ and $u$ is incompatible with 
either $p_\lambda$ or $q_\lambda$. 
If $u$ is incompatible with $p_\lambda$, then let 
$F_\lambda(u) := 1$. 
Otherwise, $u$ is compatible with $p_\lambda$ but incompatible 
with $q_\lambda$, and we let $F_\lambda(u) := 2$.

This defines for each $\lambda \in U_1$ a set 
$F_\lambda \subseteq N_\lambda$. 
Since $U_1 \in J^+$, by Lemma 2.3 
we can find a stationary set 
$W \subseteq U_1$ and a set $F \subseteq N$ 
such that for all $\lambda \in W$, $F \cap N_\lambda = F_\lambda$. 
It is easy to check that $F$ is a function, $\dom(F)$ is an antichain 
of $\p_\alpha$, 
and for all $\lambda \in W$, 
$\dom(F) \cap N_\lambda = A_\lambda$.

Now we are ready to derive a contradiction. 
Fix $\lambda < \mu$ in $W$. 
Then $p_\lambda$ and $q_\mu$ are compatible, so 
fix $w \in \p_\alpha$ which is below both of them. 
Then in particular, $w \le_\alpha p(N_\lambda,\alpha), p(N_\mu,\alpha)$. 
Since $A_\mu$ is a maximal antichain of $N_\mu \cap \p_\alpha$ 
and $p(N_\mu,\alpha)$ is strongly $(N_\mu,\p_\alpha)$-generic, 
$A_\mu$ is predense below $p(N_\mu,\alpha)$. 
As $w \le_\alpha p(N_\mu,\alpha)$, we can find $t_\mu \in A_\mu$ and 
$w_0$ such that $w_0 \le_\alpha w, t_\mu$. 
Then $w_0 \le_\alpha w \le_\alpha p(N_\lambda,\alpha)$. 
Since $A_\lambda$ is a maximal antichain of $N_\lambda \cap \p_\alpha$ 
and $p(N_\lambda,\alpha)$ is strongly $(N_\lambda,\p_\alpha)$-generic, 
$A_\lambda$ is predense below $p(N_\lambda,\alpha)$. 
So we can find $t_\lambda \in A_\lambda$ and $w_1$ such that 
$w_1 \le_\alpha w_0, t_\lambda$.

Since $A_\mu = \dom(F) \cap N_\mu$ and 
$A_\lambda = \dom(F) \cap N_\lambda$, 
$t_\mu$ and $t_\lambda$ are in $\dom(F)$. 
As $\dom(F)$ is an antichain and 
$w_1 \le_\alpha t_\lambda, t_\mu$, it follows that 
$t_\lambda = t_\mu$. 
Let $t^* := t_\lambda = t_\mu$.

Now $F \restriction A_\lambda = F_\lambda$ and 
$F \restriction A_\mu = F_\mu$. 
Therefore, $F(t^*) = F_\lambda(t_\lambda) = F_\mu(t_\mu)$. 
As $t^* = t_\lambda \in A_\lambda \subseteq H_\lambda$, 
$t^*$ is either incompatible with or below $s_\lambda$. 
But $w_1 \le_\alpha t^*$ and $w_1 \le_\alpha w_0 \le_\alpha w \le_\alpha 
p_\lambda \le_\alpha s_\lambda$. 
So $t_\lambda = t^* \le_\alpha s_\lambda = s_\mu$. 
By the definition of $F_\lambda$, 
we have that $F(t^*) = F_\lambda(t_\lambda) \ne 0$. 
So either $F(t^*) = 1$ or $F(t^*) = 2$.

Assume that $F(t^*) = 1$. 
Then $F_\lambda(t_\lambda) = 1$. 
By the definition of $F_\lambda$, 
$t_\lambda$ is incompatible with $p_\lambda$. 
But that is false, since $w_1 \le_\alpha t_\lambda$ and 
$w_1 \le_\alpha w_0 \le_\alpha w \le_\alpha p_\lambda$. 
So $F(t^*) = 2$. 
Hence, $F_\mu(t_\mu) = 2$. 
By the definition of $F_\mu$, it follows that 
$t^* = t_\mu$ is incompatible with $q_\mu$. 
But that is also false, since $w_0 \le_\alpha t_\mu$ and 
$w_0 \le_\alpha w \le_\alpha q_\mu$. 
Thus, we have reached a contradiction.
\end{proof}

\bibliographystyle{plain}
\bibliography{paper31}

\end{document}